\documentclass{au}

\usepackage[english]{babel}
\usepackage{amssymb}
\usepackage{amsmath}
\usepackage{array}
\usepackage{enumitem}
\usepackage[bookmarks]{hyperref}

\newtheorem{thm}{Theorem}[section]
\newtheorem{lem}[thm]{Lemma}
\newtheorem{sublem}{Claim}[thm]

\newtheorem{cor}[thm]{Corollary}

\newtheorem{conj}[thm]{Conjecture}
\newtheorem*{trivialpatternLemma}{Lemma~\ref{trivialpatternLem}}
\newtheorem*{AlmostTrivialPatternLemma}{Lemma~\ref{AlmostTrivialPattern}}
\newtheorem*{NUtheorem}{Theorem~\ref{NUtheoremLabel}}
\newtheorem*{theoremaboutsemilattice}{Theorem~\ref{aboutsemilattice}}
\newtheorem*{THMmainTheoremForFullPattern}{Theorem~\ref{mainTheoremForFullPattern}}

\DeclareMathOperator{\Pol}{Pol}

\DeclareMathOperator{\Key}{Key}

\DeclareMathOperator{\proj}{pr}
\DeclareMathOperator{\ar}{ar}

\let \kappa = \varkappa

\begin{document}

%
%

\title{Key (critical) relations preserved by a weak near-unanimity function}

\author[D. N. Zhuk]{Dmitriy N. Zhuk}
\email{zhuk@intsys.msu.ru}
\address{Moscow State University\\
119899 Moscow\\ Russia}
\thanks{The research of the author is supported by grant RFFI 13-01-00684-a.}


\begin{abstract}
In the paper we introduce a notion of a key relation, which is similar to the notion of a critical relation
introduced by Keith A.\,Kearnes and \'Agnes Szendrei.
All clones on finite sets can be defined by only key relations.
In addition there is a nice description of all key relations on 2 elements.
These are exactly the relations that can be defined as a disjunction of linear equations.
In the paper we show that, in general key relations do not have such a nice description.
Nevertheless, we obtain a nice characterization of
all key relations preserved by a weak near-unanimity function.
This characterization is presented in the paper.

\end{abstract}

\maketitle

\section{Introduction}

The main result in clone theory is apparently the description of all clones on 2 elements obtained by E.Post in \cite{post1,post2}.
Nevertheless, it seems unrealistic to describe all clones on bigger sets.
For example, we know that we have continuum of them. 
Also, we have a lot of results that prove that the lattice of all clones is not only uncountable, but very complicated.

It turned out that uncountability is not crucial,
for example in \cite{mvlsc} the lattice of all clones of self-dual operations on 3 elements was described,
even though this lattice has continuum cardinality.
The main idea of that paper and many other papers in clone theory is an accurate work with relations.
The fact that we have known all maximal clones for 45 years \cite{rosmax}
and still don't have any description of all minimal clones
just proves that working with relations is much easier than with operations.

We have $2^{|A|^{n}}$ relations of arity $n$ on a set $A$, which is a huge number even for $|A|=4$ and $n=3$.
But if we check most of the significant papers in clone theory
we will see that all the relations arising there have a nice characterization: they are symmetric
or have some regular structure.
In this paper we will try to provide a mathematical background to this observation.

First, it is easy to notice that
we don't need relations that can be represented as a conjunction of relations with smaller arities
\cite{MinimalClones,mvlsc}. Relations that cannot be represented in this way are called \emph{essential}.
Second, observe that
if a relation is an intersection of other relations from the relational clone
then we don't need this relation to define this relational clone.
Relations that cannot be represented in this way are called maximal in~\cite{mvlsc} and
critical in~\cite{agnes}.

It turned out that all critical relations $\rho\subseteq A^{h}$ have the following property:
there exists a tuple $\beta\in A^h\setminus \rho$ such that for every
$\alpha\in A^h\setminus \rho$ there exists
a unary vector-function $\Psi = (\psi_1,\ldots,\psi_{h})$ which
preserves $\rho$ and gives $\Psi(\alpha) = \beta$.
This means that every tuple which is not from $\rho$ can be mapped to $\beta$ by a vector-function
preserving $\rho$.
A relation satisfying this property is called a \emph{key relation}, and
a tuple $\beta$ is called a \emph{key tuple} for this relation.

This property seems to be profitable because it is a combinatorial property of
a relation which doesn't involve any difficult objects (no clones, no relational clones, no primitive positive formulas).
Another motivation to study key relations is
a nice description of all key relations on 2 elements.
These are exactly the relations that can be defined as a disjunction of linear equations.

As we show in the paper, key relations on bigger sets can be complicated.
But it turned out that we can get a very similar characterization of
key relations if they are preserved by a weak near-unanimity function (WNU).
In this case we show that all the variables of the relation can be divided into two groups,
and the relation can be divided into two parts.
The first part is very similar to the relation $\{a,b\}^{n}\setminus \{a\}^{n}$,
and the second part can be defined by a linear equation in some abelian group.

The consideration of key relations preserved by a WNU seems to be justified
because of the following reason.
First, let us consider an algebra with all the operations from a clone.
We know that if we have an idempotent algebra $\mathbb{A}$ without a weak near-unanimity term,
then we can find a factor of $\mathbb{A}$ whose operations are essentially unary,
where a factor is a homomorphic image of a subalgebra of $\mathbb{A}$ \cite{BulatovAboutCSP,miklos}.
This means that if a relational clone is not preserved by a WNU,
then we can find relations in it which are as complicated as in general, i.e. in a relational clone of all relations on
a finite set.
To show this we need to consider the idempotent reduction of the corresponding clone, and then the corresponding factor.
Thus, if we cannot describe all key relations,
then we need to consider relational clones preserved by a WNU.

Second, the importance of a WNU was discovered while studying the constraint satisfaction problem.
The standard way to parameterize interesting subclasses of the constraint satisfaction
problem is via finite relational structures \cite{FederVardi,jeavons}. The main
problem is to classify those subclasses that are tractable (solvable in polynomial
time) and those that are NP-complete.
It was conjectured that if a core of a relational structure has a WNU polymorphism
then the corresponding constraint satisfaction problem is tractable,
otherwise it is NP-complete \cite{CSPconjecture,BulatovAboutCSP}.
We believe that this characterization can be helpful in proving this conjecture.

The paper is organized as follows.
In Sections 2 and 3 we give necessary definitions and formulate the main results of the paper.
That is, a description of all key relations on 2 elements (with a proof),
a characterization of all key relations preserved by a WNU on bigger sets.
We assign an equivalence relation on the set of variables to every key relation preserved by a WNU,
and present stronger versions of this characterization if the equivalence relation is a full equivalence relation,
trivial equivalence relation, or almost trivial equivalence relation.
This equivalence relation is called the pattern of a relation.
As a result we obtain a complete description of all
key relations preserved by a near-unanimity function,
a semilattice operation, or a 2-semilattice operation.

In Section 4 we give the remaining definitions and notations we will need in the paper.
In the next section we prove several auxiliary statements which are used later.

In Section 6 we formulate and prove one of the main statement of the paper.
Precisely, we show that if a relation of arity $n$ contains exactly $|A|^{n-1}$ tuples,
projection onto any $(n-1)$ coordinates is a full relation,
and the relation is preserved by a WNU, then this relation can be defined by a linear equation.

In Section 7 we introduce a notion of a core of a key relation and prove different properties of a core.
For example, we prove that a core with full pattern can be divided into isomorphic
key blocks and each of these key blocks can be defined by a linear equation.

In Section 8 we prove the main results of the paper.
That is, a characterization of a key relation with arbitrary pattern,
and a complete description of key relations with trivial pattern and almost trivial pattern.

The last section is devoted to key relations with full pattern.
First, we prove that a core with full pattern can be divided into blocks,
then we generalize this result for a key relation with full pattern.

I want to thank my colleagues and friends from the Department of Algebra in Charles University in Prague
for the very fruitful discussions, especially Libor Barto,  Jakub Opr\v{s}al, Jakub Bulin, and Alexandr Kazda.
I am grateful to my colleagues from the Chair of Mathematical Theory of Intelligent Systems
in Moscow State University, especially my supervisor Valeriy Kudryavtsev, Alexey Galatenko and Grigoriy Bokov.
Also I want to thank Stanislav Moiseev who found the first ugly example of a key relation with a computer.
I would like to give special thanks to a very kind mathematician Hajime Machida who always supported me and my research.  

\section{Key relations}

In this section we give necessary definitions, particularly the definition of a key relation.
Then, we prove the description of all key relations on 2 elements, which is a very simple result.
Finally, we give a definition of the pattern of a key relation and formulate the main properties of the pattern.

\subsection{Main definitions}

Let $A$ be a finite set, and let $O_A^{n}:=A^{A^n}$ be the
set of all $n$-ary functions on $A$, $O_A:=\bigcup_{n \geq 1} O_A^{n}$.

For a tuple $\alpha\in A^{n}$ by $\alpha(i)$ we denote the $i$-th element of $\alpha$.
By $R_A^{n}$ we denote the set of all $n$-ary relations on the set $A$.
$R_A = \bigcup_{n=0}^{\infty} R_A^n$.
If it is not specified we always assume that a relation is defined on the set $A$.
We do not distinguish between predicates and relations, and consider positive primitive formulas over sets of relations.
For a set of relations $D$ by $[D]$ we denote the closure of $D$ over positive primitive formulas.
Closed sets of relations containing equality and empty relations are called \emph{relational clones}.
For 
$C \subseteq R_A$, we define
$\Pol (C) := \{f \in O_A \mid \forall \sigma \in C: f \text{ preserves } \sigma\}.$

A function $f$ is called \emph{idempotent} if $f(x,x,\ldots,x) = x$.
A \emph{weak near-unanimity function} (WNU) is an idempotent function $f$ satisfying the following property
\[f(x,y,y,\ldots,y) = f(y,x,y,\ldots,y) = \dots = f(y,y,\ldots,y,x).\]

A relation $\rho\in R_{A}^{h}$ is called \emph{essential} if it cannot be represented as a conjunction of relations with smaller arities.
A tuple
$\left(\begin{smallmatrix}a_1\\\vdots\\a_h\end{smallmatrix}\right)\in A^{h}\setminus \rho$ is
called \emph{essential for} $\rho$
if for every $i\in \{1,2,\ldots,h\}$ there exists
$b$ such that
$\left(\begin{smallmatrix}a_1\\\vdots\\a_{i-1}\\b\\a_{i+1}\\\vdots\\a_h\end{smallmatrix}\right)\in\rho$.
For a relation $\rho$ by $\widetilde \rho$ we denote $\rho$ filled up with all essential tuples.

The following lemma can be easily checked.
We omit the proof and refer readers to \cite{dm_post, MinimalClones,mvlsc}.
\begin{lem}\label{sushnabor}
Suppose $\rho \in R_{A}^{n},$ where $n\ge 1$. Then $\rho$ is essential
if and only if there exists an essential tuple for $\rho$.
\end{lem}

A tuple $\Psi =(\psi_1,\psi_2,\ldots,\psi_h)$, where
$\psi_i:A\to A$, is called a \emph{unary vector-function}.
We say that $\Psi$ \emph{preserves} a relation $\rho$ of arity $h$ if
$\Psi\left(\begin{smallmatrix}
a_1\\
a_2\\
\vdots\\
a_h
\end{smallmatrix}\right):=
\left(\begin{smallmatrix}
\psi_1(a_1)\\
\psi_2(a_2)\\
\vdots\\
\psi_h(a_h)
\end{smallmatrix}\right)\in \rho$
for every $\left(\begin{smallmatrix}
a_1\\
a_2\\
\vdots\\
a_h
\end{smallmatrix}\right)\in \rho$.
We say that a relation $\rho$ of arity $h$ is \emph{a key relation} if
there exists a tuple $\beta\in A^h\setminus \rho$ such that for every
$\alpha\in A^h\setminus \rho$ there exists
a vector-function $\Psi$ which
preserves $\rho$ and gives $\Psi(\alpha) = \beta$.
A tuple $\beta$ is called a \emph{key tuple} for $\rho$.

We can check the following facts about key relations.

\begin{enumerate}

\item Suppose $\rho$ is a key relation. Then $\rho$ is essential if and only if $\rho$ has no dummy variables
(Lemma~\ref{DammuVariables}).

\item Suppose $\rho$ = $\sigma\times A^{s}$. Then $\rho$ is a key relation if and only if $\sigma$ is a key relation
(Lemma~\ref{KeyIsAlwaysKey}).

\item Suppose $\sigma(x_2,\ldots,x_n) = \rho(b_1,x_2,\ldots,x_n)$,
$(b_1,\ldots,b_n)$ is a key tuple for $\rho$.
Then $\sigma$ is a key relation and
$(b_2,\ldots,b_n)$ is a key tuple for $\sigma$ (Lemma~\ref{ReduceArityOfKey}).

\item
Suppose $\alpha$ is a key tuple for $\rho$, and a unary vector-function $\Psi$ preserves
$\rho$. Then either $\Psi(\alpha)\in \rho$, or $\Psi(\alpha)$ is a key tuple for $\rho$ (Lemma~\ref{KeyGotoRho}).

\end{enumerate}

A relation $\rho$ is called \emph{maximal in a relational clone $C$} if
there exists an essential tuple $\alpha$ for $\rho$ such that
$\rho$ is a maximal relation in $C$ with the property $\alpha\notin \rho$.
A relation is called \emph{critical in a relational clone $C$} if it is completely $\cap$-irreducible in $C$ and directly
indecomposable. 
\begin{lem}\cite[Lemma 2.1]{agnes} \label{agnesLemma}%
A relation $\rho$ is critical in a relational clone $C$ if and only if
it is maximal in a relational clone $C$.
\end{lem}
It follows from the definition that every relation in a relational clone
can be defined as a conjunction of critical relations from the relational clone,
thus we need only critical relations to generate any relational clone.

For a relational clone $C$ and a relation $\rho$ of arity $n$
by $\langle\rho\rangle_{C}$ we denote the minimal relation of arity $n$ in $C$ containing $\rho$.
It follows from the Galois connection between clones and relational clones that we have the following lemma \cite[Section 2.2]{lau}.
\begin{lem}\label{reformulation}
Suppose $\rho\in R_{A}^{n}$, $\rho = \{\alpha_1,\ldots,\alpha_s\}$, $C$ is a relational clone. Then
$\langle\rho\rangle_{C}=\{f(\alpha_1,\ldots,\alpha_s)\mid f\in \Pol(C)\cap O_A^s\}.$

\end{lem}

\begin{lem}\label{CriticalIsKey}
Suppose $\rho$ is a critical (maximal) relation in a relational clone $C$.
Then $\rho$ is a key relation.
\end{lem}

\begin{proof}
By Lemma~\ref{agnesLemma}
there exists an essential tuple $\beta$ for $\rho$ such that
$\rho$ is a maximal relation in $C$ such that $\beta\notin\rho$.
We want to show that $\beta$ is a key tuple for $\rho$.
Let $n$ be the arity of $\rho$.
For a tuple $\alpha\in A^{n}\setminus\rho$ we consider the
relation $\langle\rho\cup\{\alpha\}\rangle_{C}$.
Since $\rho$ is maximal, we have
$\beta\in \langle\rho\cup\{\alpha\}\rangle_{C}$.
Let $\rho = \{\gamma_1,\gamma_2,\ldots,\gamma_n\}$. It follows from Lemma~\ref{reformulation} that
there exists a function
$f$ preserving $\rho$ such that
$f(\gamma_1,\ldots,\gamma_n,\alpha) = \beta$.
Let
$\Psi\left(\begin{smallmatrix}
x_1\\
\vdots\\
x_h
\end{smallmatrix}\right)=
f\left(\gamma_1,\ldots,\gamma_n,\begin{smallmatrix}x_1\\\vdots\\x_h\end{smallmatrix}\right)$.
It is easy to see that
$\Psi$ preserves $\rho$ and
$\Psi(\alpha) = \beta$. This completes the proof.
\end{proof}

The next theorem follows from the definition of a critical relation and Lemma~\ref{CriticalIsKey}.
\begin{thm}
Suppose $KR_A$ is the set of all essential key relations from $R_A$.
Then
$[C\cap KR_A] = C$ for every relational clone $C$.
\end{thm}

This means that every relational clone can be determined by only key relations from this relational clone.

\subsection{Key relations on two elements}\label{twovaluedcaseSection}

Let $A=\{0,1\}$.
An equation \[a_{1}x_{1}+\ldots+a_{s}x_{s} = a_{0}\]
is called a \emph{linear equation} (``$+$'' is addition modulo 2).

\begin{thm}\label{OnlyLinearOn2Elements}

Suppose $\rho\in R_A^n$ , $A=\{0,1\}$.
Then $\rho$ is a key relation if and only if
$\rho(x_1,\ldots,x_n)=L_{1}\vee L_{2}\vee \ldots \vee L_{m}$
for some linear equations $L_{1},L_{2},\ldots,L_{m}$.

\end{thm}

\begin{proof}
By Lemmas~\ref{DammuVariables} and \ref{KeyIsAlwaysKey}, without loss of generality we can assume that
$\rho$ is essential.

Suppose $\rho$ is a key relation, and $\beta$ is a key tuple for $\rho$. Let us prove by induction on the arity of $\rho$ that
$\rho$ can be represented as a disjunction of linear equations.
This is obvious if the arity of $\rho$ is less than 2.

Let $\rho' (x_2,\ldots,x_n) = \rho(\beta(1),x_2,\ldots, x_n)$.
By Lemma~\ref{ReduceArityOfKey}, $\rho'$ is a key relation and $(\beta(2),\ldots,\beta(n))$ is a key tuple for $\rho'$.
By the inductive assumption,
$\rho'(x_2,\ldots,x_n) = L_1' \vee L_2' \vee \dots \vee L_s'$, where $L_i'$ is a linear equation for every $i$.

We consider two cases. First, assume that
for every $\alpha\in A^n$ with $\alpha(1)\neq \beta(1)$ we have $\alpha \in \rho$.
Then the following equation proves the statement in this case
\[\rho(x_1,\ldots,x_n) = L_1' \vee L_2' \vee \dots \vee L_s' \vee (x_1 = \beta(1) +1).\]

Second, assume that there exists $\alpha\in A^{n}\setminus\rho$ such that
$\alpha(1)\neq \beta(1)$. Since $\beta$ is a key tuple,
there exists a vector-function $\Psi=(\psi_1,\ldots,\psi_n)$ preserving $\rho$ such that
$\Psi(\alpha) = \beta$.
Assume that $\psi_i$ is a constant for some $i$.
Since $\rho$ is essential, we can find a tuple $\beta'\in\rho$ that can be obtained from $\beta$
by changing the $i$-th component.
We can check that $\Psi(\Psi(\beta')) = \Psi(\Psi(\beta)) = \beta$, which contradicts the fact that
$\Psi$ preserves $\rho$.

Thus we know that $\psi_i$ is not a constant for every $i$,
then $\psi_i(x) = x+a_i$, where $a_i\in \{0,1\}$.
Suppose that for every $i\in\{1,2,\ldots,s\}$
\[L_i' = (b_{i,2}x_2+b_{i,3}x_3+ \ldots +b_{i,n}x_n = b_{i,0}).\]
For $i\in \{1,2,\ldots,s\}$ by $L_i$ we denote the following linear equation
\[(b_{i,2} a_2 + \ldots + b_{i,n} a_n)(x_1+\beta(1))+
b_{i,2} x_2 + \ldots +b_{i,n} x_n = b_{i,0}.\]
Put
$\sigma(x_1,\ldots,x_n) = L_1\vee L_2\vee \ldots \vee L_s$.
Let us prove that $\sigma = \rho$.
Let $\gamma \in A^{n}$. If $\gamma(1) = \beta(1)$ then it is obvious that
$\gamma\in \sigma\Leftrightarrow \gamma\in \rho.$

Suppose $\gamma(1) \neq \beta(1)$. It is easy to check that $\Psi$ preserves $\sigma$.
Then, since $\Psi$ is a bijection,
we have
\[\gamma\in \sigma\Leftrightarrow \Psi(\gamma)\in \sigma\Leftrightarrow \Psi(\gamma)\in \rho\Leftrightarrow \gamma\in \rho.\]
This completes the second case.

It remains to show that a disjunction of linear equations always defines a key relation.
Let $\rho(x_1,\ldots,x_n) = L_1\vee L_2\vee \ldots \vee L_s$,
where $L_1,\ldots,L_s$ are linear equations.
Let us show that every tuple $\beta \in A^n\setminus \rho$ is a key tuple.
For every $\alpha \in A^n\setminus \rho$ we have to find a unary vector-function $\Psi$ such that
$\Psi(\alpha) = \beta$.
Let $\Psi$ be the bijective vector-function with the above property.
It remains to show that $\Psi$ preserves $\rho$.
It is easy to see that a bijective vector-function either preserves a linear equation, or
maps all solutions of the equation to nonsolutions. 
Since $\alpha,\beta\in A^n\setminus \rho$, all equations $L_1,L_2,\ldots,L_s$ are incorrect for $\alpha$ and $\beta$.
Therefore $\Psi$ preserves all the above equations and preserves $\rho$.
\end{proof}

Theorem~\ref{OnlyLinearOn2Elements} shows that all key relations on two elements have a regular structure.
This fact allows to get a nice proof of Post's Lattice Theorem.

Let us consider another example of using this idea.
By $O_{A,s}^{n}$
we denote the set of all tuples $(f_{1},f_{2},\ldots,f_{s})$
such that $f_{1},f_{2},\ldots,f_{s}\in O_{A}^{n}.$
Let $O_{A,s} = \bigcup \limits_{n \ge 1} O_{A,s}^{n}.$
Elements of $O_{A,s}$ are called \emph{vector-functions}.

Then, in a natural way we define clones of such vector-functions.
It is proved in~\cite{taymanov} that we have only countably many clones of vector-functions for every $s$.
Recall that in this section $A = \{0,1\}$.

A relation is called $s$-sorted if every variable of this relation has a sort from the set $\{1,2,\ldots,s\}$.
A set of $s$-sorted relations is called a \emph{relational clone} if it is
closed under positive primitive formulas (where we cannot identify variables of different sorts) and
contains equality and empty relations.
It is shown in~\cite{romov} that there exists a one-to-one correspondence
between clones of $O_{A,s}$ and relational clones of $s$-sorted relations.

Note that Theorem~\ref{OnlyLinearOn2Elements} holds for multi-sorted relations.
Thus, all clones of vector-functions on 2 elements
can be described by disjunctions of linear equations with variables of different sorts.
This idea gives a simple proof of the fact that the set of all such clones is countable.
Also, in~\cite[Section 6]{MinimalClones}
it was shown that there are two types of
essential relations on three elements
preserved by the semiprojection
$s_5(x,y,z):= \begin{cases}
x, & |\{x,y,z\}|<3\\
y, & |\{x,y,z\}|=3\\
\end{cases}$ :
graphs of permutations and
relations whose projection onto every coordinate is a 2-element set.
The latter relations can be observed as multi-sorted relations with variables of 3 sorts,
where sorts depend on the projection onto the corresponding coordinate, that is  $\{0,1\}$,$\{1,2\}$, or $\{0,2\}.$
Therefore, clones on three elements containing this semiprojection can be described by disjunctions of linear equations.

Unfortunately, in general, for $|A|>2$, the author could not find such a nice characterization.
For example, the relation $\left(\begin{smallmatrix}
0 & 0 & 0 & 0 & 1 & 1 & 1 & 1 & 2 & 2 & 2 & 2\\
0 & 1 & 1 & 2 & 0 & 1 & 2 & 2 & 0 & 0 & 0 & 1\\
1 & 0 & 1 & 2 & 0 & 0 & 0 & 1 & 0 & 1 & 2 & 2
\end{smallmatrix}\right)$ is a key relation and
$\left(\begin{smallmatrix}
0 \\0 \\0
\end{smallmatrix}\right)$ is a key tuple for this relation.
However, 
the only proof of the above fact the author knows is to check manually that any tuple which is not from the relation
can be mapped to the key tuple by a vector-function preserving $\rho$.

\subsection{Pattern of a key relation}

For a relation $\rho\in R_{A}^{n}$ we define a binary relation on the set $\{1,2,\ldots,n\}$.
We say that $i\overset{\rho}{\sim} j$ if there do not exist
$a_1,\ldots,a_n,b_i,b_j\in A$ such that
%
$\left(\begin{smallmatrix}
a_1\\\dots\\a_{i-1}\\a_i\\a_{i+1}\\\dots\\a_{j-1}\\a_j\\a_{j+1}\\\dots\\a_n
\end{smallmatrix}\right)\notin \rho$,
$\left(\begin{smallmatrix}
a_1\\\dots\\a_{i-1}\\a_i\\a_{i+1}\\\dots\\a_{j-1}\\b_j\\a_{j+1}\\\dots\\a_n
\end{smallmatrix}\right),
\left(\begin{smallmatrix}
a_1\\\dots\\a_{i-1}\\b_i\\a_{i+1}\\\dots\\a_{j-1}\\a_j\\a_{j+1}\\\dots\\a_n
\end{smallmatrix}\right),
\left(\begin{smallmatrix}
a_1\\\dots\\a_{i-1}\\b_i\\a_{i+1}\\\dots\\a_{j-1}\\b_j\\a_{j+1}\\\dots\\a_n
\end{smallmatrix}\right)\in \rho$.
We put by definition that $i\overset{\rho}{\sim}i$ for every $i\in\{1,2,\ldots,n\}$.
The next lemma follows from the definition of a key tuple.

\begin{lem}\label{patternAndKeyTuple}

Suppose $\left(\begin{smallmatrix}a_1\\\vdots\\a_n\end{smallmatrix}\right)$ is a key tuple for $\rho$.
Then $i\overset{\rho}{\sim} j$ if and only if there do not exist $b_i,b_j\in A$ such that
$\left(\begin{smallmatrix}
a_1\\\dots\\a_{i-1}\\a_i\\a_{i+1}\\\dots\\a_{j-1}\\b_j\\a_{j+1}\\\dots\\a_n
\end{smallmatrix}\right),
\left(\begin{smallmatrix}
a_1\\\dots\\a_{i-1}\\b_i\\a_{i+1}\\\dots\\a_{j-1}\\a_j\\a_{j+1}\\\dots\\a_n
\end{smallmatrix}\right),
\left(\begin{smallmatrix}
a_1\\\dots\\a_{i-1}\\b_i\\a_{i+1}\\\dots\\a_{j-1}\\b_j\\a_{j+1}\\\dots\\a_n
\end{smallmatrix}\right)\in \rho$.

\end{lem}

\begin{cor}\label{PatternChanging}

Suppose $(a_{1},\ldots,a_{n})$ is a key tuple for $\rho$,
$\sigma(x_{1},\ldots,x_{n-1}) = \rho(x_{1},\ldots,x_{n-1},a_{n})$.
Then for $i,j\in \{1,2,\ldots,n-1\}$ we have
$i\overset{\rho}{\sim} j \Leftrightarrow i\overset{\sigma}{\sim} j$.

\end{cor}

The relation $\overset{\rho}{\sim}$ is called the \emph{pattern} of $\rho$.

Unfortunately, this relation is not an equivalence relation in general.
To show this let us consider
a relation on the set $A = \{0,1,2,3\}$ defined as follows
\[\rho = \{(x,y,z) \mid x,y\in A, z\in \{0,2\}, x+y+z\in \{0,1\}\},\]
where ``$+$'' is addition modulo 4.
This relation is shown in the following figure.

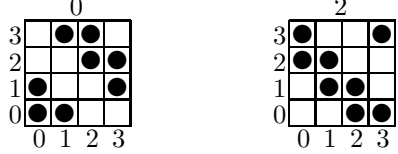
\begin{picture}(120,60)

\put(27,52){0}
\put(127,52){2}

\put(4,11){0}
\put(4,21){1}
\put(4,31){2}
\put(4,41){3}

\put(104,11){0}
\put(104,21){1}
\put(104,31){2}
\put(104,41){3}

\put(13,2){0}
\put(23,2){1}
\put(33,2){2}
\put(43,2){3}

\put(113,2){0}
\put(123,2){1}
\put(133,2){2}
\put(143,2){3}

\multiput(10,10)(0,10){5}
{\line(1,0){40}}
\multiput(10,10)(10,0){5}
{\line(0,1){40}}

\multiput(110,10)(0,10){5}
{\line(1,0){40}}
\multiput(110,10)(10,0){5}
{\line(0,1){40}}

\put(15,15){\circle*{7}}
\put(25,15){\circle*{7}}
\put(15,25){\circle*{7}}
\put(45,25){\circle*{7}}
\put(25,45){\circle*{7}}
\put(35,35){\circle*{7}}
\put(45,35){\circle*{7}}
\put(35,45){\circle*{7}}

\put(145,15){\circle*{7}}
\put(135,25){\circle*{7}}
\put(125,35){\circle*{7}}
\put(115,45){\circle*{7}}

\put(135,15){\circle*{7}}
\put(125,25){\circle*{7}}
\put(115,35){\circle*{7}}

\put(145,45){\circle*{7}}

\end{picture}

It is easy to see that $1\overset{\rho}{\sim} 3$, $2\overset{\rho}{\sim} 3$, and $1\overset{\rho}{\not\sim} 2$.
To prove that the relation $\rho$ is a key relation we just need to show
that for every two tuples \[\alpha_{1},\alpha_{2}\in (A\times A\times \{0,2\})\setminus\rho\]
there exists a vector-function $\Psi$ preserving $\rho$ such that
$\Psi(\alpha_{1}) = \alpha_{2}$.
Combining
bijective vector-functions $(x+1,y-1,z)$, $(x+2,y,z+2)$, and $(-x,1-y,z)$,
we can easily get all necessary vector-functions.

Nevertheless, as we prove later, if $\rho$ is preserved by a WNU then the pattern is an equivalence relation.

\section{Main Results}\label{MainResultsSection}

The aim of this section is to formulate the main results of the paper.
Here, we present the characterization of key relations preserved by a WNU.
Then, we consider three special cases of the pattern (a trivial equivalence relation,
an almost trivial equivalence relation, and a full equivalence relation)
in more details and provide stronger statements for these cases.
All the statements in this section are listed without the proof and will be proved in the next sections.

\subsection{The pattern of a key relation}

\begin{thm}\label{PatternIsEquivalenceForWNUF}

Suppose $\rho$ is a key essential relation preserved by a WNU. Then the pattern of $\rho$ is an equivalence relation.
Moreover, at most one equivalence class contains more than one element.

\end{thm}

We say that the pattern is \emph{full} if it is a full equivalence relation,
the pattern is \emph{trivial} if it is a trivial equivalence relation,
the pattern is \emph{almost trivial} if it is an equivalence relation such that
just one equivalence class contains 2 elements, all other classes contain one
element.

\subsection{A characterization of key relations preserved by a WNU}

As we know from Section~\ref{twovaluedcaseSection} every key relation on two elements can be represented as a disjunction of
linear equations. Moreover, if this relation is preserved by a WNU, then we can check that only one equation
contains more than one variable.
To show this, it is sufficient to check that every minimal WNU on 2 elements (conjunction, disjunction, majority operation, x+y+z)
cannot preserve a disjunction of several nontrivial equations.
In this section we will generalize this statement for bigger sets.
In the paper we always assume that $a\neq b$ if we consider the set $\{a,b\}$.

\begin{thm}\label{MainTheorem}
Suppose $\rho$ is a key essential relation of arity $n$ preserved by a WNU
whose pattern is $\{\{1,2,\ldots,r\},\{r+1\},\{r+2\},\ldots,\{n\}\},$ $r\ge 1$.
Then for every key tuple $(a_1,\ldots,a_n)$ there exist
$\boldsymbol B = B_1\times B_2 \times \dots \times B_n$,
a prime number $p$
and bijective mappings
$\phi_i: B_i\to \mathbb Z_{p}$ for $i=1,2,\ldots,r$ such that
$(a_1,\ldots,a_n) \in \boldsymbol B$,
$B_i = \{a_i,b_i\}$ for $i=r+1,\ldots,n$,
\[
\rho \cap \boldsymbol B
=
(\phi_1(x_1) + \ldots +\phi_r(x_r) = 0)
\vee
(x_{r+1} = b_{r+1})
\vee \dots \vee (x_n = b_n),
\]
and every tuple $\gamma \in \boldsymbol B\setminus \rho$ is a key tuple for $\rho$.
\end{thm}

This means, that in every key relation preserved by a WNU we can find a part $\boldsymbol B$ which is
well-organized. This part is defined as a disjunction of at most one nontrivial linear equation and
several trivial linear equations. Thus, we proved the statement which is very similar to the statement we have for $|A|=2$.

\subsection{Key relations with trivial pattern}

Here we consider the first special case of a pattern, i.e. a key relation whose pattern is a trivial equivalence relation.

By Theorem~\ref{MainTheorem} for any key essential
relation $\rho$ with trivial pattern preserved by a WNU
we can find  a part 
which is organized as follows:
\[(x_1 = b_1) \vee (x_2=b_2) \vee  \dots \vee (x_n=b_n),\]
or equivalently
there exist
$(a_1,a_2,\ldots,a_n)\notin\rho$ and $b_1,b_2,\ldots,b_n\in A$ such that
\[(\{a_1,b_1\}\times \{a_2,b_2\}\times \dots\times\{a_n,b_n\})\setminus \{(a_1,a_2,\ldots,a_n)\}\subseteq \rho.\]

It turned out that this is not only a necessary condition but also a sufficient condition.

\begin{lem}\label{trivialpatternLem}

Suppose $\rho\in R_{A}^{n}$,
$(a_1,a_2,\ldots,a_n)\notin\rho$,
$b_{1},b_2,\ldots,b_{n}\in A$, 
and
$(\{a_1,b_1\}\times \{a_2,b_2\}\times \dots\times\{a_n,b_n\})\setminus \{(a_1,a_2,\ldots,a_n)\}\subseteq \rho.$
Then $\rho$ is a key relation and $(a_1,a_2,\ldots,a_n)$ is a key tuple for the relation $\rho$.

\end{lem}

Thus, we have the following characterization of key relations with trivial pattern preserved by a WNU.

\begin{thm}\label{trivialPattern}

Suppose $\rho$ is a relation preserved
by a WNU whose pattern is a trivial equivalence relation.
Then $\rho$ is a key relation if and only if
there exist $(a_1,a_2,\ldots,a_n)\notin\rho$ and
$b_1,b_2,\ldots,b_n\in A$ such that
\[(\{a_1,b_1\}\times \{a_2,b_2\}\times \dots\times\{a_n,b_n\})\setminus \{(a_1,a_2,\ldots,a_n)\}\subseteq \rho.\]

\end{thm}

The following example shows that the existence of a WNU preserving the relation is a necessary condition.
We consider the relation $\left(\begin{smallmatrix}
0 & 0 & 0 & 0 & 1 & 1 & 1 & 1 & 2 & 2 & 2 & 2\\
0 & 1 & 1 & 2 & 0 & 1 & 2 & 2 & 0 & 0 & 0 & 1\\
1 & 0 & 1 & 2 & 0 & 0 & 0 & 1 & 0 & 1 & 2 & 2
\end{smallmatrix}\right)$, which was already mentioned in Section~\ref{twovaluedcaseSection}.
The pattern of this relation is trivial,
$(0,0,0)$ is the only key tuple
but we cannot find
$b_1,b_2,b_3\in \{1,2\}$ such that
$(\{0,b_1\}\times \{0,b_2\}\times\{0,b_3\})\setminus \{(0,0,0)\}\subseteq \rho.$
Thus we have a key relation with trivial pattern which does not satisfy the condition of  Theorem~\ref{trivialPattern}.

The case when the pattern of a key relation is a trivial equivalence relation
arises if the relation is preserved by a near-unanimity function, where
\emph{a near unanimity function} is a function $f$ satisfying
\[f(x,\ldots,x,y) = f(x,\ldots,x,y,x) = \dots = f(y,x,\ldots,x)=x.\]

\begin{thm}\label{NUtheoremLabel}
Suppose $\rho$ is a key essential relation of arity greater than 2 preserved by a near-unanimity function.
Then the pattern of $\rho$ is a trivial equivalence relation.
\end{thm}

\begin{cor}

Suppose $\rho$ (of arity greater than 2) is preserved by a near-unanimity function.
Then $\rho$ is a key essential relation if and only if there exist
$(a_1,a_2,\ldots,a_n)\notin\rho$ and $b_1,b_2,\ldots,b_n\in A$  such that
\[(\{a_1,b_1\}\times \{a_2,b_2\}\times \dots\times\{a_n,b_n\})\setminus \{(a_1,a_2,\ldots,a_n)\}\subseteq \rho.\]

\end{cor}

\subsection{Key relations with almost trivial pattern}

Suppose $\rho$ is a key essential relation preserved by a WNU whose pattern is $\{\{1,2\},\{3\},\ldots,\{n\}\}$.
By Theorem~\ref{MainTheorem} we can find  a part
which is organized as follows:
$(x_1 + x_2=0) \vee (x_3=b_3)\vee \dots \vee (x_n=b_n)$.
Hence, there exist
$(a_1,a_2,\ldots,a_n)\notin\rho$ and $b_1,b_2\ldots,b_n\in A$ such that
\[(\{a_1,b_1\}\times \dots\times\{a_n,b_n\})\setminus \{(a_1,a_2,a_3,\ldots,a_n),
(b_1,b_2,a_3,\ldots,a_n)\}\subseteq \rho.\]

It turned out that this is not only a necessary condition but also a sufficient condition.

\begin{lem}\label{AlmostTrivialPattern}

Suppose $1\overset{\rho}{\sim}2$, $(a_1,a_2,\ldots,a_n)\notin\rho$, $b_{1},\ldots,b_{n}\in A$
and
\[(\{a_1,b_1\}\times \dots\times\{a_n,b_n\})\setminus \{(a_1,a_2,a_3,\ldots,a_n),
(b_1,b_2,a_3,\ldots,a_n)\}\subseteq \rho.\]
Then $\rho$ is a key relation and $(a_1,a_2,\ldots,a_n)$ is a key tuple for $\rho$.

\end{lem}

Thus, we have the following characterization of key relations with almost trivial pattern preserved by a WNU.

\begin{thm}

Suppose $\rho$ is a relation preserved
by a WNU, the pattern of $\rho$ is $\{\{1,2\},\{3\},\ldots,\{n\}\}$.
Then $\rho$ is a key essential relation iff
there exist $(a_1,\ldots,a_n)\notin\rho$ and
$b_1,\ldots,b_n\in A$ such that
\[(\{a_1,b_1\}\times \dots\times\{a_n,b_n\})\setminus \{(a_1,a_2,a_3,\ldots,a_n),
(b_1,b_2,a_3,\ldots,a_n)\}\subseteq \rho.\]

\end{thm}

The case when the pattern of a key relation is almost trivial arises if we consider
relations preserved by a 2-semilattice operation or a semilattice operation.
A \emph{semilattice operation} is a binary associative commutative idempotent operation.
A \emph{2-semilattice operation} is a binary commutative idempotent operation satisfying
$f(x,f(x,y)) = f(x,y)$.

\begin{thm}\label{aboutsemilattice}
Suppose $\rho$ is a key essential relation preserved by a semilattice operation or a 2-semilattice operation.
Then the pattern of $\rho$ is either trivial, or almost trivial.
\end{thm}

\begin{cor}

Suppose $\rho$ is a relation preserved by a semilattice or a 2-semilattice operation,
$1\overset{\rho}{\sim}2$.
Then $\rho$ is a key essential relation if and only if there exist
$(a_1,a_2,\ldots,a_n)\notin\rho$ and $b_1,b_2,\ldots,b_n\in A$ such that
\[(\{a_1,b_1\}\times \dots\times\{a_n,b_n\})\setminus \{(a_1,a_2,a_3,\ldots,a_n),
(b_1,b_2,a_3,\ldots,a_n)\}\subseteq \rho.\]

\end{cor}

\subsection{Key relations with full pattern}

Here we consider key relations
preserved by a WNU whose pattern is a full equivalence relation.
For example, this case arises if a relation is preserved by a Mal'tsev operation.

We state that
any key relation can be divided into blocks
such that every block is defined by a linear equation.

Recall that $\widetilde \rho$ is a relation $\rho$ filled up with
all essential tuples.
We define a graph whose vertices are tuples from $\widetilde \rho$.
Two tuples are \emph{adjacent} in the graph if they differ just in one element.
Then tuples of $\widetilde \rho$ can be divided into connected components.
A connected component of $\widetilde \rho$ is called \emph{a block of $\rho$}.
A block is called \emph{trivial} if it contains only tuples from $\rho$.
We have the following characterization of key relations with full pattern.

\begin{thm}\label{mainTheoremForFullPattern}

Suppose $\rho$ is a key essential relation  of arity greater than 2 preserved by a WNU, the pattern of $\rho$ is
a full equivalence relation. Then
\begin{enumerate}
\item Every block of $\rho$ equals $B_{1}\times \dots \times B_{n}$ for some
$B_{1},\ldots,B_{n}\subseteq A$.

\item For every nontrivial block $\boldsymbol B = B_{1}\times \dots \times B_{n}$ the intersection
$\rho\cap\boldsymbol B$
can be defined as follows.
There exist an abelian group $(G;+,-,0)$, whose order is a power of a prime number,
and surjective mappings
$\phi_i: B_{i}\to G$ for $i=1,2,\ldots,n$ such that
\[\rho\cap \boldsymbol B = \{(x_1,\ldots,x_n)\mid \phi_1(x_1)+\phi_2(x_2) + \ldots +\phi_n(x_n) = 0\}.\]
\end{enumerate}

\end{thm}

Note that the existence of a WNU preserving the relation $\rho$ is a necessary condition.
As a counterexample, let us consider the following key relation.
Let $s_0,s_1,\ldots,s_5$ be all permutations on the set $\{0,1,2\}$.
Put $A= \{0,1,2,\ldots,5\}$ and
\[\rho = \{(i,a,b) \mid i\in \{0,1,\ldots,5\},a,b\in \{0,1,2\},s_i(a) = b\}.\]
Let us show that $\rho$ is a key relation and every tuple
\[\alpha \in (\{0,1,\ldots,5\}\times \{0,1,2\}\times \{0,1,2\})\setminus \rho\] is a key tuple.
For every two permutations $\psi_2$, $\psi_3$ on the set $\{0,1,2\}$ we can find an appropriate
permutation $\psi_1$ on the set $\{0,1,\ldots,5\}$
such that the vector-function $(\psi_1,\psi_2,\psi_3)$ preserves $\rho$.
It is easy to see that using these vector-functions we can map
any tuple from $\{0,1,\ldots,5\}\times \{0,1,2\}\times \{0,1,2\}$ to $\alpha$.
The pattern of the relation is a full equivalence relation.
We can check that this relation doesn't satisfy the statement of the theorem.

\section{Definitions and Notations}

In this section we give the remaining definitions we need in the paper.

By $O_{A,s}^{n}$
we denote the set of all tuples $(f_{1},f_{2},\ldots,f_{s})$
such that \[f_{1},f_{2},\ldots,f_{s}\in O_{A}^{n}\]
Let $O_{A,s} = \bigcup \limits_{n \ge 1} O_{A,s}^{n}.$
The tuple $(f_{1},f_{2},\ldots,f_{s})$ is called a \emph{vector-function}.
To distinguish vector-functions and functions, we denote vector-functions with bold symbols,
except for unary vector functions which we usually denote by capital Greek letters.
For a vector function $\boldsymbol f\in O_{A,s}^{n}$,
the corresponding tuple of functions is
$(\boldsymbol f^{(1)},\boldsymbol f^{(2)},\ldots,\boldsymbol f^{(s)}).$
We define the composition for vector-functions in the following natural way.
The equation \[\boldsymbol h(x_{1},\ldots,x_n) =
\boldsymbol f(\boldsymbol g_1(x_{1},\ldots,x_n),\ldots,\boldsymbol g_{n}(x_{1},\ldots,x_n))\]
means that for every $i\in\{1,2,\ldots,s\}$ we have
\[\boldsymbol h^{(i)}(x_{1},\ldots,x_n) =
\boldsymbol f^{(i)}(\boldsymbol g_1^{(i)}(x_{1},\ldots,x_n),\ldots,\boldsymbol g_{n}^{(i)}(x_{1},\ldots,x_n)).\]
\emph{A clone of vector-functions} is a set of vector-functions closed under composition
and containing the vector function $(id,id,\ldots,id)$, where $id(x) = x$ for every $x\in A$.
A vector-function is called a WNU if every function in it is a WNU.

As it was mentioned in Section~\ref{twovaluedcaseSection},
a relation is called $s$-sorted (or multi-sorted) if every variable of this relation has a sort from the set $\{1,2,\ldots,s\}$.
The set of all $s$-sorted relations we denote by $R_{A,s}.$

Let $\boldsymbol \rho$ be an $s$-sorted relation of arity $h$,
and $r_{i}$ be the sort of the $i$-th variable for every $i\in\{1,2,\ldots,h\}$.
We say that $\boldsymbol f$
\emph{preserves} $\boldsymbol \rho$ 
if
\[\boldsymbol f\begin{pmatrix}
a_{1,1} & a_{1,2} & \dots & a_{1,n}\\
a_{2,1} & a_{2,2} & \dots & a_{2,n}\\
\vdots & \vdots & \ddots & \vdots\\
a_{h,1} & a_{h,2} & \dots & a_{h,n}
 \end{pmatrix}:=
 \begin{pmatrix}
\boldsymbol f^{(r_{1})}(a_{1,1} , a_{1,2} , \ldots , a_{1,n})\\
\boldsymbol f^{(r_{2})}(a_{2,1} , a_{2,2} , \ldots , a_{2,n})\\
 \vdots\\
\boldsymbol f^{(r_{h})}(a_{h,1} , a_{h,2} , \ldots , a_{h,n})\\
 \end{pmatrix}
 \in \boldsymbol \rho\]
for all
\[
\begin{pmatrix}a_{1,1}\\ a_{2,1}\\ \vdots \\a_{h,1}\end{pmatrix},
\begin{pmatrix}a_{1,2}\\ a_{2,2}\\ \vdots \\a_{h,2}\end{pmatrix},
 \ldots,
\begin{pmatrix}a_{1,n}\\ a_{2,n}\\ \vdots \\a_{h,n}\end{pmatrix}
\in \boldsymbol \rho.\]

Suppose $i\in \{1,2,\ldots,s\},$ then by $\sigma_{=}^{i,s}$ we denote the $s$-sorted relation
whose variables are of the $i$-th sort
such that $(x,y)\in\sigma_{=}^{i,s}\Longleftrightarrow (x = y).$
By $false$ we denote the empty relation of arity 0.
Put $\Sigma_{s} = \{\sigma_{=}^{i,s} \mid 1\le i \le s\}\cup\{false\}.$
In the same way as for the set $R_{A}$
we can define the closure operator on the set $R_{A,s}.$
Suppose $S\subseteq R_{A,s},$ then by $[S]$ we denote the set of all
$s$-sorted relations $\sigma\in R_{A,s}$ that can be represented by a positive primitive formula:
\[\boldsymbol\rho (x_{1},\ldots,x_{n}) = \exists
y_{1}\ldots \exists y_{l} \;
\boldsymbol\rho_{1}(z_{1,1},\ldots,z_{1,n_{1}})\wedge \ldots \wedge
\boldsymbol\rho_{m}(z_{m,1},\ldots,z_{m,n_{m}}),\]
where $\boldsymbol\rho_{1},\ldots,\boldsymbol\rho_{m}\in S,$ the variable symbols  
$z_{i,j}\in \{x_{1},\ldots,x_{n},y_{1},\ldots,y_{l}\}$ are subject to the following restriction:
if a variable is substituted in some relation as a variable of the $l$-th sort, then
this variable cannot be substituted in any relation as a variable of other sort.
It is shown in~\cite{romov} that there exists a one-to-one correspondence
between clones of $O_{A,s}$ 
and closed subsets of $R_{A,s}$ containing $\Sigma_{s}$.

In this paper we consider only two types of relations.
First, relations whose variables are of different sorts, moreover, the $i$-th variable
has the $i$-th sort. Second, relations whose variables are of one sort.
To distinguish them
relations with variables of different sorts we denote by bold symbols, like $\boldsymbol\rho$, $\boldsymbol\delta$.
Relations from $R_{A}$ are considered as relations with variables of the first sort.

For $c\in A$ and $i\in \{1,2,\ldots,s\}$ by $=_c^{(i)}$ we denote the unary relation
with the variable of sort $i$ containing only element $c$.

Suppose $\Psi_{1}$ and $\Psi_2$ are unary vector functions.
By $\Psi_{1}\circ\Psi_2$ we denote the unary vector-function
$\Psi$ defined as follows $\Psi(x) = \Psi_{1}(\Psi_2(x))$.

By $\ar(\rho)$ we denote the arity of the relation $\rho$,
by $\ar(f)$ we denote the arity of the function $f$.
By 0 we always denote the identity in an abelian group or the additive identity for a field.

For $\rho\in R_{A}^n$ and $i\in\{1,2,\ldots,n\}$
by $\proj_{i}\rho$ we denote the projection of $\rho$ onto the $i$-th coordinate,
that is
\[\proj_{i}\rho = \{c\mid \exists a_1\dots\exists a_n \colon(a_{1},\ldots,a_{i-1},c,a_{i+1},\ldots,a_n)\in \rho\}.\]
Denote $\proj\rho = \proj_{1}\rho\times \proj_{2}\rho\times \dots\times\proj_{n}\rho.$
We say that a tuple $(b_{1},\ldots,b_{n})$
\emph{witnesses} that
$(a_{1},\ldots,a_{n})$ is an essential tuple for $\rho$
if for every $i\in\{1,2,\ldots,n\}$
we have
$(a_{1},\ldots,a_{i-1},b_{i},a_{i+1},\ldots,a_{n})\in\rho$.

By $\Key(\boldsymbol\rho)$ we denote the relation $\boldsymbol\rho$ filled up with all key tuples for $\boldsymbol\rho$.

To simplify explanation,
we sometimes define tuples as words, for example
$a^{n}b^{m}$ is the tuple
$(\underbrace{a,\ldots,a}_n,\underbrace{b,\ldots,b}_m)$.

As it was mentioned in Section~\ref{MainResultsSection},
sometimes we consider a graph corresponding to a relation, where
tuples are vertices, and two tuples are adjacent if they differ just in one element.
Then we may consider a path in the graph and connected components of this graph.
Usually, we refer to a path or a connected component of the graph as
to a path of the relation and a connected component of the relation.

\section{Auxiliary statements}

\begin{lem}\label{DammuVariables}
Suppose $\rho$ is a key relation. Then $\rho$ is essential if and only if $\rho$ has no dummy variables.
\end{lem}

\begin{proof}

Assume that $\rho$ has no dummy variables and $\ar(\rho) = n$.
Let us prove that a key tuple $(a_1,\ldots, a_n)$ is also an essential tuple.
Let $i\in \{1,2,\ldots, n\}$. Since $\rho$ has no dummy variables,
there exist a tuple $(c_1,\ldots, c_n)\notin \rho$ and $d_i$ such that
$(c_1,\ldots,c_{i-1},d_{i},c_{i+1},\ldots, c_n)\in \rho$.
We know that $(c_1,\ldots, c_n)$ can be mapped to the key tuple by a vector-function
$(\psi_{1},\ldots,\psi_{n})$ which
preserves $\rho$. Therefore,
\begin{multline*}
(\psi_1(c_1),\ldots,\psi_{i-1}(c_{i-1}),\psi_{i}(d_{i}),\psi_{i+1}(c_{i+1}),\ldots,\psi_{n}(c_{n})= \\
(a_1,\ldots,a_{i-1},\psi_{i}(d_{i}),a_{i+1},\ldots,a_{n})\in \rho.
\end{multline*}
Thus, for every $i\in \{1,2,\ldots, n\}$ we can change the $i$-th component of the key tuple to get a tuple from $\rho$.
Then, $(a_1,\ldots, a_n)$ is an essential tuple and, by Lemma \ref{sushnabor}, $\rho$ is essential.

It is obvious, that a relation that has dummy variables cannot be essential.
\end{proof}

\begin{lem}\label{KeyIsAlwaysKey}
Suppose $\sigma\in R_{A}$, $\rho$ = $\sigma\times A^{s}$.
Then $\rho$ is a key relation if and only if $\sigma$ is a key relation.
\end{lem}

\begin{proof}

Assume that $\sigma$ is a key relation and $\alpha$ is a key tuple for $\sigma$.
Choose $\beta\in A^{s}$.
Let us prove that $\alpha\beta$ is a key tuple for $\rho$.
Suppose $\delta\in A^{\ar(\rho)}\setminus\rho$.
Remove the last $s$ elements of $\delta$ to get a tuple $\delta'$.
Obviously, $\delta'\notin \sigma$.
Then there exists a unary vector-function $\Psi$ which maps $\delta'$ to $\alpha$.
Define a unary vector-function $\Psi'$ as follows
$\Psi'^{(i)} = \Psi^{(i)}$ for $i\le \ar(\sigma)$,
$\Psi'^{(i)} = \beta(i-\ar(\sigma))$ for $i> \ar(\sigma)$.
We can check that $\Psi'$ maps $\delta$ to $\alpha\beta$ and preserves $\rho$.
Thus, $\rho$ is a key relation.

Assume that $\rho$ is a key relation and $\alpha$ is a key tuple for $\rho$.
Let $\beta$ be obtained from $\alpha$ by removing the last $s$ elements.
To prove that $\beta$ is a key tuple for $\sigma$, we
just add random $s$ elements to the end of a tuple $\gamma\in A^{\ar(\sigma)}\setminus \sigma$,
and consider a vector-function preserving
$\rho$ that maps the obtained tuple to $\alpha$. Then we remove the last $s$ functions of the vector-function to get a
vector-function that maps $\gamma$ to $\beta$ and preserves $\sigma$.
\end{proof}

\begin{lem}\label{ReduceArityOfKey}
Suppose $\sigma(x_2,\ldots,x_n) = \rho(b_1,x_2,\ldots,x_n)$,
$(b_1,\ldots,b_n)$ is a key tuple for $\rho$.
Then $\sigma$ is a key relation and
$(b_2,\ldots,b_n)$ is a key tuple for $\sigma$.
\end{lem}

\begin{proof}

For every $(c_{2},\ldots,c_{n})\notin \sigma$ we need to find a vector-function which maps
$(c_{2},\ldots,c_{n})$ to $(b_2,\ldots,b_n)$.
We know that $(b_{1},c_{2},\ldots,c_{n})\notin \rho$,
therefore there exists a vector-function $\Psi$ which maps $(b_{1},c_{2},\ldots,c_{n})$
to $(b_1,b_2,\ldots,b_n)$ and preserves $\rho$.
It is easy to check that the vector function $(\Psi^{(2)},\ldots,\Psi^{(n)})$ preserves the relation $\sigma$
and maps $(c_{2},\ldots,c_{n})$ to $(b_2,\ldots,b_n)$.
This completes the proof.
\end{proof}

\begin{lem}\label{KeyGotoRho}
Suppose $\alpha$ is a key tuple for $\rho$, and a unary vector-function $\Psi$ preserves
$\rho$. Then either $\Psi(\alpha)\in \rho$, or $\Psi(\alpha)$ is a key tuple for $\rho$.
\end{lem}

\begin{proof}
Assume that $\Psi(\alpha)\notin \rho$. 
Let $\beta\notin \rho$.
We know that
there exists a vector-function $\Psi'$ which maps
$\beta$ to $\alpha$. Then
$\Psi\circ\Psi'$ is a vector function preserving $\rho$ which maps $\beta$ to $\Psi(\alpha)$.
Hence $\Psi(\alpha)$ is a key tuple for $\rho$.
\end{proof}

\begin{lem}\label{findbetterWNU}

Suppose $\boldsymbol f\in O_{A,s}$ is a WNU.
Then using composition we can derive a WNU $\boldsymbol f'$
such that
for every $\alpha\in A^{s}$ and $\boldsymbol h(x) = \boldsymbol f'(\alpha,\alpha,\ldots,\alpha,x)$
we have
$\boldsymbol h(\boldsymbol h(x)) = \boldsymbol h(x)$.

\end{lem}

\begin{proof}

Let $\boldsymbol f_{1} = \boldsymbol f$, $m = \ar(f)$.
Put
\[\boldsymbol f_{i+1}(x_{1}, \ldots, x_{m^{i+1}}) = \boldsymbol f(\boldsymbol f_{i}(x_{1},\ldots,x_{m^{i}}),\ldots,
\boldsymbol f_{i}(x_{m^{i}(m-1)+1},\ldots,x_{m^{i+1}})),\]
$\boldsymbol h_{i}(x) = \boldsymbol f_{i}(\alpha,\ldots,\alpha,x)$.
We can easily check that
$\boldsymbol h_{i+1}(x) = \boldsymbol h_1(\boldsymbol h_{i}(x))$.
Therefore, for $k = |A|!$ we have
$\boldsymbol h_{k}(\boldsymbol h_{k}(x)) = \boldsymbol h_{k}(x)$, which means that
we can take $\boldsymbol f_k$ for $\boldsymbol f'$.
\end{proof}

\begin{lem}\label{DeriveKey}

Suppose $\boldsymbol f$ preserves a key relation $\boldsymbol\rho$.
Then $\boldsymbol f$ preserves $\Key(\boldsymbol\rho)$.

\end{lem}

\begin{proof}

Let $\alpha_1,\ldots,\alpha_{m}\in\Key(\boldsymbol\rho)$,
we need to show that $\beta = \boldsymbol f(\alpha_{1},\ldots,\alpha_m)\in \Key(\boldsymbol\rho)$.
Without loss of generality we assume that $\alpha_{i}$ is a key tuple if $i\le k$, and
$\alpha_{i}\in \boldsymbol\rho$ if $i>k$.
Let $\alpha$ be a key tuple for $\boldsymbol\rho$.
By the definition of a key tuple, for every $i\in\{1,2,\ldots,k\}$
there exists a vector-function $\Psi_{i}$ that preserves $\boldsymbol\rho$ and maps $\alpha$ to $\alpha_{i}$.
Put $\Psi(x) = \boldsymbol f(\Psi_{1}(x),\ldots,\Psi_{k}(x),\alpha_{k+1},\ldots,\alpha_{m})$.
Obviously $\Psi$ preserves $\boldsymbol\rho$ and maps $\alpha$ to $\beta$.
By Lemma~\ref{KeyGotoRho}, we obtain $\beta\in \Key(\boldsymbol\rho)$.
\end{proof}

Recall that
we have a Galois connection between clones of vector-functions and closed sets
of multi-sorted relations. Then it follows from the above lemma that the relation
$\Key(\boldsymbol\rho)$ can be derived from $\boldsymbol\rho$
using positive primitive formulas.

\begin{lem}\label{DeriveRhoTilde}

Suppose $\boldsymbol\rho$ is a multi-sorted relation, then
$\widetilde {\boldsymbol\rho} \in[\{\boldsymbol\rho\}].$

\end{lem}
\begin{proof}
It is sufficient to check the following positive primitive formula
\[\widetilde{\boldsymbol\rho}(x_1,\ldots,x_n) =
\exists y_1\dots\exists y_n
\bigwedge\limits_{j=1}^{n}\boldsymbol\rho(x_1,\ldots,x_{j-1},y_{j},x_{j+1},\ldots,x_{n}).\qedhere\]
\end{proof}

\begin{lem}\label{preserveConnectedComponent}

Suppose $\boldsymbol\rho$ is preserved by an idempotent vector-function $\boldsymbol f$.
Then $\boldsymbol f$ preserves every connected component of $\boldsymbol\rho$.

\end{lem}

\begin{proof}
Let $\boldsymbol\delta$ be a connected component of $\boldsymbol\rho$,
and $(b_1,\ldots,b_{n})\in\boldsymbol\delta$.
Let us define a sequence of relations of arity $n$.
Put $\boldsymbol\zeta_{0} = \{(b_{1},\ldots,b_{n})\}$,
\[\boldsymbol\zeta_{j+1}(x_1,\ldots,x_{n}) =
\exists y_1\dots\exists y_n
\bigwedge\limits_{i=0}^{n}
\boldsymbol\rho(x_1,\ldots,x_{i},y_{i+1},\ldots,y_{n})
\wedge \boldsymbol\zeta_{j}(y_1,\ldots,y_n).\]
Obviously, for
$j>|A|^{n}$ we get $\boldsymbol\zeta_{j} = \boldsymbol\delta$. Therefore,
\[\boldsymbol\delta\in [\{\boldsymbol\rho\}\cup \{=_{c}^{(i)}\mid c\in A, i=1,2,\ldots,n\}]\]
and $\boldsymbol\delta$ is preserved by an idempotent vector-function $\boldsymbol f$.
\end{proof}

In the remaining part of this section
we prove that a maximal clone defined by an $h$-universal relation cannot contain a
WNU.

Put $E_{k} = \{0,1,\ldots,k-1\}$ for every positive integer $k$.
We can represent every $a\in E_{h^{m}}$ uniquely  in the following form
\[a = a^{(m-1)}\cdot h^{m-1} + a^{(m-2)}\cdot h^{m-2} + \ldots
+ a^{(1)}\cdot h +a^{(0)}.\]

A relation $\rho \subseteq A^{ð}$ is called $h$-universal relation if there exist $m\ge 1$ and
a surjective mapping $q\colon A\to E_{h^{m}}$ such that
\[(a_0,\ldots,a_{h-1})\in\rho
\Leftrightarrow \forall i\in E_{m}\colon |\{(q(a_{0}))^{(i)},(q(a_{1}))^{(i)},\ldots,(q(a_{h-1}))^{(i)})\}|<h.\]

\begin{thm}\label{slupetskythm}\cite{LauSlupecky}\cite[Theorem 5.2.6.1]{lau}
Suppose $\rho$ is an $h$-universal relation and $q:A\to E_{h^{m}}$ is an appropriate mapping.
Then $f:A^{n}\to A$ belongs to $\Pol(\rho)$ if and only if
for every $i\in\{0,1,\ldots,m-1\}$ and $f_{i}(x_1,\ldots,x_n) := (q(f(x_1,\ldots,x_n))^{(i)}$
we have $|Im(f_{i})|<h$
or there exist
$j\in\{1,2,\ldots,n\}$, $\nu\in E_{m}$, a permutation
$s$ on $E_{h}$ such that
$f_{i}(x_1,\ldots,x_{n}) = s((q(x_{j}))^{(\nu)})$.
\end{thm}

\begin{cor}\label{slupetskyLem}
A WNU cannot preserve an $h$-universal relation.
\end{cor}
\begin{proof}
Assume the converse. Suppose a WNU $f\in O_{A}^{n}$ preserves an $h$-universal relation $\rho$.
Let $q:A\to E_{h^{m}}$ be an appropriate surjective mapping.
Put $f_{i}(x_1,\ldots,x_n) := (q(f(x_1,\ldots,x_n))^{(i)}$.
Then we apply Theorem~\ref{slupetskythm} to $\rho$ and $f$.
Since $f$ is idempotent, $|Im(f_{i})|= h$ for every $i$.
Hence for every $i$
there exist
$j\in\{1,2,\ldots,n\}$, $\nu\in E_{m}$, a permutation
$s$ on $E_{h}$ such that
$f_{i}(x_1,\ldots,x_{n}) = s((q(x_{j}))^{(\nu)})$.
This contradicts the fact that
\[f_{i}(x,y,\ldots,y)= f_{i}(y,x,y,\ldots,y) = \dots = f_{i}(y,\ldots,y,x).\qedhere\]
\end{proof} 

By Lemma~\ref{DammuVariables}
a key relation is essential if and only if it has no dummy variables.
That is why later in the paper we always assume that every key relation is essential.

\section{Strongly rich relations preserved by a WNU}

A relation $\boldsymbol\rho\subseteq A^{n}$ is called \emph{(strongly) rich}
if for every tuple
$(a_{1},\ldots,a_{n})$ and every $j\in \{1,\ldots,n\}$ there exists (a unique) $b\in A$
such that \[(a_{1},\ldots,a_{j-1},b,a_{j+1},\ldots,a_n)\in \boldsymbol\rho.\]

A relation of arity $n$ is called \emph{totally reflexive} if it contains
all tuples $(a_{1},\ldots,a_{n})\in A^{n}$ such that
$|\{a_1,\ldots,a_n\}|<n$.
A relation $\rho\in R_{A}^{n}$ is called \emph{symmetric} if
for every permutation $\sigma:\{1,2,\ldots,n\}\to\{1,2,\ldots,n\}$
we have
$\rho(x_1,\ldots,x_n) =
\rho(x_{\sigma(1)},\ldots,x_{\sigma(n)})$.
A relation $\rho\in R_{A}^{n}$ is called \emph{full} if $\rho = A^{n}$.

\begin{lem}\label{PermutationFromStronglyRich}

Suppose
$\boldsymbol\rho\in R_{A,n}$, $n\ge 3$, is a strongly rich relation.
Then for every $a,b\in A$ there exists a bijective mapping
$\psi:A\to A$ such that $\psi(a) = b$ and
$\psi(\sigma)\in [\{\boldsymbol\rho,\sigma\}\cup \{=_{c}^{(i)}\mid c\in A, i=1,2,\ldots,n\}]$
for every $\sigma\in R_{A}$.

\end{lem}

\begin{proof}

Let $(a,a_{2}, \ldots a_{n})$ be  a tuple from $\boldsymbol\rho$.
Since $\boldsymbol\rho$ is rich,
we can find  $c\in A$
such that
$(b,c, a_{3},\ldots a_{n})\in \rho$.
Let
\[\zeta(x,y) = \exists z \;\boldsymbol\rho(x,a_{2},\ldots,a_{n-1},z)\wedge \boldsymbol\rho(y,c,a_{3},\ldots,a_{n-1},z).\]
Since $\boldsymbol\rho$ is strongly rich, for every $d$ there exists a unique $e$ such that
$(d,e)\in \zeta$. Also, it is easy to see that $(a,b)\in \zeta$.
We define $\psi$ as follows $\psi(x) = y\Leftrightarrow (x,y)\in \zeta$.
Let $\sigma' = \psi(\sigma)$.
It is easy to check that
\[\sigma'(x_{1},\ldots,x_m) = \exists y_1 \dots \exists y_m \;\;\sigma(y_1,\ldots,y_m)
\wedge \zeta(y_{1},x_{1})\wedge \dots \wedge\zeta(y_{m},x_{m}).\]
This completes the proof.
\end{proof}

\begin{lem}\label{NoTotallyReflexiveA}
Suppose
$\sigma$ is a totally reflexive relation of arity $m\ge 2$ preserved by a WNU;
if $m=2$ then $\sigma$ is symmetric and the graph defined by $\sigma$ is connected;
for every $a,b\in A$ there exists a bijection $\psi_{a,b}:A\to A$ that
maps $a$ to $b$ and preserves $\sigma$.
Then $\sigma$ is a full relation.
\end{lem}

\begin{proof}

Assume that $\sigma$ is not a full relation.
Then $\Pol(\sigma)$ belongs to a maximal clone on $A$.
By Rosenberg Theorem \cite{rosmax, lau}, we have one of the following cases.
\begin{enumerate}

\item Maximal clone of monotone functions;

\item Maximal clone of autodual functions;

\item Maximal clone defined by an equivalence relation;

\item Maximal clone of quasi-linear functions;

\item Maximal clone defined by a central relation;

\item Maximal clone defined by an $h$-universal relation.

\end{enumerate}

If $m\ge 3$, then since $\sigma$ is totally reflexive, $\Pol(\sigma)$ contains all functions that take only two values.
If $m=2$, then since $\sigma$ defines a connected graph,
for every edge $(a,b)$ in this graph $\Pol(\sigma)$ contains all functions that take only two values $a$ and $b$.
Therefore cases (1), (2), (3) and (4) are not possible.

Since $\psi_{a,b}$ preserves $\sigma$ for every $a,b\in A$, $\psi_{a,b}$ belongs to the maximal clone.
Therefore case (5) cannot happen as well.
By Corollary~\ref{slupetskyLem},
a WNU cannot preserve an $h$-universal relation.
Thus, we get a contradiction, which means that
$\sigma$ is a full relation.
\end{proof}

\begin{lem}\label{NoTotallyReflexive}

Suppose
$\boldsymbol\rho\subseteq A^{n}$, $n\ge 3$, is a strongly rich relation,
$\sigma$ is a totally reflexive relation of arity $m\ge 3$,
$\boldsymbol\rho$ and $\sigma$ are preserved by a WNU $\boldsymbol f$.
Then $\sigma$ is a full relation.

\end{lem}

\begin{proof}

By Lemma~\ref{PermutationFromStronglyRich}, for every
$a,b\in A$ we have a permutation $\psi_{a,b}:A\to A$
such that for every $\delta\in R_{A}$
\[\psi_{a,b}(\delta)\in [\{\boldsymbol\rho,\delta\}\cup \{=_{c}^{(i)}\mid c\in A, i=1,2,\ldots,n\}].\]
Let $\Psi$ be the set of all permutations $\psi:A\to A$
satisfying the property
$\psi(\delta)\in [\{\boldsymbol\rho,\delta\}\cup \{=_{c}^{(i)}\mid c\in A, i=1,2,\ldots,n\}]$
for every $\delta\in R_{A}$.
Obviously, $\Psi$ is closed under composition.

Put $\sigma_{0} = \bigcap\limits_{\psi\in \Psi} \psi(\sigma)$.
For every $a,b\in A$ we have
\[\psi_{a,b}(\sigma_{0}) = \bigcap\limits_{\psi\in \Psi} \psi_{a,b}(\psi(\sigma))
\supseteq \bigcap\limits_{\psi\in \Psi} \psi(\sigma)=\sigma_0.\]
Since $\psi_{a,b}$ is a permutation,  $\sigma_{0}$ is preserved by $\psi_{a,b}$.

Since $\sigma_{0}$ is derived from $\boldsymbol\rho,\sigma$, and $\{=_{c}^{(i)}\mid c\in A, i=1,2,\ldots,n\}$,
$\sigma_{0}$ is preserved by a WNU $\boldsymbol f^{(1)}$.
We can check that $\sigma_{0}$ is totally reflexive.
By Lemma~\ref{NoTotallyReflexiveA}, $\sigma_{0}$ is a full relation, hence the relation $\sigma$ is also full.
\end{proof}

\begin{lem}\label{LinearWNU}

Suppose $(G;+)$ is a finite abelian group,
the relation $\sigma\subseteq G^{4}$ is defined by
$\sigma = \{(a_1,a_2,a_3,a_4)\mid a_1+a_2=a_3+a_4\}$,
$\sigma$ is preserved by a WNU $f$.
Then $f(x_{1},\ldots,x_{n}) = t\cdot x_{1}+t\cdot x_2 + \ldots + t\cdot x_{n}$
for some $t\in \{1,2,3,\ldots\}$.

\end{lem}

\begin{proof}

Denote $h(x) = f(0,0,\ldots,0,x)$.
Let us prove the equation \[f(x_{1},\ldots,x_{m},0,\ldots,0) = h(x_1)+\ldots+h(x_{m})\]
by induction on $m$.
For $m=1$ it follows from the definition.
We know that
$f\left(\begin{smallmatrix}
x_{1} & x_{2} & \dots & x_{m} & x_{m+1} & 0 & \dots & 0 \\
0     & 0     & \dots & 0     & 0       & 0 & \dots & 0 \\
x_{1} & x_{2} & \dots & x_{m} & 0       & 0 & \dots & 0 \\
0     & 0     & \dots & 0     & x_{m+1} & 0 & \dots & 0 \\
\end{smallmatrix}\right)\in \sigma$,
which by the inductive assumption gives
\begin{multline*}
f(x_{1},\ldots,x_{m},x_{m+1},0,\ldots,0) =\\
f(x_{1},\ldots,x_{m},0,\ldots,0) + h(x_{m+1}) =
h(x_1)+\ldots+h(x_{m})+h(x_{m+1}).
\end{multline*}

Thus, we know that
$f(x_{1},\ldots,x_{n}) = h(x_1)+\ldots+h(x_{n})$.
Let $k$ be the maximal order of an element in the group $(G;+)$.
We know that for every $a\in A$ we have $\underbrace{h(a)+h(a)+\ldots+h(a)}_n = a$.
Hence $k$ and $n$ are coprime,
and $h(x) = t\cdot x$ for any integer $t$ such that
$t\cdot n = 1 (\mod k)$.
This completes the proof.
\end{proof}

\begin{lem}\label{WNUIsLinearInAnyBlock}
Suppose $(G;+)$ is a finite abelian group,
the relation $\boldsymbol\rho\subseteq G^{n}$, $n>2$, is defined by
$\boldsymbol\rho = \{(a_{1},\ldots,a_{n})\mid a_{1}+\dots+a_{n}=0\}$,
$\boldsymbol \rho$ is preserved by a WNU $\boldsymbol f$ of arity $m$.
Then for every $j\in\{1,2,\ldots,n\}$
there exists $t\in\{1,2,3,\ldots\}$ such that
$\boldsymbol f^{(j)}(x_{1},\ldots,x_{m}) = t\cdot x_{1}+t\cdot x_2 + \ldots + t\cdot x_{m}$
\end{lem}
\begin{proof}
Without loss of generality we prove the statement just for $j=1$.
Let
\begin{multline*}\delta(x_1,x_2,x_3,x_4) =
\exists y_2\dots\exists y_n \exists z_2\dots\exists z_n\;
\boldsymbol\rho(x_{1},y_2,y_3,\ldots,y_n) \wedge \\
\boldsymbol\rho(x_{2},z_2,z_3,\ldots,z_n) \wedge
\boldsymbol\rho(x_{3},y_2,z_3,\ldots,z_n) \wedge
\boldsymbol\rho(x_{4},z_2,y_3,\ldots,y_n)
.\end{multline*}
It is easy to show that
$\delta = \{(a_1,a_2,a_3,a_4)\mid a_1+a_2=a_3+a_4\}.$
Then, the statement of the lemma follows from Lemma~\ref{LinearWNU}.
\end{proof}

\begin{lem}\label{OnlyOneLinearWithWNU}

Suppose $(A;+)$ is an abelian group ,
$\sigma, \delta\subseteq A^{4}$,
\[\sigma = \{(a_1,a_2,a_3,a_4)\mid a_1+a_2=a_3+a_4\};\]
if $(a_{1},a_2,a_3,a_4) \in \delta$ then
$a_1 = a_3 \Leftrightarrow a_2= a_4$ and
$a_1 = a_4 \Leftrightarrow a_2= a_3$;
$(a,b,a,b),(a,b,b,a)\in\delta$ for all $a,b\in A$;
$\sigma$ and $\delta$ are preserved by a WNU $f$.
Then $\sigma = \delta$.
\end{lem}

\begin{proof}

First, let us show that
$\sigma\subseteq \delta$.
We can assume that $f$ is chosen using Lemma~\ref{findbetterWNU}.
By Lemma~\ref{LinearWNU}
$f(x_{1},\ldots,x_{n}) = t\cdot x_{1}+t\cdot x_2 + \ldots + t\cdot x_{n}$.
Since $f$ is idempotent, $t$ is coprime to the order of the group.
Because of the condition from Lemma~\ref{findbetterWNU}
we have $f(0,0,\ldots,0,f(0,0,\ldots,0,x)) = f(0,0,\ldots,0,x)$,
hence $t^{2}\cdot x = t\cdot x = x$.
Suppose
$(a,b,c,d)\in \sigma$, then
$a+b = c+d$.
Since $b-c = d-a$ and $f$ preserves $\delta $ we get
$\left(\begin{smallmatrix}
a \\
b \\
c \\
d \\
\end{smallmatrix}\right)
=
f
\left(\begin{smallmatrix}
a & 0   & 0 & \dots & 0\\
c & b-c & 0 & \dots & 0 \\
c & 0 & 0 &\dots & 0 \\
a & d-a & 0 & \dots & 0
\end{smallmatrix}\right)\in \delta.$
Thus $\sigma\subseteq \delta$.
Assume that $\sigma\neq \delta$
and $(a,b,c,d)\in \delta\setminus \sigma$.
Then
$f
\left(\begin{smallmatrix}
a & -a   & 0 & \dots & 0 \\
b & -b   & 0 & \dots & 0 \\
c & -c & 0 &\dots & 0 \\
d & c-a-b & 0 & \dots & 0
\end{smallmatrix}\right) =
\left(\begin{smallmatrix}
0 \\
0 \\
0 \\
c+d-a-b \\
\end{smallmatrix}\right)\in \delta$.
Combining $a+b\neq c+d$
and
$a_1 = a_3 \Leftrightarrow a_2= a_4$
for $(a_{1},a_2,a_3,a_4) \in \delta$,
we get a contradiction.
\end{proof}

\begin{thm}\label{StronglyRichRelationTHM}

Suppose $\boldsymbol\rho\subseteq A^{n}$ is a strongly rich relation
preserved by a WNU.
Then there exists an abelian group $(A;+)$
and bijective mappings
$\phi_1$, $\phi_2$, \ldots,$\phi_n: A\to A$ such that
\[\boldsymbol\rho = \{(x_1,\ldots,x_n)\mid \phi_1(x_1)+\phi_2(x_2) + \ldots +\phi_n(x_n) = 0\}.\]

\end{thm}

\begin{proof}
If $\ar(\boldsymbol\rho)<3$ then the statement can be easily checked. Thus, we assume that $\ar(\boldsymbol\rho)\ge 3$.
Let $(a_{3},\ldots,a_n) \in A^{n-2}$,
\begin{multline*}
\sigma(x_{1},x_{2},x_{3},x_{4}) =
\exists y \exists y' \exists z'\;\;
\boldsymbol\rho(x_{1},y,a_3,a_4,\ldots,a_n) \wedge \\
\boldsymbol\rho(x_{2},y',z',a_4,\ldots,a_n) \wedge
\boldsymbol\rho(x_{3},y',a_3,a_4,\ldots,a_n) \wedge
\boldsymbol\rho(x_{4},y,z',a_4,\ldots,a_n).
\end{multline*}

\begin{sublem}

$\sigma$ is a strongly rich relation.

\end{sublem}

\begin{proof}

We know that $\boldsymbol\rho$ is strongly rich.
Therefore, if we pick values for 3 of 4 variables of $\sigma$
, then
the remaining variable can be uniquely calculated using $\boldsymbol\rho$.
For example, if we pick values for $x_{1},x_{2},x_{4}$ then we have a unique value for $y$, then we have a unique value for $z'$,
then a unique value for $y'$, and therefore a unique value for $x_3$.
\end{proof}

\begin{sublem}

Suppose $(x_{1},x_{2},x_{3},x_{4}) \in \sigma$,
then $x_{1}=x_{3} \Leftrightarrow x_2=x_4$
and $x_{1}=x_{4} \Leftrightarrow x_2=x_3.$

\end{sublem}
\begin{proof}
Let us prove the first statement. Assume that $x_1=x_3$.
Since $\boldsymbol\rho$ is strongly rich, $y = y'$, therefore $x_2 = x_4$.
\end{proof}

\begin{sublem}\label{commutativity}
$\sigma(x_1,x_2,x_3,x_4) = \sigma(x_2,x_1,x_3,x_4)$.
\end{sublem}

\begin{proof}

Put $\sigma'(x_1,x_2,x_3,x_4) = \sigma(x_1,x_2,x_3,x_4) \wedge \sigma(x_2,x_1,x_3,x_4)$.
Let us show that $\sigma'=\sigma$.

Put $\sigma_{4}'(x_{1},x_{2},x_{3}) = \exists x_4 \;\sigma' (x_{1},x_{2},x_{3},x_{4})$.
Obviously, if $x_1=x_2$ then for every $x_3$ we can choose $x_{4}$ such that
$(x_{1},x_2,x_3,x_4)\in \sigma$. Therefore, if $x_1=x_2$ then $(x_{1},x_{2},x_{3})\in \sigma_{4}'$.
If $x_1=x_3$ then we can take $x_4=x_2$,
if $x_2=x_3$ then we can take $x_4=x_1$.
Then $\sigma_{4}'$ is totally reflexive, hence by Lemma~\ref{NoTotallyReflexive} $\sigma_4'$ is a full relation.
Therefore, $\sigma'$ contains at least $|A|^{3}$ tuples.
We know that $\sigma$ is strongly rich, which means that $\sigma$ contains exactly $|A|^{3}$ tuples.
Combining this with $\sigma'\subseteq \sigma$ we obtain $\sigma'=\sigma$.
\end{proof}

Let $d\in A$, we say $a\oplus_{d} b = c$ if $
(a,b,c,d)\in \sigma$. Since $\sigma$ is strongly rich, the operation is well-defined.
Let us prove the following properties of the operation $\oplus_d$.
\begin{sublem}\label{plusproperties}

For every $a,b,c,d\in A$
\begin{enumerate}
\item $a\oplus_d b = b\oplus_d a$ (commutativity);
\item $a\oplus_d d = a$, $d\oplus_d a = a$ (neutral element);
\item there exists $e$ such that $a\oplus_d e = d$(inverse element);
\item $(a\oplus_d b)\oplus_d c = c\oplus_d (b\oplus_d c)$ (associativity).

\end{enumerate}

\end{sublem}

\begin{proof}
Commutativity, existence of a neutral element and an inverse element easily
follow from the above properties of $\sigma$. Let us prove associativity.
Put
\begin{multline*}\sigma'(x_1,x_2,x_3,x_4) =
\exists y_1\exists y_2\exists y_3\;\;
\sigma(x_1,x_2,y_1,x_4)\wedge
\sigma(y_1,x_3,y_3,x_4)\wedge\\
\sigma(x_2,x_3,y_2,x_4)\wedge
\sigma(x_1,y_2,y_3,x_4),
\end{multline*}
which means
$\sigma' =\{(a,b,c,d) \mid
(a\oplus_d b)\oplus_d c =   a\oplus_d (b\oplus_d c)\}$.
We consider a tuple $(a,b,c,d)$.
If $a=d$, $b=d$, or $c=d$, then $(a,b,c,d)\in \sigma'$ because of the above properties.
If $a=c$ then it follows from the commutativity.
Let $\sigma''(x_{1},x_{2},x_{3}) = \sigma'(x_1,x_1,x_2,x_3)\wedge \sigma'(x_2,x_1,x_1,x_3)$.
Then $\sigma''$ is totally reflexive,
which, by Lemma~\ref{NoTotallyReflexive}, means that $\sigma''$ is a full relation.
Hence, $\sigma'$ is totally reflexive, which means
that $\sigma'$ is a full relation and $\oplus_d$ is an associative operation.
\end{proof}

Thus, $\oplus_d$ is a group operation for every $d\in A$.

\begin{sublem}\label{plusdefinition}
$\sigma = \{(a,b,c,d)\mid a\oplus_e b = c\oplus_e d\}$ for every $e\in A$.
\end{sublem}

\begin{proof}
First, we want to prove that $(a\oplus_d b)\oplus_b d = a$.
It follows from the definition
that $\sigma(x_1,x_2,x_3,x_4)=\sigma(x_3,x_4,x_1,x_2)$.
Therefore
the formula $\exists x_3\;\sigma(x_1,x_2,x_3,x_4)\wedge \sigma(x_3,x_4,x_1,x_2)$
defines a full relation, which means
that $(x_1\oplus_{x_4} x_2)\oplus_{x_2} x_4 = x_1$.
It remains to put $x_1 = a$, $x_2 = b$, $x_4 = d$.

Second, we prove that $(a\oplus_d b)\oplus_c d = a\oplus_c b$.
Put
\begin{multline*}
\sigma'(x_1,x_2,x_3,x_4) =
\exists y_1\exists y_2\exists y_3\;\;
\sigma(x_1,x_2,y_1,x_4)\wedge\\
\sigma(y_1,x_4,y_3,x_3)\wedge
\sigma(x_1,x_2,y_3,x_3),
\end{multline*}
which means
$\sigma' =\{(a,b,c,d) \mid
(a\oplus_d b)\oplus_c d =   a\oplus_c b\}$.
If $a = d$, $b = d$, or $c=d$ then we can easily check that $(a,b,c,d)\in\sigma'$.
If $a=c$ or $b =c$ then it follows from the previous statement that
$(a,b,c,d)\in\sigma'$.
Then the relation
$\sigma''(x_1,x_3,x_4) = \sigma'(x_1,x_1,x_3,x_4)$ is totally reflexive, and
therefore, by Lemma~\ref{NoTotallyReflexive}, is full.
Hence, $\sigma'$ is also totally reflexive and full,
which proves that $(a\oplus_d b)\oplus_c d = a\oplus_c b$.

We know that
\[\sigma = \{(a,b,c,d) \mid a\oplus_d b = c\} =
\{(a,b,c,d) \mid (a\oplus_d b)\oplus_e d = c\oplus_e d\}.\]
Using the above equation we obtain
$\sigma = \{(a,b,c,d) \mid a\oplus_e b = c\oplus_e d\}.$
\end{proof}

We choose an element from $A$ which we denote by $0$.
We denote the operation $\oplus_0$ by $+$.
Let $\phi_1$ be defined by $\phi_1(x) = x$.

We put
$\phi_2(x) = -y
\Leftrightarrow
(y, x, a_3,\ldots,a_n)\in \boldsymbol\rho$.
We define $\phi_i$ for
$i\ge 3$ as follows.
\[\phi_i(x) = -y
\Leftrightarrow
(y, \phi_{2}^{-1}(0),\ldots,\phi_{i-1}^{-1}(0),x, a_{i+1},\ldots,a_n)\in \boldsymbol\rho.\]
We can prove recursively that $a_{i} = \phi_{i}^{-1}(0)$ for $i\ge 3$.

Now we are ready to prove the statement of the theorem.
We need to prove that for every $b_{1},\ldots,b_{n}\in A$
we have
\[b_{1}+\ldots+b_{n} = 0 \Leftrightarrow (b_1,\phi_{2}^{-1}(b_2),\ldots,\phi_{n}^{-1}(b_n))\in \boldsymbol\rho.\]

Let $m$ be the number of $i\in \{2,3,\ldots,n\}$
such that $b_{i}\neq 0$.
We will prove the statement by induction on $m$.
For $m =0$ and $m=1$ it follows from the definition.

For $m\ge 2$, without loss of generality we assume that $b_{2}\neq 0$ and $b_{n}\neq 0$.
Let
\begin{multline}\label{eq5}
\delta(x_{1},x_{2},x_{3},x_{4}) =
\exists y_2 \dots\exists y_n \exists z\;\;
\boldsymbol\rho(x_{1},y_2,y_3,y_4,\ldots,y_{n-1},y_n) \wedge \\
\boldsymbol\rho(x_{2},z,y_3,y_4,\ldots,y_{n-1},a_n) \wedge
\boldsymbol\rho(x_{3},z,y_3,y_4,\ldots,y_{n-1},y_n) \wedge \\
\boldsymbol\rho(x_{4},y_2,y_3,y_4,\ldots,y_{n-1},a_n)
\end{multline}
We can easily check that
if $(x_{1},x_2,x_3,x_4)\in \delta$
then $x_1=x_3\Leftrightarrow x_{2}=x_{4}$
and $x_1=x_4\Leftrightarrow x_{2}=x_{3}$,
and $(a,b,a,b),$ $(a,b,b,a)\in \delta$ for all $a,b\in A$.
Then by Lemma~\ref{OnlyOneLinearWithWNU}
we get $\delta = \sigma$.

Put $y_i = \phi_{i}^{-1}(b_{i})$ for every $i\in \{2,3,\ldots,n\}$,
and $z = \phi_{2}^{-1}(0)$.
Since $\boldsymbol\rho$ is a strongly rich relation,
there exist unique $x_{1},x_{2},x_{3},x_{4}$ satisfying (\ref{eq5}).
Thus, $(x_{1},x_{2},x_{3},x_{4})\in\delta$,
and therefore $(x_{1},x_{2},x_{3},x_{4})\in \sigma$.
By the inductive assumption we have
\begin{align*}
x_{2} + b_3+&\ldots +b_{n-1} = 0,\\
x_{3} + b_3+&\ldots +b_{n-1} +b_n = 0,\\
x_{4} + b_2+b_3+&\ldots +b_{n-1} = 0.
\end{align*}

By Claim~\ref{plusdefinition}, we have
$x_{1}+x_{2} = x_{3} + x_{4}$ which means
$x_{1} + b_2+\ldots +b_{n} = 0.$

If $b_{1}+\ldots+b_{n}=0$ then $x_{1} = b_{1}$, and
$(b_1,\phi_{2}^{-1}(b_2),\ldots,\phi_{n}^{-1}(b_n))\in \boldsymbol\rho$.
If $(b_1,\phi_{2}^{-1}(b_2),\ldots,\phi_{n}^{-1}(b_n))\in \boldsymbol\rho$
then $b_1 = x_1$ because $\boldsymbol\rho$ is strongly rich.
Therefore $b_{1}+\ldots+b_{n}=0$.
This completes the proof.
\end{proof}

\section{A core of a key relation}

\subsection{Main properties}

A relation $\boldsymbol\sigma\subseteq \boldsymbol\rho$ is called
a core for a key relation $\boldsymbol\rho$ if
there exists a unary vector-function $\Psi$ such that
\begin{enumerate}
\item $\Psi$ preserves $\boldsymbol\rho$;
\item there exists a key tuple $\alpha$ for $\boldsymbol\rho$ such that $\Psi(\alpha) = \alpha$;
\item $\Psi(\boldsymbol\rho) = \boldsymbol\sigma$;
\item $\Psi\circ \Psi = \Psi$;
\item there does not exist a vector-function $\Psi'$ satisfying (1)-(4) such that
$\Psi'(A^n)\subset \Psi(A^n)$.

\end{enumerate}
We say that $\Psi$ is a \emph{restricting vector-function}.

\begin{lem}\label{CoreExistence}
For every key tuple $\alpha$ in a relation $\boldsymbol\rho$
there exists a core $\boldsymbol\sigma$ and a restricting vector-function $\Psi$ satisfying $\Psi(\alpha) = \alpha$.
\end{lem}

\begin{proof}

We consider the set $\Psi(A^n)$
for every vector-function $\Psi$ satisfying properties (1),(2) and (4).
We choose a minimal set among these sets and choose the corresponding vector-function $\Psi$
as a restricting vector-function.
Since identity satisfies the above properties
we always can find a restricting vector-function.
It remains to put $\boldsymbol\sigma = \Psi(\boldsymbol\rho)$.
\end{proof}

Different properties of a core are summarized in the following lemma.

\begin{lem}\label{CoreProperties}

Suppose $\boldsymbol\sigma$ is a core of a key relation $\boldsymbol\rho$,
$\Psi$ is a restricting vector-function.
Then we have the following properties.
\begin{enumerate}
\item $\boldsymbol\sigma$ is a key relation;
\item the patterns of $\boldsymbol\sigma$ and $\boldsymbol\rho$ are equal;
\item $\boldsymbol\sigma$ is a core of $\boldsymbol\sigma$;
\item suppose $\alpha\in \Psi(A^n)$, then $\alpha$ is a key tuple for $\boldsymbol\rho$ if and only if
$\alpha$ is a key tuple for $\boldsymbol\sigma$.
\item
suppose $\alpha_{1}$ is a key tuple for $\boldsymbol\sigma$, $\alpha_{2}\notin\boldsymbol\sigma$, then
every vector-function $\Omega$ preserving $\boldsymbol\sigma$ and satisfying $\Omega(\alpha_{1}) = \alpha_{2}$ is a bijective mapping on $\proj\boldsymbol\sigma$;
\item
suppose a vector-function $\Phi$ preserves $\boldsymbol\sigma$, $\beta\in A^{n}\setminus \boldsymbol\sigma$, $\beta$ is not a key tuple for $\boldsymbol\sigma$,
but $\Phi(\beta)$ is a key tuple for $\boldsymbol\sigma$,
then $\Phi(\alpha) \in \boldsymbol\sigma$ for every key tuple $\alpha$ for $\boldsymbol\sigma$.
\end{enumerate}

\end{lem}

\begin{proof}

First, let us prove statement (4).
Suppose $\alpha$ is a key tuple for $\boldsymbol\rho$ which is in $\Psi(\alpha)$ (so $\Psi(\alpha) = \alpha$).
Let us prove that $\alpha$ is a key tuple for $\boldsymbol\sigma$.
Suppose $\gamma\in A^{n}\setminus \boldsymbol\sigma$.
We need to find an appropriate vector-function
which maps $\gamma$ to $\alpha$.
If $\gamma\notin\boldsymbol\rho$, then by definition
there exists a vector-function $\Omega$ which preserves $\boldsymbol\rho$ and maps $\gamma$ to $\alpha$.
Obviously $\Psi\circ\Omega$ preserves $\boldsymbol\sigma$ and maps $\gamma$ to $\alpha$.
Assume that $\gamma \in \boldsymbol\rho\setminus \boldsymbol\sigma,$
therefore
$\gamma \notin \Psi(A^{n})$.
Without loss of generality, we assume that
$\gamma(1)\notin \Psi^{(1)}(A)$.
Since $\boldsymbol\rho$ is essential, $\alpha $ is an essential tuple for $\boldsymbol\rho$
and there exists $a$ such that
$(a,\alpha(2),\ldots,\alpha(n))\in \boldsymbol\rho$.
Let us define $\Omega$ in the following way.
$\Omega^{(i)}(x) = \alpha(i)$ for every $i\in \{2,3,\ldots,n\}$,
$\Omega^{(1)}(x) =
\begin{cases}
\alpha(1), & \text{if $x = \gamma(1)$,}\\
\Psi^{(1)}(a), & \text{if $x \neq \gamma(1)$.}
\end{cases}.$
We can check that $\Omega$ preserves $\boldsymbol\sigma$ and maps $\gamma$ to $\alpha$.

Suppose $\alpha$ is a key tuple for $\boldsymbol\sigma$. Let us prove that $\alpha$ is a key tuple for $\boldsymbol\rho$.
By the definition of a restricting vector-function we have a key tuple for $\boldsymbol\rho$ such that $\Psi(\beta) = \beta$.
Using above argument we prove that $\beta$ is a key tuple for $\boldsymbol\sigma$.
Suppose $\gamma \notin \boldsymbol\rho$, then there exists a vector-function $\Omega$ that preserves
$\boldsymbol\rho$ and maps $\gamma$ to $\beta$.
Since $\alpha$ is a key tuple for $\boldsymbol\sigma$,
there exists a vector-function $\Omega_{0}$ that preserves $\boldsymbol\sigma$ and maps $\beta$ to $\alpha$.
Obviously, $\Omega_{0}\circ \Psi\circ\Omega$ maps $\gamma$ to $\alpha$ and preserves $\boldsymbol\rho$. Hence,
$\alpha$ is a key tuple for $\boldsymbol\rho$.

Statement (1) follows from statement (4).

Let us prove statement (2).
Suppose $\left(\begin{smallmatrix}a_1\\\vdots\\a_n\end{smallmatrix}\right)$ is a key tuple for $\boldsymbol\sigma$, and therefore it is a key tuple for $\boldsymbol\rho$.
Combining
Lemma~\ref{patternAndKeyTuple}
with the fact that $\boldsymbol\sigma\subseteq \boldsymbol\rho$
we obtain that
$i\overset{\boldsymbol\sigma}{\not\sim} j$ implies $i\overset{\boldsymbol\rho}{\not\sim} j$. 
Assume that $i\overset{\boldsymbol\rho}{\not\sim} j$.
By Lemma~\ref{patternAndKeyTuple} there exist $b_i,b_j\in A$ such that
$\left(\begin{smallmatrix}
a_1\\\dots\\a_{i-1}\\a_i\\a_{i+1}\\\dots\\a_{j-1}\\b_j\\a_{j+1}\\\dots\\a_n
\end{smallmatrix}\right),
\left(\begin{smallmatrix}
a_1\\\dots\\a_{i-1}\\b_i\\a_{i+1}\\\dots\\a_{j-1}\\a_j\\a_{j+1}\\\dots\\a_n
\end{smallmatrix}\right),
\left(\begin{smallmatrix}
a_1\\\dots\\a_{i-1}\\b_i\\a_{i+1}\\\dots\\a_{j-1}\\b_j\\a_{j+1}\\\dots\\a_n
\end{smallmatrix}\right)\in \boldsymbol\rho$.
It remains to apply the restricting function $\Psi$ to these tuples and apply Lemma~\ref{patternAndKeyTuple} to show that $i\overset{\boldsymbol\sigma}{\not\sim} j$.

Let us prove statement (3).
Assume that $\boldsymbol\sigma$ is not a core of $\sigma$, then there exists
a core $\boldsymbol\sigma'$ that can be obtained from $\boldsymbol\sigma$ by applying a restricting vector-function
$\Omega$ that maps  a key tuple $\alpha$ to $\alpha$.
By statement (4) the tuple $\alpha$ is also a key tuple for $\boldsymbol\rho$.
Then the vector-function
$\Psi' = \Omega\circ \Psi$ preserves $\boldsymbol\rho$, maps $\alpha$ to $\alpha$,
and $\Psi'(A^{n})\subset \Psi(A^n)$, which contradicts the definition of a core.

Let us prove statement (5).
Suppose a vector-function $\Omega$ preserves $\boldsymbol\sigma$ and maps
$\alpha_{1}$ to $\alpha_{2}$
Assume that $\Omega(\Psi(A^{n}))\subset \Psi(A^{n})$.
By the definition, there exists a vector-function $\Omega'$ which maps
$\alpha_{2}$ to $\alpha_1$ and preserves $\boldsymbol\sigma$.
Then the vector-function $\Psi' = \Omega'\circ\Omega\circ\Psi$
preserves $\boldsymbol\rho$ and maps $\alpha_{1}$ to $\alpha_{1}$.
We can easily find $t$ such that $\Psi'' = \underbrace{\Psi'\circ \Psi'\circ \dots\circ \Psi'}_{t}$
and $\Psi''\circ\Psi'' = \Psi''$.
Thus we have $\Psi''(A^{n})\subset\Psi(A^n)$
which contradicts the definition of a core.
Hence $\Omega(\Psi(A^{n}))= \Psi(A^{n})$, which means that $\Omega$ is a bijective mapping on $\proj\boldsymbol \sigma=\Psi(A^{n})$.

Let us prove statement (6).
Assume that $\Phi(\alpha)\notin \boldsymbol\sigma$ for a key tuple $\alpha$.
Then by statement (5), $\Phi$ is a bijection on $\proj\boldsymbol\sigma$.
Hence, by Lemma~\ref{KeyGotoRho} $\Phi$ maps key tuples to key tuples,
which contradicts our assumptions.
\end{proof}

\begin{lem}\label{coreispreservedbyWNU}

If a key relation $\boldsymbol\rho$ is preserved by a WNU $\boldsymbol f$,
then a core of $\boldsymbol\rho$ is preserved by a WNU of the same arity.

\end{lem}

\begin{proof}

Let $\boldsymbol\sigma$ be a core of $\boldsymbol\rho$ and $\Psi$ is a restricting vector-function.
Put
$\boldsymbol f'(x_{1},\ldots,x_{m})=\Psi(\boldsymbol f(x_{1},\ldots,x_{m}))$.
It is easy to see that $\boldsymbol f'$ preserves $\boldsymbol\sigma$ and  $\boldsymbol f'$ is a WNU on $\proj\boldsymbol\sigma$.
To complete the proof we define a WNU $\boldsymbol f''$ that coincides with $\boldsymbol f'$ on $\proj\boldsymbol\sigma$.
\end{proof}

\subsection{Pattern of a core}

Suppose $\boldsymbol\rho\subseteq A^{n}$,
$\alpha,\beta\in A^{n}$.
We say that a pair of tuples $(\alpha,\beta)$ \emph{witnesses} $i\overset{\boldsymbol\rho}{\not\sim}j$
if $\alpha$ is a key tuple for $\boldsymbol\rho$,
$\alpha(k) = \beta(k)$ for every $k\notin\{i,j\}$,
and the tuples
$(\alpha(1), \ldots,\alpha(i-1),\beta(i),\alpha(i+1),\ldots,\alpha(n)),$
$(\alpha(1), \ldots,\alpha(j-1),\beta(j),\alpha(j+1),\ldots,\alpha(n))$,
and $\beta$
belong to $\boldsymbol\rho$.


\begin{lem}\label{FirstLemma}

Suppose $\boldsymbol\rho$ is a key relation preserved by a WNU $\boldsymbol f$,
$1\overset{\boldsymbol\rho}{\sim}3$,
the pair
$\left(
\left(\begin{smallmatrix}
a_1\\a_2\\a_{3}\\\vdots\\a_n
\end{smallmatrix}\right),
\left(\begin{smallmatrix}
b_1\\b_2\\a_{3}\\\vdots\\a_n
\end{smallmatrix}\right)\right)$
witnesses $1\overset{\boldsymbol\rho}{\not\sim}2$.
Then
$\left(\begin{smallmatrix}
a_1\\\boldsymbol f^{(2)}(b_2,a_2,\ldots,a_2)\\a_3\\\vdots\\a_n
\end{smallmatrix}\right)\in\boldsymbol\rho$.
\end{lem}

\begin{proof}
Since $\boldsymbol\rho$ is essential, there is $b_{3}\in A$ such that
$(a_1,a_2,b_3,a_{4},\ldots,a_{n})\in \boldsymbol\rho$.
Let $m$ be the arity of $\boldsymbol f$.
For $j\in \{0,1,2,\ldots,m\}$
let \[c_{j} = \boldsymbol f^{(1)}(\underbrace{b_1,\ldots,b_1}_j,a_1,\ldots,a_1),\;\;
d_{j} = \boldsymbol f^{(3)}(\underbrace{a_3,\ldots,a_3}_j,b_3,\ldots,b_3).\]
Denote $b_{2}' = \boldsymbol f^{(2)}(b_2,a_2,\ldots,a_2)$,
$\beta_{j} = (c_{j}, b_{2}', a_{3},\ldots, a_{n})$.

Let us show by induction that
for every $j\in \{0,1, 2,\ldots,m-1\}$ we have
$\beta_j\in \boldsymbol\rho$.
Since $\boldsymbol f^{(2)}(b_2,a_2,\ldots,a_2)=\boldsymbol f^{(2)}(a_2,\ldots,a_2,b_2)$, we have $\beta_{m-1} \in \boldsymbol\rho$.

Since $\overset{\boldsymbol\rho}{1\sim 3}$ we have
\[
\left(\begin{smallmatrix}c_{j}\\b_{2}'\\d_{j}\\a_{4}\\\vdots\\a_n
\end{smallmatrix}\right),
\left(\begin{smallmatrix}c_{j-1}\\b_{2}'\\d_{j}\\a_{4}\\\vdots\\a_n
\end{smallmatrix}\right),
\left(\begin{smallmatrix}c_{j}\\b_{2}'\\a_3\\a_{4}\\\vdots\\a_n
\end{smallmatrix}\right)
\in \boldsymbol\rho\Rightarrow
\left(\begin{smallmatrix}c_{j-1}\\b_{2}'\\a_3\\a_{4}\\\vdots\\a_n
\end{smallmatrix}\right)\in \boldsymbol\rho.
\]
We can check that the first two tuples in the above formula always belong to $\boldsymbol\rho$.
Hence, $\beta_{j}\in \boldsymbol\rho \Rightarrow \beta_{j-1}\in \boldsymbol\rho$.
By induction we get $\beta_{0}\in \boldsymbol\rho$. This completes the proof.
\end{proof}

\begin{lem}\label{Keytupleb1b2}

Suppose $\boldsymbol\rho$ is a core,
$(b_{1},a_{2},\ldots,a_{n}),(a_{1},b_{2},a_{3},\ldots,a_{n})\in \boldsymbol\rho$,
$(a_{1},\ldots,a_{n})$ is a key tuple for $\boldsymbol\rho$,
$1\overset{\boldsymbol\rho}{\sim}2$.
Then $(b_1,b_2,a_{3},\ldots,a_n)$
is a key tuple for $\boldsymbol\rho$.

\end{lem}

\begin{proof}

Assume that $(b_1,b_2,a_{3},\ldots,a_n)$ is not a key tuple for $\boldsymbol\rho$.
Since $\alpha = (a_{1},\ldots,a_{n})$ is a key tuple for $\boldsymbol\rho$ there exists a vector-function
$\Psi$ preserving $\boldsymbol\rho$ which maps $(b_{1},b_{2},a_{3},\ldots,a_{n})$ to $(a_{1},\ldots,a_{n})$.
By Lemma~\ref{CoreProperties} (6) we have
$\Psi(\alpha)\in \boldsymbol\rho$, then
$\left(\left(\begin{smallmatrix}
a_1\\a_2\\a_{3}\\\vdots\\a_n
\end{smallmatrix}\right),\left(\begin{smallmatrix}
\Psi^{(1)}(a_1)\\\Psi^{(2)}(a_2)\\a_{3}\\\vdots\\a_n
\end{smallmatrix}\right)\right)$ witnesses $1\overset{\boldsymbol\rho}{\not \sim}2$. Contradiction.
\end{proof}

\begin{lem}\label{KeyBijections}

Suppose $\boldsymbol\rho\subseteq A^{n}$ is a core preserved by a WNU $\boldsymbol f$, $p\in \{3, \ldots, n\}$,
$1\overset{\boldsymbol\rho}{\not\sim} 2$, $1\overset{\boldsymbol\rho}{\sim}  3$, $2\overset{\boldsymbol\rho}{\sim}  p$.
Then for every key tuple $(a_{1},\ldots, a_{n})$ for $\boldsymbol\rho$
and every $i\in \{1,\ldots,n\}$
the mapping $g(x) = \boldsymbol f^{(i)}(a_i,a_i,\ldots,a_i,x)$ is a bijection on $\proj_{i} \boldsymbol\rho$.

\end{lem}

\begin{proof}

Let
$\left(\left(\begin{smallmatrix}
a_1\\a_2\\a_{3}\\\vdots\\a_n
\end{smallmatrix}\right),\left(\begin{smallmatrix}
b_1\\b_2\\a_{3}\\\vdots\\a_n
\end{smallmatrix}\right)\right)$
be a pair that witnesses $1\overset{\boldsymbol\rho}{\not\sim}2$,
let
$\left(\begin{smallmatrix}a_1\\a_2\\b_3\\a_4\\\vdots\\a_n
\end{smallmatrix}\right)\in \boldsymbol\rho$.
By Lemma \ref{FirstLemma} we have
$\left(\begin{smallmatrix}
\boldsymbol f^{(1)}(b_1,a_1\ldots,a_1)\\
a_2\\
\vdots\\a_n
\end{smallmatrix}\right)\in\boldsymbol\rho$.
Since
$
\left(\begin{smallmatrix}a_1\\a_2\\b_3\\a_4\\\vdots\\a_n
\end{smallmatrix}\right)\in \boldsymbol\rho,
\left(\begin{smallmatrix}a_1\\a_2\\a_3\\a_4\\\vdots\\a_n
\end{smallmatrix}\right)\not\in \boldsymbol\rho
$, and
$1\overset{\boldsymbol\rho}{\sim} 3$, we get
$\left(\begin{smallmatrix}\boldsymbol f^{(1)}(b_1,a_1\ldots,a_1)\\a_2\\b_3\\a_4\\\dots\\a_n
\end{smallmatrix}\right)\notin \boldsymbol\rho$.
Put
$\boldsymbol h\left(\begin{smallmatrix}
x_1\\x_2\\\vdots\\x_n
\end{smallmatrix}\right)
=
\boldsymbol f\left(\begin{smallmatrix}
a_1 & \dots & a_1 & x_1\\
a_2 & \dots & a_2 & x_2\\
b_3 & \dots & b_3 & x_3\\
a_4 & \dots & a_4 & x_4\\
\vdots & \ddots & \vdots & \vdots\\
a_n & \dots & a_n & x_n
\end{smallmatrix}\right).$
Obviously, $\boldsymbol h$ preserves $\boldsymbol\rho$.
It follows from Lemma~\ref{Keytupleb1b2} that the tuple
$\left(\begin{smallmatrix}b_1\\a_2\\b_3\\a_4\\\dots\\a_n\end{smallmatrix}\right)$
is a key tuple for $\boldsymbol\rho$.
Since $\boldsymbol h\left(\begin{smallmatrix}
b_1\\a_2\\b_3\\a_4\\\dots\\a_n
\end{smallmatrix}\right) = \left(\begin{smallmatrix}\boldsymbol f^{(1)}(b_1,a_1,\ldots,a_1)\\a_2\\b_3\\a_4\\\dots\\a_n
\end{smallmatrix}\right)\notin \boldsymbol\rho$,
by Lemma~\ref{CoreProperties} (5) the mapping $\boldsymbol h$ is a bijection on $\proj \boldsymbol\rho$.
So, we have $\boldsymbol h^{(i)}(x) = \boldsymbol f^{(i)}(a_i,a_i,\ldots,a_i,x)$
for $i\neq 3$,
therefore $\boldsymbol f^{(i)}(a_i,a_i,\ldots,a_i,x)$ is a bijection on $\proj_{i} \boldsymbol\rho$.

It remains to prove the statement for $i = 3$.
In this case we choose
$(b_{1},a_{2},b_{3},a_{4},\ldots,a_{n})$ as a key tuple
and $(b_{1},a_{2},a_{3},a_{4},\ldots,a_{n})$ as a tuple from $\boldsymbol\rho$,
and repeat the whole proof.
As a result we prove that $\boldsymbol f^{(3)}(a_3,a_3,\ldots,a_3,x)$ is a bijection on $\proj_{3} \boldsymbol\rho$.
\end{proof}

\begin{lem}\label{MainPatternLemma}

Suppose $\boldsymbol\rho\subseteq A^{n}$ is a core preserved by a WNU $\boldsymbol f$,
$1\overset{\boldsymbol\rho}{\not\sim} 2$, $1\overset{\boldsymbol\rho}{\sim}  3$.
Then $2\overset{\boldsymbol\rho}{\not\sim} p$ for any $p\in \{3,\ldots,n\}$.

\end{lem}

\begin{proof}
Assume that
 $2\overset{\boldsymbol\rho}{\sim} p$ for $p\in \{3,\ldots,n\}$. Let
$\left(\left(\begin{smallmatrix}
a_1\\a_2\\a_{3}\\\vdots\\a_n
\end{smallmatrix}\right),\left(\begin{smallmatrix}
b_1\\b_2\\a_{3}\\\vdots\\a_n
\end{smallmatrix}\right)\right)$
be a pair that witnesses $1\overset{\boldsymbol\rho}{\not\sim}2$,
let
$\left(\begin{smallmatrix}a_1\\a_2\\b_3\\a_4\\\vdots\\a_n
\end{smallmatrix}\right)\in \boldsymbol\rho$.
Since $\boldsymbol f$ preserves $\boldsymbol\rho$,
we have
\[
\left(\begin{smallmatrix}a_1\\\boldsymbol f^{(2)}(b_2,\ldots,b_2,a_2)\\\boldsymbol f^{(3)}(a_3,\ldots,a_3,b_3)\\a_{4}\\\vdots\\a_{n}
\end{smallmatrix}\right),
\left(\begin{smallmatrix}\boldsymbol f^{(1)}(b_1,a_1,\ldots,a_1)\\\boldsymbol f^{(2)}(b_2,\ldots,b_2,a_2)\\a_3\\a_{4}\\\vdots\\a_{n}
\end{smallmatrix}\right),
\left(\begin{smallmatrix}\boldsymbol f^{(1)}(b_1,a_1,\ldots,a_1)\\\boldsymbol f^{(2)}(b_2,\ldots,b_2,a_2)\\\boldsymbol f^{(3)}(a_3,\ldots,a_3,b_3)\\a_{4}\\\vdots\\a_{n}
\end{smallmatrix}\right)\in \boldsymbol\rho.\]
$1\overset{\boldsymbol\rho}{\sim} 3$, therefore
$\left(\begin{smallmatrix}a_1\\\boldsymbol f^{(2)}(b_2,\ldots,b_2,a_2)\\a_{3}\\\vdots\\a_{n}
\end{smallmatrix}\right)\in \boldsymbol\rho$.
Put
$\boldsymbol h\left(\begin{smallmatrix}
x_1\\x_2\\\vdots\\x_n
\end{smallmatrix}\right)
=
\boldsymbol f\left(\begin{smallmatrix}
a_1 & \dots & a_1 & x_1\\
b_2 & \dots & b_2 & x_2\\
a_3 & \dots & a_3 & x_3\\
\vdots & \ddots &\vdots & \vdots \\
a_n & \dots & a_n & x_n
\end{smallmatrix}\right).$
By Lemma \ref{Keytupleb1b2}, $b_2$ occurs in a key tuple for $\boldsymbol\rho$.
Then by Lemma \ref{KeyBijections},
$\boldsymbol h$ is a bijection on $\proj \boldsymbol\rho$.
But $\boldsymbol h$ preserves $\boldsymbol\rho$ and maps a tuple $(a_{1},\ldots,a_{n})$, which is not in $\boldsymbol\rho$,
to a tuple from $\boldsymbol\rho$. This contradiction completes the proof.
\end{proof}

The following theorem follows directly from Lemma~\ref{MainPatternLemma}.

\begin{thm}\label{patternOfCoreIsEquiv}

Suppose $\boldsymbol\rho$ is a core preserved by a WNU, then $\overset{\boldsymbol\rho}{\sim}$ is an equivalence relation such that at most one equivalence class contains more than 1 element.

\end{thm}

Combining the above theorem and Lemma~\ref{CoreProperties}(2) we obtain the proof of
Theorem~\ref{PatternIsEquivalenceForWNUF}
from Section~\ref{MainResultsSection}.

\subsection{Core of a key relation with full pattern}

\begin{lem}\label{KeyBlockDescription}

Suppose $\boldsymbol\sigma$ is a core with full pattern.
Then every connected component of $\Key(\boldsymbol\sigma)$ can be represented as
$A_{1}\times A_{2}\times \dots \times A_{n}$ for some
$A_{1},\ldots,A_{n}\subseteq A$.

\end{lem}

\begin{proof}

Let $a_1,\ldots,a_n, b_1,b_2\in A$. The tuples
$\left(\begin{smallmatrix}
a_1\\a_2\\a_3\\\dots\\a_n
\end{smallmatrix}\right)$,
$\left(\begin{smallmatrix}
b_1\\a_2\\a_3\\\dots\\a_n
\end{smallmatrix}\right)$,
$\left(\begin{smallmatrix}
a_1\\b_2\\a_3\\\dots\\a_n
\end{smallmatrix}\right)$,
$\left(\begin{smallmatrix}
b_1\\b_2\\a_3\\\dots\\a_n
\end{smallmatrix}\right)$
we denote by
$\alpha,\beta_1,\beta_2,\beta_3$ correspondingly.
Assume that $\alpha\notin \Key(\boldsymbol\sigma)$
but each of the tuples $\beta_1$, $\beta_2$, $\beta_3$ belongs to $\Key(\boldsymbol\sigma)$.
By Lemma~\ref{CoreProperties}, $\boldsymbol\sigma$ is a key relation, then there exists a
key tuple $\beta$ for $\boldsymbol\sigma$ and a vector-function $\Psi$ preserving $\boldsymbol\sigma$ such that
$\Psi(\alpha) = \beta$.
Then, by the property (6) in Lemma~\ref{CoreProperties} we have
$\Psi(\beta_1), \Psi(\beta_2),\Psi(\beta_3)\in \boldsymbol\sigma$. Therefore
$\overset{\boldsymbol\sigma}{1\not\sim 2},$ which contradicts the
fact that the pattern is full. Hence $\alpha\in \Key(\boldsymbol\sigma)$.

Using the same argument for other coordinates we can show that
every connected component of $\Key(\boldsymbol\sigma)$ can be represented as
$A_{1}\times A_{2}\times \dots \times A_{n}$.
\end{proof}

Connected components of $\Key(\boldsymbol\sigma)$ that contain a key tuple of $\boldsymbol\sigma$ are called \emph{key blocks}.
If a core is preserved by a WNU, then we can get a description of key blocks of a core.

\begin{lem}\label{keyBlockBijection}

Suppose $\boldsymbol \sigma$ is a core preserved by a WNU $\boldsymbol f$,
$k\overset{\boldsymbol\sigma}{\sim}l$, $k\neq l$,
$(a_1,\ldots,a_n)$ is a key tuple for $\boldsymbol\sigma$.
Then $\boldsymbol f^{(i)}(a_{i},\ldots,a_{i},x)$ is a bijection on $\proj_{i} \boldsymbol\sigma$
for every $i\notin\{k,l\}$.

\end{lem}

\begin{proof}

Without loss of generality we assume that $k=1$, $l=2$. Let $m$ be the arity of $\boldsymbol f$.
Since $\boldsymbol\sigma$ is essential,
$(b_{1},a_{2},\ldots,a_{n}), (a_{1},b_{2},a_{3},\ldots,a_{n})\in \boldsymbol\sigma$
for some $b_{1}, b_{2}\in A$.
Put $c_{j} = \boldsymbol f^{(1)}(b_{1}^{j}a_{1}^{m-j})$,
$d_{j} = \boldsymbol f^{(2)}(a_{2}^{m-j}b_{2}^{j})$,
$\alpha = (a_{1},\ldots,a_{n})$,
$\beta_{j} = (c_{j},d_{m-j-1},a_{3},\ldots,a_{n})$,
$\gamma_{j} = (c_{j},d_{m-j},a_{3},\ldots,a_{n})$.

Assume that $\beta_{j}\in \boldsymbol \sigma$ for every $j\in\{0,1,\ldots,m-1\}$.
Then we have the following path in $\boldsymbol\sigma$:
$\gamma_{0}-\beta_{0}-\gamma_{1}-\beta_{1}-\dots-\gamma_{m-1}-\beta_{m-1}-\gamma_{m}$.
Since $1\overset{\boldsymbol\sigma}{\sim}2$, we can easily derive that
$(c_{0}, d_{0}, a_{3},\ldots,a_{n}) = \alpha\in \boldsymbol \sigma.$
Contradiction.

Assume that $\beta_{j}\notin \boldsymbol \sigma$ for some $j\in\{0,1,\ldots,m-1\}$.
Then define
$\boldsymbol h\left(\begin{smallmatrix}
x_1\\x_2\\x_{3}\\\dots\\x_n
\end{smallmatrix}\right)
=
\boldsymbol f\left(\begin{smallmatrix}
b_1^{j} x_{1} a_{1}^{m-j-1} \\
a_2^{j} x_{2}b_{2}^{m-j-1}\\
a_{3}^{j}x_{3}a_{3}^{m-1-j}\\
\vdots \\
a_{n}^{j}x_{n}a_{n}^{m-1-j}
\end{smallmatrix}\right).$
We can check that
$\boldsymbol h(\alpha) = \beta_{j}$ and
$\boldsymbol h$ preserves $\boldsymbol \sigma$.
By Lemma~\ref{CoreProperties} (5), $\boldsymbol h$ is a bijection on $\proj\boldsymbol \sigma$.
Hence $\boldsymbol h^{(i)}(x) = \boldsymbol f^{(i)}(a_{i},\ldots,a_{i},x)$
is a bijection on $\proj_{i}\boldsymbol\sigma$ for every $i\in\{3,4,\ldots,n\}$.
\end{proof}

\begin{lem}\label{KeyBlockSpread}

Suppose $\boldsymbol\sigma$ is a core with full pattern of arity $n\ge 3$ preserved by a WNU $\boldsymbol f$,
$a_1\alpha$ and $b_1\alpha$ belong to a key block.
Then for every $\beta \in A^{n-1}$ we have
either $a_1\beta, b_1\beta\in \Key(\boldsymbol\sigma)$, or
$a_1\beta, b_1\beta\notin \Key(\boldsymbol\sigma)$.

\end{lem}

\begin{proof}

Let $\alpha=(a_2,\ldots,a_n)$.
Assume that $b_1\beta\in \Key(\boldsymbol\sigma)$, let us show that $a_1\beta\in \Key(\boldsymbol\sigma)$.
Put $\boldsymbol h\left(\begin{smallmatrix}x_1\\x_2\\x_3\\\dots\\x_n\end{smallmatrix}\right)
:=\left(\begin{smallmatrix}\boldsymbol f^{(1)}(b_1,\ldots, b_1,x_1)\\
\boldsymbol f^{(2)}(a_2,\ldots, a_2,x_2)\\
\boldsymbol f^{(3)}(a_3,\ldots, a_3,x_3)\\
\dots
\\
\boldsymbol f^{(n)}(a_n,\ldots, a_n,x_n)
\end{smallmatrix}\right) $.
Since $\boldsymbol \sigma$ has full pattern, for every tuple $\delta$ from a key block and every $j\in\{1,2,\ldots,n\}$
there exists a key tuple $\gamma$ from the key block such that $\gamma(j) = \delta(j)$.
Hence, by Lemma~\ref{keyBlockBijection}, $\boldsymbol h$ is a bijection on $\proj \boldsymbol\sigma$.

By Lemma~\ref{DeriveKey}, the vector-function $\boldsymbol h$ preserves $\Key(\boldsymbol\sigma)$,
therefore
$\boldsymbol h(a_1\beta) =\left(\begin{smallmatrix}\boldsymbol f^{(1)}(a_1,b_{1},\ldots,b_1)\\
\boldsymbol f^{(2)}(a_2,\ldots, a_2,\beta(1))\\
\boldsymbol f^{(3)}(a_3,\ldots, a_3,\beta(2))\\
\dots
\\
\boldsymbol f^{(n)}(a_n,\ldots, a_n,\beta(n-1))
\end{smallmatrix}\right)\in\Key(\boldsymbol\sigma)$.
Since $\boldsymbol h$ is a bijection on $\proj \boldsymbol \sigma$ preserving $\Key(\boldsymbol\sigma)$,
we obtain $a_1\beta\in \Key(\boldsymbol\sigma)$.
\end{proof}

\begin{cor}\label{KeyBlockIsCubic}
Suppose $\boldsymbol\sigma$ is a core of arity greater than 2 with full pattern preserved by a WNU.
Then every key block of $\boldsymbol\sigma$ can be represented as
$A_{1}\times\dots\times A_{n}$ for some $A_1,\ldots,A_n\subseteq A$.
\end{cor}

Using the same argument as in Lemma~\ref{KeyBlockSpread} we can prove the following lemma
(we just need change $\Key(\boldsymbol\sigma)$ to $\boldsymbol\sigma$ everywhere in the proof).

\begin{lem}\label{KeyBlockSpread2}

Suppose $\boldsymbol\sigma$ is a core with full pattern of arity $n\ge 3$ preserved by a WNU $\boldsymbol f$,
$a_1\alpha,$ $b_1\alpha$ belong to a key block,
and $a_1\alpha, b_1\alpha\in \boldsymbol \sigma$.
Then for every $\beta \in A^{n-1}$ we have
either $a_1\beta, b_1\beta\in \boldsymbol\sigma$, or
$a_1\beta, b_1\beta\notin \boldsymbol\sigma$.

\end{lem}

\begin{lem}\label{OrderOfElements}
Suppose $(G;+)$ is a finite abelian group, $n\ge 3$,
and the relation $\boldsymbol\rho\subseteq G^{n}$ is defined by
$\boldsymbol\rho = \{(a_{1},\ldots,a_{n})\mid a_{1}+\dots+a_{n}=0\}$,
a vector-function $\Psi$ preserves $\boldsymbol \rho$. 
Then for every $a,b\in G$ and every $i\in\{1,2,\ldots,n\}$ the order of $\Psi^{(i)}(a)-\Psi^{(i)}(b)$ divides the order of $a-b$.
\end{lem}

\begin{proof}
Without loss of generality we assume that $i=1$.
Let
\begin{multline*}\sigma(x_1,x_2,x_3,x_4) = \exists y_2' \exists y_3' \exists y_2 \exists y_3\dots \exists y_n \;
\boldsymbol \rho (x_1,y_2,y_3,y_4\ldots,y_n)\wedge\\
\rho (x_2,y_2',y_3',y_4,\ldots,y_n)\wedge
\rho (x_3,y_2,y_3',y_4,\ldots,y_n)\wedge
\rho (x_4,y_2',y_3,y_4,\ldots,y_n).
\end{multline*}
We can easily check that $(x_1,x_2,x_3,x_4)\in \sigma\Leftrightarrow (x_1+x_2=x_3+x_4)$.
Let $\sigma_{2}(x,y,z) = \sigma(x,x,y,z)$,
$\sigma_{j+1}(x,y,z) = \exists y' \;\sigma_j(x,y',z) \wedge \sigma(x,y',y,z)$.
We can check that
$(x,y,z)\in \sigma_{j}\Leftrightarrow j\cdot x = y + (j-1)\cdot z$.
Put $\delta_{j}(x,y) = \sigma_{j}(x,y,y)$, then $(x,y) \in \delta_{j}$ means that $j\cdot (x-y) = 0$.
We know that $\Psi^{(1)}$ preserves $\delta_{j}$ for every $j$.
Assume that the order of $(a-b)$ equals $k$, then $(a,b)\in \delta_{k}$.
Hence $(\Psi^{(1)}(a),\Psi^{(1)}(b))\in\delta_{k}$ and the order of $\Psi^{(1)}(a)-\Psi^{(1)}(b)$ divides $k$.
\end{proof}

\begin{lem}\label{NoMappingBetweenDifferentGroups}
Suppose $(G_1;+)$ and $(G_2;+)$ are finite abelian groups, $n\ge 3$,
the relations $\boldsymbol\rho_1\subseteq G_1^{n}$, $\boldsymbol\rho_2\subseteq G_2^{n}$ are defined by
\[
\boldsymbol\rho_1 = \{(a_{1},\ldots,a_{n})\mid a_{1}+\dots+a_{n}=0\},\;
\boldsymbol\rho_2 = \{(a_{1},\ldots,a_{n})\mid a_{1}+\dots+a_{n}=0\};
\]
there exists a vector-function
$\Psi = (\psi_1,\ldots,\psi_n)$,
where $\psi_{i}\colon G_1\to G_2$,
such that $\Psi(\boldsymbol\rho_{1})\subseteq \boldsymbol\rho_2$ and $\Psi((c_1,0,\ldots,0))=(c_2,0,\ldots,0)$.
Then the order of $c_2$ divides the order of $c_1$.
\end{lem}

\begin{proof}
Consider a group $(G_1\times G_2;+)$
and the relation $\boldsymbol\rho\subseteq (G_1\times G_2)^n$ defined by
\[\boldsymbol\rho= \{((a_{1},b_1),\ldots,(a_{n},b_n))\mid a_{1}+\dots+a_{n}=0\wedge b_{1}+\dots+b_{n}=0\}.\]
Define a vector-function $\Omega$ by
$\Omega^{(i)}((a,b)) = (0,\psi_i(a))$.
We can easily check that $\Omega$ preserves $\boldsymbol\rho$
and $\psi_1(0) = 0$.
Then by Lemma~\ref{OrderOfElements}
the order of $(0,c_2)-(0,0)$ divides the order
of $(c_1,0)-(0,0)$, which means that the order of $c_2$ divides the order of $c_1$.
\end{proof}

\begin{lem}\label{OnlyPrimeOrder}

Suppose $(G;+)$ is a finite abelian group, $n\ge 3$,
and the relation $\boldsymbol\rho\subseteq G^{n}$ is defined by
$\boldsymbol\rho = \{(a_{1},\ldots,a_{n})\mid a_{1}+\dots+a_{n}=0\}$,
$\Key(\boldsymbol\rho)= G^{n}$.
Then the order of every element in $G$ is a prime number.

\end{lem}

\begin{proof}
Assume the converse.
Let $a$ be the element of $G$ whose order is not prime.
Let the order of this element be equal to $k_1\cdot k_2$ for $k_1,k_2>1$.
Put $b = k_1 \cdot a$, and consider a vector-function $\Psi$ preserving $\boldsymbol \rho$
that maps $(b,0,0,\ldots,0)$ to $(a,0,\ldots,0)$.
Obviously $\Psi$ maps $(0,0,\ldots,0)$ to $(0,0,\ldots,0)$.
Applying Lemma~\ref{OrderOfElements} for the first component we obtain that
the order of $(a-0)$ divides the order of $(b-0)$. Contradiction.
\end{proof}

\begin{lem}\label{OnlyPowerOfPrime}

Suppose $(G;+)$ is a finite abelian group, $n\ge 3$,
and a key relation $\boldsymbol\rho\subseteq G^{n}$ is defined by
$\boldsymbol\rho = \{(a_{1},\ldots,a_{n})\mid a_{1}+\dots+a_{n}=0\}$.
Then the order of the group $G$ is a power of a prime number.
\end{lem}

\begin{proof}
Let $(a_1,\ldots,a_n)$ be a key tuple of $\rho$.
Obviously, $(b_1,a_2,\ldots,a_n)\in\boldsymbol\rho$
for $b_1 = -a_2-a_3-\ldots-a_n$.
Let the order of $(b_1-a_1)$ be equal to $k$.

Assume the converse. Then there exists an element $c\in G\setminus\{0\}$ such that
$k$ doesn't divide the order of $c$.
By Lemma~\ref{OrderOfElements} we cannot map
$(c,0,0\ldots,0)$ to $(a_1,\ldots,a_n)$ by a vector-function preserving $\boldsymbol \rho$.
This contradicts the fact that $(a_1,\ldots,a_n)$ is a key tuple.
\end{proof}

\begin{thm}\label{MainForFullPattern}

Suppose $\boldsymbol\sigma$ is a core with full pattern preserved by a WNU, $n\ge 3$,
$A_{1}\times A_{2}\times \dots \times A_{n}$ is a key block for $\boldsymbol\sigma$.
Then there exists a finite field $F$, and bijective mappings
$\phi_i: A_i\to F$ for $i=1,2,\ldots,n$ such that
\[\boldsymbol\sigma\cap (A_{1}\times A_{2}\times \dots \times A_{n}) =\\  \{(x_1,\dots,x_n) \mid \phi_1(x_1) + \phi_2(x_2) + \dots +\phi_n(x_n)=0\}.\]

\end{thm}

\begin{proof}

Assume that we have two different tuples $a_1\alpha, b_1\alpha\in \boldsymbol\sigma$
and they belong to the key block $A_{1}\times \dots \times A_{n}$.
Let $\Psi$ be a unary vector-function
such that $\Psi(A^{n}) = \proj\boldsymbol\sigma$
and $\Psi(x) = x$ if $x\in \proj\boldsymbol\sigma$.
Obviously, $\Psi$ preserves $\boldsymbol\sigma$.
Let us define $\Psi'$ as follows.
$\Psi'^{(i)} = \Psi^{(i)}$ for $i\in \{2,3,\ldots,n\}$.
$\Psi'^{(1)}(x) = \begin{cases}
a_{1}, & \text{if $x \in \{a_{1},b_{1}\}$}\\
\Psi^{(1)}(x), & \text{otherwise.}
\end{cases}$.
It follows from Lemma~\ref{KeyBlockSpread2} that
$a_{1}\beta\in \boldsymbol\sigma \Leftrightarrow b_{1}\beta\in \boldsymbol\sigma$ for every $\beta\in A^{n-1}$.
Hence, $\Psi'$ preserves $\boldsymbol\sigma$
and $\Psi'(A^{n})\subset \proj\boldsymbol\sigma$, which contradicts the fact that $\boldsymbol\sigma$ is a core.

Hence, if $a_1\alpha\in \boldsymbol\sigma$ belongs to a key block,
then for every $b_1\neq a_{1}$ we have $b_1\alpha\notin \boldsymbol\sigma$.
We can use the same argument not only for the first coordinate.

Suppose $(a_{1},\ldots,a_{n})\in A_{1}\times \dots \times A_{n}$.
We know that the tuple $(a_{1},\ldots,a_{n})$ is either a key tuple, or belongs to $\boldsymbol\sigma$,
and we know that every key tuple is an essential tuple.
This means that for every
$i\in \{1,2,\ldots,n\}$ there exists $b$ such that
$(a_{1},\ldots,a_{i-1},b,a_{i+1},\ldots,a_{n})\in \boldsymbol\sigma.$
By the previous argument, we can state that
for every
$i\in \{1,2,\ldots,n\}$ there exists a unique $b$ such that
$(a_{1},\ldots,a_{i-1},b,a_{i+1},\ldots,a_{n})\in \boldsymbol\sigma.$
Then we can consider this key block as strongly rich relation.
Combining Lemma~\ref{DeriveKey} and Lemma~\ref{preserveConnectedComponent},
we can show that the WNU preserves the key block.
Then, by Theorem~\ref{StronglyRichRelationTHM}
there exist a finite abelian group $(G;+)$
and bijective mappings
$\phi_{i}:A_{i}\to G$ for $i=1,2,\ldots,n$ such that
\[\boldsymbol\sigma\cap (A_{1}\times A_{2}\times \dots \times A_{n}) =\\
\{(x_1,\dots,x_n) \mid \phi_1(x_1) + \phi_2(x_2) + \dots +\phi_n(x_n)=0\}.\]
By Lemma~\ref{OnlyPrimeOrder}, the order of every element in the group $G$ is a prime number,
hence, we may consider a finite field instead of group.
\end{proof}

We prefer to consider a finite field instead of group in Theorem~\ref{MainForFullPattern} because
the multiplication in $F$ can be used in the following way.

\begin{lem}
Suppose $F$ is a finite field,
the relation $\boldsymbol \rho\subseteq F^{n}$ is defined
by $\boldsymbol\rho = \{(x_{1},\ldots,x_{n})\mid x_{1}+\dots+x_{n}=0\}$.
Then $\boldsymbol \rho$ is a key relation and
$\Key(\rho) = F^{n}$.
\end{lem}

\begin{proof}

Let us show that all tuples that are not from the relation are key tuples.
For every $\alpha,\beta\in F^{n}\setminus \boldsymbol\rho$ we need to find a vector-function $\Psi$ preserving
$\boldsymbol\rho$ such that $\Psi(\alpha) = \beta$.
Let $a = \alpha(1)+\dots+\alpha(n)$ and
$b = \beta(1)+\dots+\beta(n)$.
We know that $a\neq 0$ and $b\neq 0$, then there exists $d\in F$ such that
$a\cdot d = b$.
Define $\Psi$ as follows
$\Psi^{(i)}(x) = d\cdot x + \beta(i) - d\cdot \alpha(i)$
for $i\in\{1,2,\ldots,n\}$. Obviously, $\Psi(\alpha) = \beta$.
It remains to check that $\Psi$ preserves $\boldsymbol\rho$, which can be easily done.
\end{proof}

There is a conjecture that Theorem~\ref{MainForFullPattern} can be strengthened as follows.

\begin{conj}

Suppose $\boldsymbol\sigma$ is a core with full pattern preserved by a WNU, $n\ge 3$,
$A_{1}\times A_{2}\times \dots \times A_{n}$ is a key block for $\boldsymbol\sigma$.
Then there exist a prime number $p$, and bijective mappings
$\phi_i: A_i\to \mathbb Z_{p}$ for $i=1,2,\ldots,n$ such that
\[\boldsymbol\sigma\cap (A_{1}\times A_{2}\times \dots \times A_{n}) =  \{(x_1,\dots,x_n) \mid \phi_1(x_1) + \phi_2(x_2) + \dots +\phi_n(x_n)=0\}.\]

\end{conj}

\begin{lem}\label{WNUIsLinear}
Suppose $\boldsymbol\sigma$ is a core with full pattern preserved by a WNU $\boldsymbol f$ of arity $m$,
$n\ge 3$,
$A_{1}\times A_{2}\times \dots \times A_{n}$ is a key block for $\boldsymbol\sigma$.
Then for every $j\in\{1,\ldots,n\}$ there exist an abelian group
$(A_{j};+)$ and an integer $t$ such that
for every $a_1,\ldots,a_m \in A_{j}$
we have $\boldsymbol f^{(j)}(a_{1},\ldots,a_{m}) = t\cdot a_{1}+t\cdot a_2 + \ldots + t\cdot a_{m}$.
\end{lem}

\begin{proof}

Put $\boldsymbol\sigma' = \boldsymbol\sigma\cap (A_{1}\times A_2\times\dots \times A_{n})$.
Using Lemma~\ref{DeriveKey} and Lemma~\ref{preserveConnectedComponent}, we obtain that $\boldsymbol f$ preserves the key block,
hence $\boldsymbol f$ preserves $\boldsymbol\sigma'$.
By Theorem~\ref{MainForFullPattern}
there exist a finite field $F$ and bijective mappings
$\phi_i: A_i\to F$ for $i=1,2,\ldots,n$ such that
\[\boldsymbol\sigma'=\{(x_1,\dots,x_n) \mid \phi_1(x_1) + \phi_2(x_2) + \dots +\phi_n(x_n)=0\}.\]
Using the bijection $\phi_j$ we define an abelian group $(A_{j};+)$.
Then, by Lemma~\ref{WNUIsLinearInAnyBlock} there exists an integer $t$ such that
$\boldsymbol f^{(j)}(a_{1},\ldots,a_{m}) = t\cdot a_{1}+ \ldots + t\cdot a_{m}$
for every $a_1,\ldots,a_m \in A_{j}$.
\end{proof}

\section{Proof of the main results}

\subsection{Key relations with trivial pattern}

In this section we prove statements from Section~\ref{MainResultsSection} for a key relation with trivial pattern.
First, we prove the following lemma from that section.

\begin{trivialpatternLemma}

Suppose $\boldsymbol\rho\in R_{A,n}$,
$(a_1,a_2,\ldots,a_n)\notin\boldsymbol\rho$,
$b_{1},\ldots,b_{n}\in A$, 
and
$(\{a_1,b_1\}\times \{a_2,b_2\}\times \dots\times\{a_n,b_n\})\setminus \{(a_1,a_2,\ldots,a_n)\}\subseteq \boldsymbol\rho.$
Then $\boldsymbol\rho$ is a key relation and $(a_1,a_2,\ldots,a_n)$ is a key tuple for the relation $\boldsymbol\rho$.

\end{trivialpatternLemma}

\begin{proof}

To prove the statement, for every $(c_{1},\ldots,c_{n})\notin \boldsymbol\rho$ we need to find
a vector-function $\Psi$ preserving $\boldsymbol\rho$ that maps $(c_{1},\ldots,c_{n})$ to $(a_{1},\ldots,a_{n})$.
Put $\Psi^{(i)}(x) =
\begin{cases}
a_{i}, & \text{if $x = c_{i}$;}\\
b_{i}, & \text{otherwise.}
\end{cases}$
It is easy to see that it satisfies the above properties.
\end{proof}

A pair of tuples $(a_{1},\ldots,a_{n}) - (b_{1},\ldots,b_{n})$ is called \emph{perfect}
if $(a_{1},\ldots,a_{n})\notin\boldsymbol\rho$,
$b_{i}\neq a_{i}$ for every $i$, and
$(\{a_1,b_1\}\times \{a_2,b_2\}\times \dots\{a_n,b_n\})\setminus (a_1,a_2,\ldots,a_n)\subseteq \boldsymbol\rho$.
By the previous lemma, if there exists a perfect pair for a relation then the relation is a key relation.

\begin{lem}\label{trivialPatternMainLemma}

Suppose $\boldsymbol\rho$ is a key relation preserved
by a WNU $\boldsymbol f$, pattern of $\boldsymbol\rho$ is a trivial equivalence relation.
Then for every key tuple $\alpha$ there exists a
tuple $\beta$ such that
$(\alpha)-(\beta)$ is a perfect pair for $\boldsymbol\rho$.

\end{lem}

\begin{proof}

It is enough to show that there exists a perfect pair $(\gamma)-(\delta)$ for $\boldsymbol\rho$
because we can map $\gamma$ to the key tuple $\alpha$ by a vector-function $\Psi$ preserving $\boldsymbol\rho$
and put $\beta = \Psi(\delta)$.

The proof is by induction on the arity of $\boldsymbol\rho$. Denote $n = \ar(\boldsymbol\rho)$, $r = \ar(\boldsymbol f)$.
If $n=2$ then it follows from the definition.
To simplify explanation we assume that $(0,0,\ldots,0)$ is a key tuple for $\boldsymbol\rho$.

We say that a tuple $(d_{1},\ldots,d_{n})$ is \emph{good}
if \[(\{0,d_1\}\times \{0,d_2\}\times \dots\{0,d_n\})\setminus \{(0,0,\ldots,0)\}\subseteq \boldsymbol\rho.\]
Here we break our agreement and admit that $d_{i} = 0$.

We know from Lemma~\ref{ReduceArityOfKey} and Corollary~\ref{PatternChanging} that
the relations defined by
\[\boldsymbol\rho(0,x_{2},x_{3},x_{4},\ldots,x_{n}),\;
\boldsymbol\rho(x_{1},0,x_{3},x_{4},\ldots,x_{n}),\;
\boldsymbol\rho(x_{1},x_{2},0,x_{4},\ldots,x_{n})\]
are key relations with trivial pattern of smaller arity.
Then by the inductive assumption there
exist good tuples
$(a_1,0,a_3,\ldots,a_n)$, $(0,b_2,b_3,\ldots,b_n)$, $(c_1,c_2,0,c_{4},\ldots,c_n)$,
where $a_{i}\neq 0$, $b_{i}\neq 0$, and $c_{i}\neq 0$ for every $i$.
Let
\[B = \left\{m'\in \{0,\ldots,r-1\}\mid \left(\begin{smallmatrix}\boldsymbol f^{(1)}(0^{r-1}a_1)\\0\\\boldsymbol f^{(3)}(0^{r-m'}{a_3}^{m'})\\ 0\\0\\\dots\\0
\end{smallmatrix}\right)\in \boldsymbol\rho\right\}.\]
Obviously, $r-1\in B$.

Assume that $0\notin B$, then let $m$ be the maximal integer which is not from $B$.
Then the following pair is a perfect pair
\[\left(\begin{smallmatrix}\boldsymbol f^{(1)}(0^{r-1}a_1)\\0\\\boldsymbol f^{(3)}(0^{r-m}{a_3}^{m})\\ 0\\\vdots\\0
\end{smallmatrix}\right)-
\left(\begin{smallmatrix}\boldsymbol f^{(1)}( {c_1}^{r-1-m}a_{1} 0^{m})\\
\boldsymbol f^{(2)}({c_2}^{r-1-m}{0}^{m+1})\\
\boldsymbol f^{(3)}(0^{r-m-1}{a_3}^{m+1})\\
\boldsymbol f^{(4)}({c_4}^{r-1-m} {0}^{m+1})\\
\vdots\\
\boldsymbol f^{(n)}({c_n}^{r-1-m}  {0}^{m+1})\\
\end{smallmatrix}\right).\]

Thus we can assume that $0\in B$ and
$\left(\begin{smallmatrix}\boldsymbol f^{(1)}(0,\ldots,0,a_1)\\0\\0\\\vdots\\0
\end{smallmatrix}\right)\in \boldsymbol\rho$.

Assume that
$\left(\begin{smallmatrix}0\\ \boldsymbol f^{(2)}(0,b_2,\ldots,b_2)\\0\\\vdots\\0
\end{smallmatrix}\right)\notin \boldsymbol\rho$.
Then the following pair is prefect
\[
\left(\begin{smallmatrix}0\\\boldsymbol f^{(2)}(0,b_2,\ldots,b_2)\\0\\\vdots\\0
\end{smallmatrix}\right)-
\left(\begin{smallmatrix}\boldsymbol f^{(1)}(c_1,0,\ldots,0)\\
\boldsymbol f^{(2)}(c_2,b_2,\ldots,b_2)\\\boldsymbol f^{(3)}(b_3,0,\ldots,0)\\ \vdots\\\boldsymbol f^{(n)}(b_n,0,\ldots,0)
\end{smallmatrix}\right).\]

Thus, we can assume that
$\left(\begin{smallmatrix}0\\ \boldsymbol f^{(2)}(0,b_2,\ldots,b_2)\\0\\\vdots\\0
\end{smallmatrix}\right)\in \boldsymbol\rho$.
It remains to check that
\[
\left(\begin{smallmatrix}0\\0\\ 0\\0\\\vdots\\0
\end{smallmatrix}\right)-
\left(\begin{smallmatrix}\boldsymbol f^{(1)}(a_1,0,\ldots,0)\\\boldsymbol f^{(2)}(0,b_2,\ldots,b_2)\\
\boldsymbol f^{(3)}(a_3,b_3,\ldots,b_3)\\\boldsymbol f^{(4)}(a_4,b_4,\ldots,b_4)\\\vdots\\
\boldsymbol f^{(n)}(a_n,b_n,\ldots,b_n)
\end{smallmatrix}\right)\]
is a perfect pair.
\end{proof}

Theorem~\ref{trivialPattern} follows from Lemma~\ref{trivialPatternMainLemma}
and Lemma~\ref{trivialpatternLem}.

Let us prove another theorem from Section~\ref{MainResultsSection}.
\begin{NUtheorem}
Suppose $\rho$ is a key relation of arity greater than 2 preserved by a near-unanimity operation $f$.
Then the pattern of $\rho$ is a trivial equivalence relation.
\end{NUtheorem}

\begin{proof}

Assume that the pattern of $\rho$ is not a trivial equivalence relation.
Without loss of generality we assume that $1\overset{\rho}{\sim}2$.
Let $m$ be the arity of $f$.
Let $(a_{1},\ldots,a_{n})$ be a key tuple for $\rho$.
Put $\sigma(x,y,z) = \rho(x,y,z,a_{4},\ldots,a_{n})$.
Since $f$ is idempotent, $\sigma$ is preserved by $f$.
By Corollary~\ref{PatternChanging} and Lemma~\ref{ReduceArityOfKey}, $1\overset{\sigma}{\sim}2$.
Choose $b_1,b_2,b_3\in A$ such that $(b_1,a_2,a_3),(a_1,b_2,a_3),(a_1,a_2,b_3)\in \sigma$.
Put
\begin{multline*}\sigma_{m}(x_{1},\ldots,x_{m+1}) = 
\exists y_{1}\dots \exists y_{m}\;
\exists z_{1}\dots \exists z_{m}\;
\sigma(a_{1},y_{1},x_{1})\wedge
\sigma(z_{1},y_{1},a_{3})\wedge\\
\sigma(z_{1},y_{2},x_{2})\wedge
\sigma(z_{2},y_{2},a_{3})\wedge
\dots
\wedge\sigma(z_{m},y_{m},a_{3})
\wedge\sigma(z_{m},a_2,x_{m+1}).
\end{multline*}

Let us consider two cases. Assume that
$(a_{3},\ldots,a_{3})\notin\sigma_{m}$.
Let us check that
$(b_{3},\ldots,b_{3})$
witnesses that
$(a_{3},\ldots,a_{3})$ is an essential tuple for $\sigma_{m}$.
Suppose $x_{j} = b_{3}$ and $x_{i} = a_{3}$ for every $i\neq j$, then we put
$z_{i} = a_{1}$ if $i< j$,
$z_{i} = b_{1}$ if $i\ge j$,
$y_{i} = b_{2}$ if $i<j$,
$y_{i} = a_{2}$ if $i\ge j$.
Put $\beta_{i} = a_{3}^{i}b_{3}a_{3}^{m-i}$.
Obviously, $f(\beta_1,\ldots,\beta_{m}) = a_{3}^{m+1}$, which
contradicts the fact that $f$ preserves $\sigma_{m}$.

Let us consider the second case. Assume that
$(a_{3},\ldots,a_{3})\in\sigma_{m}$, then we can find appropriate $y_{1},\ldots, y_{m}$, $z_{1},\ldots, z_{m}$.
Since $1\overset{\sigma}{\sim}2$ we have
\[\{a_1, z_{1},z_2,\ldots,z_{m}\}\times \{a_2, y_{1},y_2,\ldots,y_{m}\}\times\{a_3\}
\subseteq \sigma,\]
which contradicts the fact that $(a_1,a_2,a_3)\notin \sigma.$
\end{proof}

\subsection{Key relations with almost trivial pattern}

In this section we prove statements from Section~\ref{MainResultsSection} for a key relation with almost trivial pattern.
\begin{lem}\label{almostTrivialPatternFirstLemma}
Suppose $(a_1,\ldots,a_n)$ is a key tuple for a relation $\boldsymbol\rho$ preserved by a WNU $\boldsymbol f$
whose pattern is $\{\{1,2\},\{3\},\{4\},\ldots,\{n\}\}.$
Then there exist $b_1,\ldots,b_n\in A$ such that
\[(\{a_1,b_1\}\times \{a_2,b_2\}\times \dots \times\{a_n,b_n\}) \setminus\{(a_1,a_2,\ldots,a_n),(b_1,b_2,a_3,\ldots,a_n)\}\subseteq \boldsymbol\rho.\]

\end{lem}

\begin{proof}
If $n=2$ then the statement can be easily checked. Assume that $n\ge 3$.
Let us show that it is sufficient to find a pair of tuples
$(c_1,\ldots,c_n)-(d_{1},\ldots,d_n)$
such that $(c_1,\ldots,c_n)\notin\boldsymbol \rho$,
\begin{align*}
(\{c_1\}\times\{c_2\}\times\{c_{3},d_{3}\}\times\dots\times\{c_{n},d_{n}\})
\setminus \{(c_{1},\ldots,c_{n})\}&\subseteq \boldsymbol\rho,\\
\{(c_1,d_2),(c_{2},d_1)\}\times\{c_{3},d_{3}\}\times\dots\times\{c_{n},d_{n}\}&\subseteq \boldsymbol\rho.
\end{align*}
In fact,
since $\overset{\boldsymbol\rho}{1\sim 2}$, we obtain
\[(\{c_1,d_1\}\times \{c_2,d_2\}\times \dots \times\{c_n,d_n\}) \setminus\{(c_1,c_2,\ldots,c_n),(d_1,d_2,c_3,\ldots,c_n)\}\subseteq \boldsymbol\rho.\]
It remains to map $(c_1,\ldots,c_n)$ to the key tuple $(a_1,\ldots,a_n)$ by a vector-function
preserving $\boldsymbol\rho$ to complete the proof.
A pair of tuples satisfying the above properties is called \emph{almost perfect}.

By Lemma~\ref{ReduceArityOfKey} and Corollary~\ref{PatternChanging} the relations $\boldsymbol\rho_1,\boldsymbol\rho_2$ defined by
\begin{align*}
\boldsymbol\rho_{1}(x_2,x_3,\ldots,x_{n}) &= \boldsymbol\rho(a_1,x_2,x_3,\ldots,x_{n}),\\
\boldsymbol\rho_{2}(x_1,x_{3},\ldots,x_{n}) &=  \boldsymbol\rho(x_1,a_2,x_3,\ldots,x_{n})
\end{align*}
are key relations with trivial pattern.
By Lemma \ref{trivialPatternMainLemma}, we can find perfect pairs $(a_2,\ldots,a_n)-(c_2,\ldots,c_n)$ and
$(a_1,a_3,\ldots,a_n)-(d_1,d_3,\ldots,d_n)$ for $\boldsymbol\rho_1$ and $\boldsymbol\rho_2$ correspondingly.

Put $b_{i} = \boldsymbol f^{(i)}(a_i,\ldots,a_i,c_{i})$ for $i=3,4,\ldots,n$.
Let $m$ be the arity of $\boldsymbol f$.
For $j\in \{0,1,2,\ldots,m\}$
let \[k_{j} = \boldsymbol f^{(1)}(\underbrace{d_1,\ldots,d_1}_j,a_1,\ldots,a_1),\;\;
l_{j} = \boldsymbol f^{(2)}(\underbrace{a_2,\ldots,a_2}_j,c_2,\ldots,c_2).\]
Suppose
$(e_3,\ldots,e_n)\in(\{a_3,b_3\}\times\dots\times\{a_n,b_n\})\setminus\{(a_3,\ldots,a_n)\}$.
Let us show that $(a_1,a_2,e_3,\ldots,e_n)\in\boldsymbol \rho$.
Denote $\beta_{j} = (k_{j}, a_2, e_{3},\ldots, e_{n})$.

Let us show by induction that
for every $j\in \{0,1, 2,\ldots,m-1\}$ we have
$\beta_j\in \boldsymbol\rho$.
We can check that $\beta_{m-1} \in \boldsymbol\rho$.
Since $\overset{\boldsymbol\rho}{1\sim 2}$ we have
\[
\left(\begin{smallmatrix}k_{j}\\l_{j}\\e_3\\\vdots\\e_n
\end{smallmatrix}\right),
\left(\begin{smallmatrix}k_{j-1}\\l_{j}\\e_3\\\vdots\\e_n
\end{smallmatrix}\right),
\left(\begin{smallmatrix}k_{j}\\a_2\\e_3\\\vdots\\e_n
\end{smallmatrix}\right)
\in \boldsymbol\rho\Rightarrow
\left(\begin{smallmatrix}k_{j-1}\\a_2\\e_3\\\vdots\\e_n
\end{smallmatrix}\right)\in \boldsymbol\rho.
\]
We can check that the first two tuples in the above formula always belong to $\boldsymbol\rho$.
Hence, $\beta_{j}\in \boldsymbol\rho \Rightarrow \beta_{j-1}\in \boldsymbol\rho$.
By induction we get $\beta_{0}\in \boldsymbol\rho$,
which means that $(a_1,a_2,e_3,\ldots,e_n)\in\boldsymbol \rho$.
Now, we can easily find an almost perfect pair.

If
$\left(\begin{smallmatrix}
\boldsymbol f^{(1)}(d_1,\ldots, d_1,a_1)\\
a_2\\
a_3\\
\vdots\\
a_n
\end{smallmatrix}\right)\notin \boldsymbol\rho$,
then $\left(\begin{smallmatrix}
\boldsymbol f^{(1)}(d_1,\ldots, d_1,a_1)\\
a_2\\
a_3\\
\vdots\\
a_n
\end{smallmatrix}\right)-
\left(\begin{smallmatrix}
d_1\\
\boldsymbol f^{(2)}(a_2,\ldots,a_2,c_2)\\
\boldsymbol f^{(3)}(a_3,\ldots,a_3,d_3)\\
\vdots\\
\boldsymbol f^{(n)}(a_n,\ldots,a_n,d_n)
\end{smallmatrix}\right)$
 is
an almost perfect pair.
Otherwise,
$\left(\begin{smallmatrix}
a_1\\
a_2\\
a_3\\
\vdots\\
a_n
\end{smallmatrix}\right)-
\left(\begin{smallmatrix}
\boldsymbol f^{(1)}(d_1,\ldots, d_1,a_1)\\
c_2\\
\boldsymbol f^{(3)}(a_3,\ldots,a_3,c_3)\\
\vdots\\
\boldsymbol f^{(n)}(a_n,\ldots,a_n,c_n)
\end{smallmatrix}\right)$
is an almost perfect pair.
This completes the proof.
\end{proof}

Let us prove a lemma from Section~\ref{MainResultsSection}.
\begin{AlmostTrivialPatternLemma}

Suppose $1\overset{\boldsymbol\rho}{\sim}2$, $(a_1,a_2,\ldots,a_n)\notin\boldsymbol\rho$,
$b_{1},\ldots,b_{n}\in A$, 
and
\[(\{a_1,b_1\}\times \dots\times\{a_n,b_n\})\setminus \{(a_1,a_2,a_3,\ldots,a_n),
(b_1,b_2,a_3,\ldots,a_n)\}\subseteq \boldsymbol\rho.\]
Then $\boldsymbol\rho$ is a key relation and $(a_1,a_2,\ldots,a_n)$ is a key tuple for the relation $\boldsymbol\rho$.

\end{AlmostTrivialPatternLemma}

\begin{proof}

To prove the statement, for every $(c_{1},\ldots,c_{n})\notin \boldsymbol\rho$ we need to find
a vector-function $\Psi$ preserving $\boldsymbol\rho$ that maps $(c_{1},\ldots,c_{n})$ to $(a_{1},\ldots,a_{n})$.

We consider two cases. First, assume that $(c_{1}',c_2,\ldots,c_{n})\in \boldsymbol\rho$ for some $c_1'$.
Put $\boldsymbol\sigma(x_1,x_2) = \boldsymbol\rho(x_1,x_2,c_3,\ldots,c_n)$.
Since $1\overset{\boldsymbol\rho}{\sim}2$,
every connected component of $\boldsymbol\sigma$ can be defined as $A_1\times A_2$.
Let the connected component containing $(c_{1}',c_2)$ be $A_1\times A_2$.
Then we define the vector-function as follows.
$\Psi^{(1)}(x) =
\begin{cases}
a_{1}, & \text{if $x \notin A_1$;}\\
b_{1}, & \text{if $x \in A_1$.}
\end{cases}$,
$\Psi^{(2)}(x) =
\begin{cases}
a_{2}, & \text{if $x \in A_2$;}\\
b_{2}, & \text{if $x \notin A_2$.}
\end{cases}$,
$\Psi^{(i)}(x) =
\begin{cases}
a_{i}, & \text{if $x = c_{i}$;}\\
b_{i}, & \text{otherwise.}
\end{cases}$ for $i\ge 3$.
It is easy to see that it satisfies the above properties.

Second case. Assume that for every $c_1'$ we have $(c_{1}',c_2,\ldots,c_{n})\notin \boldsymbol\rho$.
Then we define $\Psi$ as follows.
$\Psi^{(1)}(x) = a_1$,
$\Psi^{(i)}(x) =
\begin{cases}
a_{i}, & \text{if $x = c_{i}$;}\\
b_{i}, & \text{otherwise.}
\end{cases}$ for $i\ge 2$.
It is easy to see that it satisfies the necessary properties. This completes the proof.
\end{proof}

\begin{theoremaboutsemilattice}
Suppose $\boldsymbol \rho$ is a key relation preserved by a semilattice operation or a 2-semilattice operation.
Then the pattern of $\boldsymbol \rho$ is either trivial, or almost trivial.
\end{theoremaboutsemilattice}

\begin{proof}
We can easily check that
for every semilattice operation or 2-semilattice operation $\boldsymbol s$
we can define a WNU of arity $m$ as follows
$\boldsymbol f_{m}(x_{1},\ldots,x_m) = \boldsymbol s(\boldsymbol s(\ldots(\boldsymbol s(\boldsymbol s(
\boldsymbol s(x_{1},x_{2}),x_3),x_4),\ldots),x_{m-1}),x_m).$
Thus, for every $m\ge 2$ there exists a WNU of arity $m$ preserving $\boldsymbol\rho$.

Assume that the pattern of $\boldsymbol\rho$ is not trivial and not almost trivial.
Without loss of generality we assume
that $1\overset{\boldsymbol\rho}{\sim}2$ and $2\overset{\boldsymbol\rho}{\sim}3$.
Let $(a_{1},\ldots,a_{n})$ be a key tuple for $\boldsymbol\rho$.
Put
$\boldsymbol\rho'(x_1,x_2,x_3) = \boldsymbol\rho(x_1,x_2,x_3,a_4,\ldots,a_n)$.
Then $\boldsymbol\rho'$ is a key relation with full pattern.

Let $\boldsymbol\sigma$ be the core of $\boldsymbol\rho'$.
Put $m = |A_1|$,
by Lemma~\ref{coreispreservedbyWNU},
there exists a WNU $\boldsymbol f$ of arity $m$ preserving $\boldsymbol\sigma$.
By Lemma~\ref{WNUIsLinear}, for every key block
$A_{1}\times A_{2}\times \dots \times A_{n}$ for $\boldsymbol\sigma$
there exist an abelian group $(A_1,+)$ and an integer $t$ such that
for every $a_1,\ldots,a_m \in A_{1}$
we have $\boldsymbol f^{(1)}(a_{1},\ldots,a_{m}) = t\cdot a_{1}+t\cdot a_2 + \ldots + t\cdot a_{m}$.
Since $m = |A_1|$, and the order of a group divides the order of any element in the group,
we obtain $\boldsymbol f^{(1)}(x,x,\ldots,x) = 0$, which is not possible
because $\boldsymbol f^{(1)}$ is idempotent.
\end{proof}

\subsection{Key relations with arbitrary pattern}

\begin{lem}\label{AddTupleToCube}
Suppose $p$ is a prime number, $A = \{0,1,\ldots,p-1\}$,
$\rho\subseteq A^n,$ $n>2$, has a full pattern,
$\{(x_1,\ldots,x_n)\mid x_1+\ldots+x_n = 0 (\mod p)\}\subset \rho.$
Then $\rho = A^n$.
\end{lem}

\begin{proof}
There exists
$(a_1,\ldots,a_{n})\in\rho$ such that
$a_1+\ldots+a_n \neq 0$.
Hence, the connected component of $\rho$ containing $(a_1,\ldots,a_{n})$ has more than one element.
Let it be $B_1\times\dots\times B_n$.
Without loss of generality we assume that $|B_1|\le |B_2|\le\dots\le|B_n|$.

Since $(-a_2-a_3-\ldots-a_{n-1}-c,a_2,\ldots,a_{n-1},c)\in \rho$ for every $c\in A$, we
have
$\{-a_2-a_3-\ldots-a_{n-1}-c\mid c\in B_{n}\}\subseteq B_1$.
Hence $|B_1|= |B_2|=\dots=|B_n|$.
Similarly,
$\{-a_2-a_3-\ldots-a_{n-2}-c-d\mid c\in B_{n-1}, d\in B_{n}\}\subseteq B_1$,
which cannot be true
because
$|B_{n-1}+B_{n}|>|B_{n}|=|B_{1}|$.
\end{proof}

\begin{lem}\label{LemmaForMainTheorem}
Suppose $\boldsymbol\rho$ is a core of arity $n$ preserved by a WNU $\boldsymbol f$
whose pattern is $\{\{1,2,\ldots,r\},\{r+1\},\{r+2\},\ldots,\{n\}\}.$
Then for every key tuple $(a_1,\ldots,a_n)$ there exist
$\boldsymbol B = B_1\times B_2 \times \dots \times B_n$,
a prime number $p$
and bijective mappings
$\phi_i: B_i\to \mathbb Z_{p}$ for $i=1,2,\ldots,r$ such that
$(a_1,\ldots,a_n) \in \boldsymbol B$,
$B_i = \{a_i,b_i\}$ for $i=r+1,\ldots,n$,
\[
\boldsymbol \rho \cap \boldsymbol B
=
(\phi_1(x_1) + \ldots +\phi_r(x_r) = 0)
\vee
(x_{r+1} = b_{r+1})
\vee \dots \vee (x_n = b_n),
\]
and every tuple $\gamma \in \boldsymbol B\setminus \boldsymbol \rho$ is a key tuple for $\boldsymbol \rho$.
\end{lem}

\begin{proof}

If $r=1$ then the statement follows from Lemma~\ref{trivialPatternMainLemma}.
If $r=2$ then the statement follows from Lemma~\ref{almostTrivialPatternFirstLemma}.
Suppose $r\ge 3$. Without loss of generality we can assume that
$\boldsymbol f$ satisfies the condition from Lemma~\ref{findbetterWNU}.
Let us define two relations:
\begin{align*}
\boldsymbol\rho_{1} (x_1,\ldots,x_r) &= \boldsymbol\rho(x_1,\ldots,x_r,a_{r+1},\ldots,a_{n}),\\
\boldsymbol\rho_{2} (x_r,\ldots,x_{n}) &= \boldsymbol\rho(a_1,\ldots,a_{r-1},x_r,\ldots,x_{n}).
\end{align*}
By Lemma~\ref{ReduceArityOfKey} and Corollary~\ref{PatternChanging}, these are key relations,
and the pattern of $\boldsymbol\rho_1$ is full, the pattern of $\boldsymbol\rho_2$ is trivial.
By Theorem~\ref{trivialPattern},
there exists a perfect pair $(a_{r},\ldots,a_{n})-(c_{r},\ldots,c_{n})$ for $\boldsymbol\rho_2$.

Choose $c_1$ such that
$(c_1,a_2,\ldots,a_n)\in\boldsymbol \rho$, then
by Lemma~\ref{Keytupleb1b2}
the tuple $(c_1,a_2,\ldots,a_{r-1},c_{r},a_{r+1},\ldots,a_{n})$ is
a key tuple for $\boldsymbol\rho$. Hence, by Lemma~\ref{keyBlockBijection},
$\boldsymbol f^{(r)}(c_{r},\ldots,c_{r},x)$
is a bijection on $\proj_{r}\boldsymbol \rho$.
By Lemma~\ref{keyBlockBijection},
$\boldsymbol f^{(i)}(a_{i},\ldots,a_{i},x)$
is a bijection on $\proj_{i}\boldsymbol \rho$ for every $i\in\{1,2,\ldots,n\}$.
Since $\boldsymbol f$ satisfies the condition from Lemma~\ref{findbetterWNU},
we obtain
$\boldsymbol f^{(i)}(a_{i},\ldots,a_{i},x) =x$ for every $x\in\proj_{i}\boldsymbol\rho$
and $\boldsymbol f^{(r)}(c_{r},\ldots,c_{r},x) =x$ for every $x\in\proj_{r}\boldsymbol\rho$.

By Lemmas~\ref{CoreExistence} and \ref{CoreProperties}
there exists a core $\boldsymbol\sigma_1$ of $\boldsymbol\rho_1$ such that $(a_1,\ldots,a_r)$ is a key tuple for $\boldsymbol\sigma_1$.
By Theorem~\ref{MainForFullPattern}, there exist a key block
$A_1\times A_2 \times \dots \times A_r$ for $\boldsymbol\sigma_1$ containing $(a_1,\ldots,a_r)$,
a finite field $F$, and bijective mappings
$\phi_i: A_i\to F$ such that
\[\boldsymbol\sigma_1\cap (A_1\times \dots \times A_r) =
 \{(x_1,\dots,x_r) \mid \phi_1(x_1) + \phi_2(x_2) + \dots +\phi_r(x_r)=0\}.\]
Let $e = \phi_1(a_1) + \phi_2(a_2) + \dots +\phi_r(a_r).$
Obviously $e\neq 0$. Let $p$ be the characteristic of $F$.
Let us define a mapping $\psi_{i}\colon \mathbb Z_{p}\to A_i$ for every $i\in\{1,2,\ldots,r\}$.
Put $\psi_{i}(x) = \phi_{i}^{-1}(\phi_{i}(a_i)+x\cdot e)$,
$B_{i} = \{\psi_{i}(x)\mid x\in\mathbb Z_{p}\}$.
Denote
$b_{i} = \boldsymbol f^{(i)}(c_i,\ldots,c_i,a_{i})$ for $i=r+1,\ldots,n$.
By the definition $\psi_{i}$ is a bijective mapping from $\mathbb Z_p$ to $B_i$ for every $i$ and
\[\boldsymbol\rho_1\cap (B_1\times \dots\times B_r) =
\{(x_1,\ldots,x_r)\mid \psi_1^{-1}(x_1)+\ldots+\psi_r^{-1}(x_r) = 0\}.\]
It remains to show that
for every $(d_{1},\ldots,d_{r})\in B_1\times\dots\times B_r$ and every
\[(d_{r+1},\ldots,d_n)\in (\{a_{r+1},b_{r+1}\}\times\dots \times\{a_{n},b_{n}\})\setminus \{(a_{r+1},\ldots,a_n)\}\]
we have $(d_{1},\ldots,d_{n})\in \boldsymbol\rho$.
Without loss of generality we assume that
$d_{i} = b_{i}$ for $i\in\{r+1,\ldots,r'\}$ and
$d_{i} = a_{i}$ for $i\in\{r'+1,\ldots,n\}$.
First, let us prove this if
$(d_1,\ldots,d_r)\in \boldsymbol\rho_1$.
Since $\boldsymbol f$ preserves $\boldsymbol\rho$ we have
\[\left(\begin{smallmatrix}
d_1\\d_2\\d_3\\\vdots\\d_n
\end{smallmatrix}\right) =
\boldsymbol f\left(\begin{smallmatrix}a_1&\dots& a_1&d_1\\
\vdots&\ddots&\vdots&\vdots
\\
a_{r-1}&\dots& a_{r-1}& d_{r-1}
\\
c_{r}&\dots& c_{r}& d_{r}
\\
c_{r+1}&\dots& c_{r+1}& a_{r+1}
\\
\vdots&\ddots&\vdots&\vdots
\\
c_{r'}&\dots& c_{r'}&a_{r'}
\\
a_{r'+1}&\dots& a_{r'+1}& a_{r'+1}
\\
\vdots& \ddots& \vdots& \vdots&
\\
a_{n}&\ldots& a_{n}&a_{n}
\end{smallmatrix}\right)\in\boldsymbol\rho.\]
Second, let us prove this fact if
$(d_1,\ldots,d_r) = (a_1,\ldots,a_r)$.
Then
\[\left(\begin{smallmatrix}
a_1\\a_2\\a_3\\\vdots\\a_r\\
d_{r+1}\\\vdots\\d_n
\end{smallmatrix}\right) =
\boldsymbol f\left(\begin{smallmatrix}
a_1&\dots& a_1&a_1\\
\vdots&\ddots&\vdots&\vdots
\\
a_{r-1}&\dots& a_{r-1}& a_{r-1}
\\
c_{r}&\dots& c_{r}& a_{r}
\\
c_{r+1}&\dots& c_{r+1}& a_{r+1}
\\
\vdots&\ddots&\vdots&\vdots
\\
c_{r'}&\dots& c_{r'}&a_{r'}
\\
a_{r'+1}&\dots& a_{r'+1}& a_{r'+1}
\\
\vdots& \ddots& \vdots& \vdots&
\\
a_{n}&\ldots& a_{n}&a_{n}
\end{smallmatrix}\right)=
\boldsymbol f\left(\begin{smallmatrix}
a_1&a_1&\dots& a_1&a_1\\
\vdots&\vdots&\ddots&\vdots&\vdots
\\
a_{r-1}&a_{r-1}&\dots& a_{r-1}& a_{r-1}
\\
a_{r}&c_r&\dots& c_{r}&c_r
\\
c_{r+1}&c_{r+1}&\dots& c_{r+1}& a_{r+1}
\\
\vdots&\vdots&\ddots&\vdots&\vdots
\\
c_{r'}&c_{r'}&\dots& c_{r'}&a_{r'}
\\
a_{r'+1}&a_{r'+1}&\dots& a_{r'+1}& a_{r'+1}
\\
\vdots& \vdots& \ddots& \vdots& \vdots&
\\
a_{n}&a_{n}&\ldots& a_{n}&a_{n}
\end{smallmatrix}\right)\in\boldsymbol\rho.\]
Then we apply Lemma~\ref{AddTupleToCube} to the relation defined by
\[\boldsymbol\sigma(x_1,\ldots,x_r) = \boldsymbol\rho(x_1,\ldots,x_r,d_{r+1},\ldots,d_n),\]
and prove that
$(d_1,\ldots,d_{n})\in\boldsymbol\rho$ in general.

It remains to show that any tuple $\beta = (d_1,d_2,\ldots,d_r,a_{r+1},\ldots,a_n)\notin\boldsymbol\rho$ is a key tuple.
Put $\beta' = (d_1,d_2,\ldots,d_r), \alpha' = (a_1,a_2,\ldots,a_r),\alpha = (a_1,a_2,\ldots,a_n)$.
Since $\alpha$ is a key tuple for $\boldsymbol\rho$, there exists a vector-function $\Psi$
that maps $\beta$ to $\alpha$ and preserves $\boldsymbol\rho$.
Let $\Omega$ be a restricting vector-function for the core $\boldsymbol\sigma_1$ of $\boldsymbol\rho_1$.
Put $\Psi' = (\Psi^{(1)},\Psi^{(2)},\ldots,\Psi^{(r)})$.
Since $\beta'$ belongs to a key block, $\beta'$ is a key tuple for $\boldsymbol\sigma_{1}$.
Hence, $\Omega\circ\Psi'$ preserves $\boldsymbol\sigma_{1}$ and maps the key tuple $\beta'$ to the key tuple $\alpha'$.
Then by Lemma~\ref{CoreProperties}(5) $\Omega(\Psi'(\alpha'))\notin\boldsymbol\sigma_{1}$. Therefore, $\Psi(\alpha)\notin\boldsymbol\rho$.
By Lemma~\ref{CoreProperties}(5), $\Psi$ is a bijective mapping on $\proj\boldsymbol\rho$.
Hence $\beta$ is a key tuple for $\boldsymbol\rho$.
\end{proof}

Now, Theorem~\ref{MainTheorem} is a trivial collorary of Lemma~\ref{LemmaForMainTheorem}.
In fact, for a key tuple we just consider a core for which it is a key tuple, and apply Lemma~\ref{LemmaForMainTheorem}.

\section{Key relations with full pattern}

\subsection{Blocks of a core}

A tuple $(a_1,\ldots,a_{n})\notin \boldsymbol\rho$ is called a \emph{weak essential tuple} for a relation $\boldsymbol \rho$ if
there exist distinct $i_{1},i_{2},i_{3}\in\{1,2,\ldots,n\}$ and $b_{1},b_{2},b_{3}\in A$ such that
$(a_{1},\ldots,a_{i_{j}-1},b_{j},a_{i_{j}+1},\ldots,a_{n})\in\boldsymbol \rho$
for every $j\in\{1,2,3\}$.
Obviously, every essential tuple of arity greater than 2 is a week essential tuple.

\begin{lem}\label{CoreConnectedMain}

Suppose $\boldsymbol\sigma$ is a core of arity greater than 2 with full pattern preserved by a WNU $\boldsymbol f$, and $\alpha-\beta-\gamma$ is a path, where
$\alpha\in \boldsymbol\sigma$ belongs to a key block,
$\beta$ is a weak essential tuple,
$\gamma\in \boldsymbol\sigma$.
Then $\gamma$ belongs to a key block.

\end{lem}

\begin{proof}

Since $\beta$ is a weak essential tuple, without loss of generality we assume that $\alpha = (a_{1},\ldots,a_n)$,
$\beta = (b_{1},a_{2},\ldots,a_{n})$,
$\gamma = (b_{1},b_{2},a_{3},\ldots,a_{n})$,
and 
$(b_{1},a_{2},b_{3},a_{4},\ldots,a_{n}) \in \boldsymbol\sigma$
for some $b_3\in A$.

Also, we assume that
$\boldsymbol f$ satisfies the condition from Lemma~\ref{findbetterWNU}.
Then, by Lemma~\ref{keyBlockBijection},
$\boldsymbol f^{(3)}(a_3,\ldots,a_3,x) = x$ for every $x\in\proj_{3}\boldsymbol\sigma$.
If $\alpha$ and $\gamma$ are adjacent in $\boldsymbol\sigma$ then the statement is obvious. Therefore,
we can assume that $b_{1}\neq a_{1}$ and $b_{2}\neq a_{2}$.

Assume that $\gamma$ doesn't belong to a key block.
By Theorem~\ref{MainForFullPattern}, there exists a key tuple $(a_{1},a_2,a_{3}',a_{4},\ldots,a_n)$ in the key block containing $\alpha$.
Then by Lemma~\ref{KeyBlockSpread} we have
$(b_{1},b_{2},a_{3}',a_4,\ldots,a_{n})\in \boldsymbol\sigma$.

Put $c_{i}^{[0]} = b_{i}$ for $i = 1,2,3$,
$c_{i}^{[j+1]} = \boldsymbol f^{(i)}(c_{i}^{[j]},\ldots, c_{i}^{[j]},a_i)$.

Let $\boldsymbol\delta= \{(d_1,d_2,d_3)\mid (d_{1},d_{2},d_{3},a_{4},\ldots,a_{n})\in \boldsymbol\sigma\}$.

Since
$\left(\begin{smallmatrix}c_{1}^{[0]}\\a_{2}\\ c_{3}^{[0]}\end{smallmatrix}\right),
\left(\begin{smallmatrix}c_{1}^{[0]}\\c_{2}^{[0]}\\ a_3\end{smallmatrix}\right),
\left(\begin{smallmatrix}c_{1}^{[0]}\\c_{2}^{[0]}\\ a_3'\end{smallmatrix}\right)\in \boldsymbol\delta$,
and
$\left(\begin{smallmatrix}
c_{1}^{[j+1]}\\
a_2\\
c_{3}^{[j+1]}
\end{smallmatrix}\right)
 = \left(\begin{smallmatrix}
\boldsymbol f^{(1)}(c_{1}^{[j]},\ldots,c_{1}^{[j]},a_1)\\
\boldsymbol f^{(2)}(a_2,\ldots,a_2)\\
\boldsymbol f^{(3)}(c_{3}^{[j]},\ldots,c_{3}^{[j]},a_3)
\end{smallmatrix}\right)$,
$\left(\begin{smallmatrix}
c_{1}^{[j+1]}\\
c_{2}^{[j+1]}\\
a_3
\end{smallmatrix}\right)
 = \left(\begin{smallmatrix}
\boldsymbol f^{(1)}(c_{1}^{[j]},\ldots,c_{1}^{[j]},a_1)\\
\boldsymbol f^{(2)}(c_{2}^{[j]},\ldots,c_{2}^{[j]},a_2)\\
\boldsymbol f^{(3)}(a_3,\ldots,a_3)
\end{smallmatrix}\right)$,
$\left(\begin{smallmatrix}
c_{1}^{[j+1]}\\
c_{2}^{[j+1]}\\
a_3'
\end{smallmatrix}\right)
 = \left(\begin{smallmatrix}
\boldsymbol f^{(1)}(c_{1}^{[j]},\ldots,c_{1}^{[j]},a_1)\\
\boldsymbol f^{(2)}(c_{2}^{[j]},\ldots,c_{2}^{[j]},a_2)\\
\boldsymbol f^{(3)}(a_3',a_3,\ldots,a_3)
\end{smallmatrix}\right)$,
we obtain that
$\left(\begin{smallmatrix}c_{1}^{[j]}\\a_{2}\\ c_{3}^{[j]}\end{smallmatrix}\right),
\left(\begin{smallmatrix}c_{1}^{[j]}\\c_{2}^{[j]}\\ a_3\end{smallmatrix}\right),
\left(\begin{smallmatrix}c_{1}^{[j]}\\c_{2}^{[j]}\\ a_3'\end{smallmatrix}\right)
\in \boldsymbol\delta$ for every $j=0,1,2,\ldots$.

We have
$\left(\begin{smallmatrix}
c_{1}^{[j]}\\
c_{2}^{[j+1]}\\
c_{3}^{[j]}
\end{smallmatrix}\right)
 = \left(\begin{smallmatrix}
\boldsymbol f^{(1)}(c_{1}^{[j]},\ldots,c_{1}^{[j]})\\
\boldsymbol f^{(2)}(c_{2}^{[j]},\ldots,c_{2}^{[j]},a_2)\\
\boldsymbol f^{(3)}(a_3,\ldots,a_3,c_{3}^{[j]})
\end{smallmatrix}\right)$,
$\left(\begin{smallmatrix}
c_{1}^{[j]}\\
c_{2}^{[k+1]}\\
c_{3}^{[j]}
\end{smallmatrix}\right)
 = \left(\begin{smallmatrix}
\boldsymbol f^{(1)}(c_{1}^{[j]},\ldots,c_{1}^{[j]},c_{1}^{[j]})\\
\boldsymbol f^{(2)}(c_{2}^{[k]},\ldots,c_{2}^{[k]},a_2)\\
\boldsymbol f^{(3)}(c_{3}^{[j]},\ldots,c_{3}^{[j]},c_{3}^{[j]})
\end{smallmatrix}\right)$
for every $j$ and $k$,
hence
we get $\left(\begin{smallmatrix}
c_{1}^{[j]}\\
c_{2}^{[k]}\\
c_{3}^{[j]}
\end{smallmatrix}\right)\in\boldsymbol\delta$
for every $k>j$.
Since $A$ is finite,
$c_{2}^{[l]} = c_{2}^{[m]}$ for some $l>m$.
The above argument gives
$\left(\begin{smallmatrix}
c_{1}^{[m]}\\
c_{2}^{[m]}\\
c_{3}^{[m]}
\end{smallmatrix}\right)\in \boldsymbol\delta$.
Since $\boldsymbol\sigma$ has full pattern and
$
\left(\begin{smallmatrix}c_{1}^{[m]}\\a_{2}\\ c_{3}^{[m]}\end{smallmatrix}\right),
\left(\begin{smallmatrix}c_{1}^{[m]}\\c_{2}^{[m]}\\ a_3\end{smallmatrix}\right),
\left(\begin{smallmatrix}c_{1}^{[m]}\\c_{2}^{[m]}\\ a_3'\end{smallmatrix}\right)
\in \boldsymbol\delta$,
we get
$\left(\begin{smallmatrix}
c_{1}^{[m]}\\
a_2\\
a_3
\end{smallmatrix}\right),
\left(\begin{smallmatrix}
c_{1}^{[m]}\\
a_2\\
a_3'
\end{smallmatrix}\right)\in \boldsymbol\delta$.
This means that
$(c_{1}^{[m]},a_2,a_3,a_4,\ldots,a_{n}),
(c_{1}^{[m]},a_2,a_3',a_{4},\ldots,a_{n})\in \boldsymbol\sigma,$
which contradicts Lemma~\ref{KeyBlockSpread2}.
Therefore $\gamma$ belongs to a key block.
\end{proof}

\begin{cor}\label{corAboutTuplesFromBlock}

Suppose $\boldsymbol\rho$ is a core of arity greater than 2 with full pattern preserved by a WNU,
$\boldsymbol \delta$ is a connected component of $\widetilde{\boldsymbol\rho}$
containing a key tuple.
Then every tuple $\beta\in\boldsymbol\rho\cap \boldsymbol\delta$
belongs to a key block.

\end{cor}

\begin{proof}
Because of connectedness,
we can find a path $\alpha_{0}-\alpha_{1}-\alpha_{2}-\dots-\alpha_{m} = \beta$
where $\alpha_{0}$ belongs to a key block and $\alpha_{i}\in \widetilde{\boldsymbol\rho}$ for every $i$.
It is easy to avoid the situation
where $\alpha_{j}$, $\alpha_{j+1} \notin \boldsymbol\rho$. In this case we just find
$\alpha_{j}'\in \boldsymbol\rho$ to get a path
$\alpha_{j}-\alpha_j'- \alpha_{j+1}$.
It remains to use Lemma~\ref{CoreConnectedMain}
\end{proof}

\begin{thm}\label{CoreBlockIsCubic}
Suppose $\boldsymbol\rho$ is a core of arity $n\ge 3$ with full pattern preserved by a WNU.
Then every connected component $\boldsymbol\delta$ of $\widetilde{\boldsymbol\rho}$
containing a key tuple can be represented as $A_{1}\times A_2 \times \dots \times A_n$,
where $A_1,\ldots, A_n\subseteq A$.
\end{thm}

\begin{proof}

Let $A_i = \proj_i \boldsymbol\delta$ for every $i=1,2,\ldots,n$.
By Lemma~\ref{preserveConnectedComponent}, every connected component of $\widetilde{\boldsymbol\rho}$ is preserved by the WNU,
therefore $A_{i}$ is preserved by the WNU for every $i$.
Thus, we have a WNU on $A_{i}$ for every $i$.

For $k=2,3,\ldots,|A_1|$ we define a relation $\sigma_{k}$  as follows:
\begin{multline*}
\sigma_{k}(y_1,y_2,\ldots,y_k) =
\exists x_1 \exists x_2\dots \exists x_n 
 \;\;\boldsymbol\delta(x_{1},x_2,x_3,\ldots,x_n) \wedge\\
 \widetilde{\boldsymbol\rho}(y_1,x_2,x_3,\ldots,x_n)
\wedge \widetilde{\boldsymbol\rho}(y_2,x_2,x_3,\ldots,x_n)
\wedge\dots
\wedge \widetilde{\boldsymbol\rho}(y_k,x_2,x_3,\ldots,x_n).
\end{multline*}

By Corollary~\ref{corAboutTuplesFromBlock} and Corollary~\ref{KeyBlockIsCubic},
for every $g_1, g_2\in A_1$ we can find key tuples $\gamma_1, \gamma_2$ such that
$\gamma_1(1) = g_1$ and  $\gamma_2(1) = g_2$.
Then by Lemma~\ref{CoreProperties}(5) there exists a unary vector-function $\Psi$ preserving $\boldsymbol\rho$ such
that $\Psi(\gamma_1) = \gamma_2$ and $\Psi$ is a bijection on $\proj\boldsymbol\rho$.
Thus, we proved that for all $g_1, g_2\in A_1$
there exists a mapping $\Psi^{(1)}$ from $g_1$ to $g_2$ which preserves
$\sigma_i$ for every $i$ and is a bijection on $A_1$.

Obviously, $\sigma_2$ is reflexive and symmetric on $A_1$.
Since $A_1$ is a projection of the connected component $\boldsymbol\delta$,
for any $b,c\in A_{1}$ there exists a path $\alpha_1,\ldots,\alpha_{s}$ in $\boldsymbol\delta$ such that
$\alpha_{1}(1) = b$,  $\alpha_{s}(1) = c$.
It follows from the definition of $\sigma_2$ that if
$\alpha_{i}$ and $\alpha_{i+1}$ differ in the first component then $(\alpha_{i}(1),\alpha_{i+1}(1))\in\sigma_2$.
Therefore, the graph defined by $\sigma_2$ is connected on $A_1$.
Then by Lemma~\ref{NoTotallyReflexiveA} we get $\sigma_2 = A_1\times A_1$.

Then we consider the minimal $l>2$ such that $\sigma_l\neq A_{1}^{l}$.
Obviously $\sigma_l$ is totally reflexive and symmetric on $A_1$, which contradicts
Lemma~\ref{NoTotallyReflexiveA}.
Therefore $\sigma_{|A_1|} = A_{1}^{|A_1|}$, which means
that there exists
$(x_1,\ldots,x_n)\in \boldsymbol\rho\cap\boldsymbol\delta$ such that
for every $y\in A_1\setminus\{x_1\}$ the tuple
$(y, x_2,\ldots,x_n)$ belongs to $\widetilde{\boldsymbol\rho}.$

Recall that we consider a core, and every tuple from $\boldsymbol\rho\cap\boldsymbol\delta$ belongs to a key block.
By Corollary~\ref{corAboutTuplesFromBlock} and Theorem~\ref{MainForFullPattern},
two tuples from $\boldsymbol\rho\cap\boldsymbol\delta$ cannot differ in just one component.
Hence,
$(y, x_2,\ldots,x_n)\notin\boldsymbol\rho$ for every $y\in A_1\setminus\{x_1\}$,
and $(y, x_2,\ldots,x_n)$ is an essential tuple for $\boldsymbol\rho$.

Denote $\alpha = (x_1,\ldots,x_n)$.
Suppose $\beta\in \boldsymbol\rho\cap\boldsymbol\delta$, and  $\alpha',\beta'$ are key tuples
obtained from $\alpha$ and $\beta$ by changing the first component.
Since two tuples from $\boldsymbol\rho\cap\boldsymbol\delta$ cannot differ in just one component,
a unary vector-function that maps $\alpha'$ to $\beta'$ and preserves $\boldsymbol\rho$ also maps $\alpha$ to $\beta$.
This means, that for every tuple $\beta = (x_1',\ldots,x_n')\in \boldsymbol\rho\cap\boldsymbol\delta$
and every $y\in A_1\setminus\{x_1'\}$ the tuple
$(y, x_2',\ldots,x_n')$ is essential for $\boldsymbol\rho.$

In the same way as we did for the first component of the tuple,
we can show that
for every tuple $(x_1,\ldots,x_n) \in \boldsymbol\rho\cap\boldsymbol\delta$,
every $i\in\{2,3,\ldots,n\}$ and every $y\in A_i\setminus\{x_i\}$
the tuple $(x_1, \ldots,x_{i-1},y,x_{i+1},\ldots,x_n)$ is essential for $\boldsymbol\rho.$

We want to show that every tuple in $A_{1}\times\dots\times A_n$ is either essential,
or from $\boldsymbol\rho$.
Assume the converse, assume that $\beta\in(A_{1}\times\dots\times A_n) \setminus \widetilde{\boldsymbol\rho}$,
$\beta$ differs from every tuple of $\boldsymbol\rho$ in at least $t\ge 1$ components,
and $t$ cannot be reduced.
If $t=1$ then $\beta$ is essential because of the above property.
Suppose $t\ge 2$ and $\beta$ differs from $\gamma\in \boldsymbol\rho$ in $t$ components. 
Without loss of generality let it be the first $t$ components.
We know that the tuple $(\beta(1),\gamma(2),\ldots,\gamma(n))$ is an essential tuple for $\boldsymbol\rho$.
Hence we can change the second component of it to get a tuple from $\boldsymbol\rho$.
We denote this tuple by $\gamma' = (\beta(1),b, \gamma(3),\ldots,\gamma(n))$.
Obviously $\beta$ and $\gamma'$ differ in $t-1$ components, which contradicts our assumption.
This completes the proof.
\end{proof}

Recall that a connected component of $\widetilde{\boldsymbol\rho}$ is called a block of $\boldsymbol\rho$.

\begin{thm}\label{CoreBlockIsLinear}

Suppose $\boldsymbol\sigma$ is a core of arity $n\ge 3$ with full pattern preserved by a WNU $\boldsymbol f$.
Then
\begin{enumerate}
\item Every block of $\boldsymbol\sigma$ containing a key tuple equals $B_{1}\times \dots \times B_{n}$ for some
$B_{1},\ldots,B_{n}\subseteq A$.

\item For every block $\boldsymbol B = B_{1}\times \dots \times B_{n}$ containing a key tuple the intersection
$\boldsymbol\sigma\cap \boldsymbol B$
can be defined as follows.
There exist an abelian group $(G;+)$, whose order is a power of a prime number,
and bijective mappings
$\phi_i: B_{i}\to G$ for $i=1,2,\ldots,n$ such that
\[\boldsymbol\sigma\cap \boldsymbol B = \{(x_1,\ldots,x_n)\mid \phi_1(x_1)+\phi_2(x_2) + \ldots +\phi_n(x_n) = 0\}.\]

\end{enumerate}

\end{thm}
\begin{proof}
The first statement follows from Theorem~\ref{CoreBlockIsCubic}.
Let us prove the second statement.
Suppose $\boldsymbol B = B_{1}\times \dots \times B_{n}$ is a connected component of $\widetilde {\boldsymbol\sigma}$
containing a key tuple for $\boldsymbol\sigma$.
Let us show that $\boldsymbol\sigma\cap \boldsymbol B$ is a strongly rich relation.
Without loss of generality it is sufficient to prove that
for every $\alpha\in B_2\times \dots \times B_n$ there exists a unique $b\in B_1$
such that $b\alpha\in \boldsymbol\sigma$.
Since $\boldsymbol B\subseteq \widetilde{\boldsymbol\sigma}$,
there exists $b\in B_1$ such that $b\alpha\in \boldsymbol\sigma$.
It remains to prove that $b$ is unique. Assume that $b'\alpha\in \boldsymbol\sigma$
for $b'\in B_1\setminus\{b\}$.
By Corollary~\ref{corAboutTuplesFromBlock} every tuple from $\boldsymbol \sigma\cap \boldsymbol B$
belongs to a key block, therefore
two tuples $b\alpha,b'\alpha$ belong to the same key block.
This contradicts Theorem~\ref{MainForFullPattern}.

Thus we proved that $\boldsymbol\sigma\cap \boldsymbol B$ is a strongly rich relation.
By Lemma~\ref{DeriveRhoTilde} and Lemma~\ref{preserveConnectedComponent},
the connected component $\boldsymbol B$ is preserved by $\boldsymbol f$.
Then by Theorem~\ref{StronglyRichRelationTHM}
we have an abelian group $(G;+)$ and bijective mappings $\phi_i:B_{i}\to G$ for $i=1,2,\ldots,n$
such that
\[\boldsymbol\sigma\cap \boldsymbol B = \{(x_1,\ldots,x_n)\mid \phi_1(x_1)+\phi_2(x_2) + \ldots +\phi_n(x_n) = 0\}.\]
Suppose $\Psi$ is a vector-function preserving $\boldsymbol\sigma$ which maps a tuple from $\boldsymbol B$
to a tuple from $\boldsymbol B$. Then we can check that
$\Psi$ maps any tuple from $\boldsymbol B$ to a tuple from $\boldsymbol B$, i.e.
$\Psi$ preserves $\boldsymbol\sigma\cap \boldsymbol B$.
Therefore, $\boldsymbol\sigma\cap \boldsymbol B$ is a key relation on $\boldsymbol B$,
and by Lemma~\ref{OnlyPowerOfPrime} the order of the group $G$ is a power of a prime number.
This completes the proof.
\end{proof}

\subsection{Blocks of a key relation}

In this section we always assume that
$\boldsymbol\rho$ is a key relation of arity $n\ge 3$ with full pattern preserved by a WNU $\boldsymbol {f_0}$ of arity $m_0$,
$\boldsymbol\sigma$ is a core of $\boldsymbol\rho$.
Using Lemma~\ref{findbetterWNU},
we can find a WNU $\boldsymbol f$  of arity $m = m_0^{|A|!}$ preserving $\boldsymbol\rho$ and
satisfying the following property:
for every $\alpha\in A^{n}$ we have
$\boldsymbol h(\boldsymbol h(x)) = \boldsymbol h(x)$, where
$\boldsymbol h(x) = \boldsymbol f(\alpha,\ldots,\alpha,x)$.

\begin{lem}\label{oneBlockForEverything}

Suppose a tuple
$(b_{1},\ldots,b_{n})$ witnesses that
$(a_{1},\ldots,a_{n})$ is an essential tuple for $\boldsymbol\rho$,
$\Omega$ is a vector-function such that $\Omega(\boldsymbol\rho)\subseteq\boldsymbol\sigma$,
$\Omega(a_{1},\ldots,a_{n})$ belongs to a block $B_{1}\times \dots \times B_{n}$ of $\boldsymbol\sigma$.
Then $\Omega$ maps tuples from the set
$\boldsymbol f^{(1)}(\{a_1,b_1\}^{m})\times \boldsymbol f^{(2)}(\{a_2,b_2\}^{m})\times \dots \times \boldsymbol f^{(n)}(\{a_n,b_n\}^{m})$
to the block $B_{1}\times \dots \times B_{n}$.

\end{lem}

\begin{proof}

First, for every $\alpha \in \{a_1,b_1\}^{m}$ we show that
there exist $c_{2},\ldots,c_{n}\in A$ such that
$(\boldsymbol f^{(1)}(\alpha), c_{2},\ldots,c_{n})\in \boldsymbol\rho$,
and
the tuples
$(\boldsymbol f^{(1)}(\alpha), c_{2},\ldots,c_{n})$, $(a_{1},\ldots,a_{n})$ are connected in $\widetilde{\boldsymbol\rho}$.
Without loss of generality we can assume that $\alpha = b_{1}^{j} a_{1}^{m-j}$.
Put
\begin{align*}
\beta_{i} =& (\boldsymbol f^{(1)}(b_{1}^{i-1} a_{1}^{m-i+1}), \boldsymbol f^{(2)}(a_{2}^{i} b_{2}^{m-i}),a_{3},\ldots,a_{n}),\\
\delta_{i} =& (\boldsymbol f^{(1)}(b_{1}^{i} a_{1}^{m-i}), \boldsymbol f^{(2)}(a_{2}^{i} b_{2}^{m-i}),a_{3},\ldots,a_{n}).
\end{align*}
We can easily check that $\delta_{i}\in \boldsymbol\rho$.
If $\beta_{i}\notin \boldsymbol\rho$, then the tuple
\[(\boldsymbol f^{(1)}(b_{1}^{i} a_{1}^{m-i}), \boldsymbol f^{(2)}(a_{2}^{i-1} b_{2}^{m-i+1}),
\boldsymbol f^{(3)}(a_{3}^{m-1}b_{3}),\ldots,\boldsymbol f^{(n)}(a_{n}^{m-1}b_{n}))\]
witnesses that $\beta_{i}$ is an essential tuple.
Therefore we have a path 
$(a_{1},\ldots,a_{n}) - \delta_{0} - \beta_{1} - \delta_{1} - \beta_{2} - \delta_{2} -\dots - \beta_{j} - \delta_j$.
By the definition, $\delta_{j}(1) = \boldsymbol f^{(1)}(\alpha),$ which completes the first step.

Therefore, $\Omega^{(1)}(\boldsymbol f^{(1)}(\alpha))\in B_{1}$.
In the same way we prove that for every $i\in\{1,2,\ldots,n\}$
and $\alpha \in \{a_i,b_i\}^{m}$
we have $\Omega^{(i)}(\boldsymbol f^{(i)}(\alpha))\in B_{i}$.
This completes the proof.
\end{proof}

\begin{lem}\label{theOrderOfElementInGroupDividesM}
Suppose $(A_{1}\times \dots \times A_{n})$ is a block of $\boldsymbol\sigma$ defined by
\[(A_{1}\times \dots \times A_{n})\cap\boldsymbol\sigma = \{(x_{1},\ldots,x_{n})\mid \phi_{1}(x_{1})+\ldots+\phi_{n}(x_{n}) = 0\},\]
where $\phi_{i}$ is a bijective mapping from $A_{i}$ to $G$, and $(G;+)$ is an abelian group.
Then the order of every element in $G$ divides $m-1$.
\end{lem}
\begin{proof}
To simplify explanation we assume that $A_{1} = G$ and $\phi_{1}(x) = x$.
By Lemma~\ref{coreispreservedbyWNU},
we can find a WNU $\boldsymbol {f_0'}$ of arity $m_0$
preserving $\boldsymbol\sigma$.
Using Lemma~\ref{DeriveRhoTilde} and Lemma~\ref{preserveConnectedComponent}, we show that $\boldsymbol {f_0'}$
preserves the block
$A_1\times\dots\times A_n$.
Then, by Lemma~\ref{WNUIsLinearInAnyBlock} there exists an integer $t_0$ such that
\[\boldsymbol {f_0'}^{(1)}(x_{1},\ldots,x_{m_0}) = t_0\cdot x_{1}+ t_0\cdot x_2+ \ldots + t_0\cdot x_{m_0}.\]
Let $k$ be the order of an element in the group $G$.
Since $\boldsymbol {f_0'}$ is idempotent,
$m_0$ and $k$ are coprime.
Recall that $m = m_{0}^{|A|!}$.
By Fermat-Euler theorem we have $m_{0}^{\phi(k)}\equiv 1 (\mod k)$,
where $\phi$ is the Euler's totient function.
Hence $m-1 = m_{0}^{|A|!}-1$ is divisible by $k$. This completes the proof.
\end{proof}

\begin{lem}\label{MapsBlockToBlock}

Suppose $\alpha = (a_{1},\ldots,a_{n}),\beta = (a_1,\ldots,a_{j-1},b_{j},a_{j+1},\ldots,a_{n})$,
$\alpha,\beta\in\widetilde{\boldsymbol\rho}$, $\alpha\notin\boldsymbol\rho$.
A vector-function $\Omega$ maps $\alpha$ to
the block $A_{1}\times \dots \times A_{n}$ of $\boldsymbol\sigma$,
$\Omega(\boldsymbol\rho)\subseteq\boldsymbol\sigma$,
and
$(A_{1}\times \dots \times A_{n})
\cap\boldsymbol\sigma = \{(x_{1},\ldots,x_{n})\mid \phi_{1}(x_{1})+\ldots+\phi_{n}(x_{n}) = 0\},$
where $\phi_{i}$ is a bijective mapping from $A_{i}$ to $G$, and $(G;+)$ is an abelian group.
Then for every $(c_{1},c_{2},\ldots,c_{m})\in\{a_{j},b_{j}\}^{m}$ we have
\[\Omega^{(j)}(\boldsymbol f^{(j)}(c_{1},c_{2},\ldots,c_{m})) =
\phi_{j}^{-1}(\phi_{j}(\Omega^{(j)}(c_{1}))+\dots+\phi_{j}(\Omega^{(j)}(c_{m}))).\]

\end{lem}

\begin{proof}
To simplify explanation we assume that $A_{i} = G$ and $\phi_{i}(x) = x$ for every $i\in\{1,2,\ldots,n\}$.
Denote $\boldsymbol g(x_{1},\ldots,x_{m}) = \Omega(\boldsymbol f(x_{1},\ldots,x_{m}))$,
$\Delta = A_{1}\times \dots \times A_{n}$.
We consider two cases.

\textbf{Case 1}. Assume that $\beta\in\boldsymbol\rho$.
Without loss of generality assume that $j=1$.
It is sufficient to prove that for
every $i\in \{0,1,\ldots,m\}$
we have
\[\Omega^{(1)}(\boldsymbol f^{(1)}(\underbrace{b_{1},\ldots,b_{1}}_i,a_{1},\ldots,a_{1})) =
i\cdot \Omega^{(1)}(b_{1}) +(m-i)\cdot \Omega^{(1)}(a_{1}).\]
Let the tuple $(b_1,\ldots,b_n)$ witness that
$\alpha$ is an essential tuple.
Let
\begin{align*}
\alpha_{i} &= (\boldsymbol f^{(1)}(b_{1}^{i-1}a_{1}^{m-i+1}),
\boldsymbol f^{(2)}(a_{2}^{i}b_{2}^{m-i}),\boldsymbol f^{(3)}(a_{3}^{m-1}b_{3}),a_{4},\ldots,a_n),\\
\beta_{i} &= (\boldsymbol f^{(1)}(b_{1}^{i}a_{1}^{m-i}),\boldsymbol f^{(2)}(a_{2}^{i}b_{2}^{m-i}),a_3,\ldots,a_n).
\end{align*}
By Lemma~\ref{oneBlockForEverything}, $\Omega$ maps tuples
$\alpha_{i}$ and $\beta_{i}$ to $\Delta$. Since $\alpha_{i},\beta_{i}\in\boldsymbol\rho$, we have
\[\boldsymbol g^{(1)}(b_1^{i-1} a_{1}^{m-i+1})+
\boldsymbol g^{(2)}(a_{2}^{i}b_{2}^{m-i})+
\boldsymbol g^{(3)}(a_{3}^{m-1}b_{3})+
\boldsymbol g^{(4)}(a_{4}^{m})+\ldots +\boldsymbol g^{(n)}(a_{n}^{m})=0,\]
\[\boldsymbol g^{(1)}(b_1^{i} a_{1}^{m-i})+
\boldsymbol g^{(2)}(a_{2}^{i}b_{2}^{m-i})+
\boldsymbol g^{(3)}(a_{3}^{m})+
\boldsymbol g^{(4)}(a_{4}^{m})+\ldots +\boldsymbol g^{(n)}(a_{n}^{m})=0.\]
Subtracting these equations we obtain the following for every $i$
\[\boldsymbol g^{(1)}(b_1^{i} a_{1}^{m-i})-\boldsymbol g^{(1)}(b_1^{i-1} a_{1}^{m-i+1}) =
\boldsymbol g^{(3)}(a_{3}^{m-1}b_{3})-\boldsymbol g^{(3)}(a_{3}^{m}).\]
Therefore
\begin{equation}\label{eq1}
\boldsymbol g^{(1)}(b_1^{i} a_{1}^{m-i}) = \Omega^{(1)}(a_{1}) + i\cdot (\boldsymbol g^{(3)}(a_{3}^{m-1}b_{3})-\Omega^{(3)}(a_3)).
\end{equation}

Obviously
$\Omega\left(\begin{smallmatrix}
\boldsymbol f^{(1)}(a_{1}^{m-1}b_{1}^{1})\\
\boldsymbol f^{(2)}(b_{2}^{m-1}a_{2}^{1})\\
a_3\\
a_4\\
\\
\dots\\
a_n
\end{smallmatrix}\right),
\Omega\left(\begin{smallmatrix}
a_1\\
\boldsymbol f^{(2)}(b_{2}^{m-1}a_{2}^{1})\\
\boldsymbol f^{(3)}(a_{3}^{m-1}b_{3}^{1})\\
a_4\\
\\
\dots\\
a_n
\end{smallmatrix}\right)\in \boldsymbol\sigma
$ and, by Lemma~\ref{oneBlockForEverything}, these tuples belong to the block $\Delta$.
Therefore we can subtract the equations corresponding to these tuples to get
\[\boldsymbol g^{(1)}(a_{1}^{m-1}b_{1})-\Omega^{(1)}(a_{1}) =
\boldsymbol g^{(3)}(a_{3}^{m-1}b_{3})-\Omega^{(3)}(a_{3}).\]
Combining this with the equation~(\ref{eq1})
we get the following equation 
\begin{equation}\label{eq11}
\boldsymbol g^{(1)}(b_1^{i} a_{1}^{m-i}) = \Omega^{(1)}(a_{1}) + i\cdot (g^{(1)}(a_{1}^{m-1}b_{1})-\Omega^{(1)}(a_1)).
\end{equation}
Put $i= m$.
By Lemma~\ref{theOrderOfElementInGroupDividesM}
the order of any element in $G$ divides $m-1$, hence
\[\Omega^{(1)}(b_{1}) = \Omega^{(1)}(a_{1})+ m\cdot (\boldsymbol g^{(1)}(a_{1}^{m-1}b_{1})-\Omega^{(1)}(a_1))=
\boldsymbol g^{(1)}(a_{1}^{m-1}b_{1}).\]
By Equation~(\ref{eq11}) we get
$\boldsymbol g^{(1)}(b_1^{i} a_{1}^{m-i}) = i\cdot \Omega^{(1)}(b_1) + (m-i)\cdot\Omega^{(1)}(a_1),$
which completes the first case.

\textbf{Case 2}. Assume that $\alpha$ and $\beta$ are essential tuples for $\boldsymbol\rho$.
Without loss of generality, assume that $j=1$.
Since $\alpha$ and $\beta$ are essential,
we can find $b_{2}$ and $b_{3}$ such that
\[(a_{1},b_{2},a_{3},\ldots,a_{n}),(b_{1},a_{2},b_{3},a_{4},\ldots,a_{n})\in \boldsymbol\rho.\]

Obviously, it is sufficient to prove that for
every $i\in \{0,1,\ldots,n\}$
we have
\[\Omega^{(1)}(\boldsymbol f^{(1)}(\underbrace{a_{1},\ldots,a_{1}}_i,b_{1},\ldots,b_{1})) =
i\cdot \Omega^{(1)}(a_{1}) +(m-i)\cdot \Omega^{(1)}(b_{1}).\]
Let
$\alpha_{i} = (\boldsymbol f^{(1)}(a_{1}^{i}b_{1}^{m-i}),\boldsymbol f^{(2)}(b_{2}^{i}a_{2}^{m-i}),\boldsymbol f^{(3)}(a_{3}^{i}b_{3}^{m-i}),a_{4},\ldots,a_n).$
We know from Lemma~\ref{oneBlockForEverything} that $\Omega$ maps tuples
$\alpha_{i}$ to $\Delta$, and $\alpha_{i}\in \boldsymbol\rho$,
therefore $\Omega(\alpha_{i})\in \boldsymbol\sigma$, which means that
\begin{equation}\label{eq2}
\boldsymbol g^{(1)}(a_{1}^{i}b_{1}^{m-i})+
\boldsymbol g^{(2)}(b_{2}^{i}a_{2}^{m-i})+
\boldsymbol g^{(3)}(a_{3}^{i}b_{3}^{m-i})+
\boldsymbol g^{(4)}(a_{4}^{m})+\ldots +\boldsymbol g^{(n)}(a_{n}^{m})=0.
\end{equation}
Since we already proved case 1, we
have
\begin{align*}
\boldsymbol g^{(2)}(b_{2}^{i}a_{2}^{m-i}) &= i\cdot \Omega^{(2)}(b_{2}) + (m-i)\cdot \Omega^{(2)}(a_{2}),\\
\boldsymbol g^{(3)}(a_{3}^{i}b_{3}^{m-i}) &= i\cdot \Omega^{(3)}(a_{3}) + (m-i)\cdot \Omega^{(3)}(b_{3}).
\end{align*}
Subtracting the equation~(\ref{eq2}) for $i$ and $i-1$, and using the above equations we get
\[\boldsymbol g^{(1)}(a_{1}^{i}b_{1}^{m-i}) - \boldsymbol g^{(1)}(a_{1}^{i-1}b_{1}^{m-i+1}) =
\Omega^{(2)}(a_{2})- \Omega^{(2)}(b_{2})+\Omega^{(3)}(b_{3}) -\Omega^{(3)}(a_{3}).\]
Therefore, for every $i$ we have
\[\boldsymbol g^{(1)}(a_{1}^{i}b_{1}^{m-i}) = \Omega^{(1)}(b_{1}) + i \cdot
(\Omega^{(2)}(a_{2})- \Omega^{(2)}(b_{2})+\Omega^{(3)}(b_{3}) -\Omega^{(3)}(a_{3})).\]
Since $(a_1,b_2,a_3,a_4,\ldots,a_n)$,
$(b_1,a_2,b_3,a_4,\ldots,a_n)\in\boldsymbol\rho$, we have
 \[\Omega^{(1)}(a_1)+\Omega^{(2)}(b_{2}) + \Omega^{(3)}(a_{3}) =
\Omega^{(1)}(b_1)+\Omega^{(2)}(a_{2}) + \Omega^{(3)}(b_{3}) 
.\]
By Lemma~\ref{theOrderOfElementInGroupDividesM}
the order of any element in $G$ divides $m-1$, hence
\[\boldsymbol g^{(1)}(a_1^{i} b_{1}^{m-i}) = \Omega^{(1)}(b_{1}) + i\cdot (\Omega^{(1)}(a_{1}) - \Omega^{(1)}(b_{1}))=
i\cdot \Omega^{(1)}(a_{1})+(m-i)\cdot \Omega^{(1)}(b_{1}).\qedhere\]
\end{proof}

\begin{cor}\label{aaaabCor}

Suppose $\alpha = (a_{1},\ldots,a_{n}),\beta = (a_1,\ldots,a_{j-1},b_{j},a_{j+1},\ldots,a_{n})$,
$\alpha,\beta\in\widetilde{\boldsymbol\rho}$, $\alpha\notin\boldsymbol\rho$ or $\beta\notin\boldsymbol\rho$,
a vector-function $\Omega$ maps $\alpha$, $\beta$ to a nontrivial block of $\boldsymbol\sigma$ containing a key tuple and
$\Omega(\boldsymbol\rho)\subseteq\boldsymbol\sigma$.
Then
$\Omega(\boldsymbol f^{(j)}(a_j,\ldots,a_j,b_j)) = \Omega(b_{j})$.

\end{cor}

\begin{proof}

Suppose
$\Omega$ maps $\alpha$, $\beta$ to the block $A_{1}\times \dots \times A_{n}$ of $\boldsymbol\sigma$,
and
\[(A_{1}\times \dots \times A_{n})
\cap\boldsymbol\sigma = \{(x_{1},\ldots,x_{n})\mid \phi_{1}(x_{1})+\ldots+\phi_{n}(x_{n}) = 0\},\]
where $\phi_{i}$ is a bijective mapping from $A_{i}$ to $G$, and $(G;+)$ is an abelian group.
Then by Lemma~\ref{MapsBlockToBlock} we have
\[\Omega^{(j)}(\boldsymbol f^{(j)}(a_j,\ldots,a_j,b_j)) =
\phi_{j}^{-1}(\phi_{j}(\Omega^{(j)}(a_{j}))+\dots+\phi_{j}(\Omega^{(j)}(a_{j}))+\phi_{j}(\Omega^{(j)}(b_{j}))).\]
By Lemma~\ref{theOrderOfElementInGroupDividesM}
the order of any element in $G$ divides $m-1$, hence
\[\Omega^{(j)}(\boldsymbol f^{(j)}(a_j,\ldots,a_j,b_j)) =
\phi_{j}^{-1}(\phi_{j}(\Omega^{(j)}(b_{j}))) = \Omega^{(j)}(b_{j}).\qedhere\]
\end{proof}

\begin{lem}

Suppose
$
\left(\begin{smallmatrix}
a_{1} \\a_{2} \\a_{3}\\ \dots \\a_{n}
\end{smallmatrix}\right),
\left(\begin{smallmatrix}
b_{1} \\a_{2} \\a_{3}\\ \dots \\a_{n}
\end{smallmatrix}\right),
\left(\begin{smallmatrix}
a_{1} \\b_{2} \\a_{3}\\ \dots \\a_{n}
\end{smallmatrix}\right)\in \widetilde{\boldsymbol\rho}$.
Then
$\left(\begin{smallmatrix}
b_{1} \\b_{2} \\a_{3}\\ \dots \\a_{n}
\end{smallmatrix}\right)\in \widetilde{\boldsymbol\rho}$

\end{lem}

\begin{proof}

It follows from the definition
that there exist $b_{3}, c_{3},d_{3}\in A$ such that
$
\left(\begin{smallmatrix}
a_{1} \\a_{2} \\b_{3}\\ a_{4}\\\dots \\a_{n}
\end{smallmatrix}\right),
\left(\begin{smallmatrix}
b_{1} \\a_{2} \\c_{3}\\ a_{4}\\\dots \\a_{n}
\end{smallmatrix}\right),
\left(\begin{smallmatrix}
a_{1} \\b_{2} \\d_{3}\\ a_{4}\\\dots \\a_{n}
\end{smallmatrix}\right)
\in\boldsymbol\rho$.
Denote
$\alpha = \left(\begin{smallmatrix}
a_{1} \\a_{2} \\a_{3}\\ \dots \\a_{n}
\end{smallmatrix}\right)$,
$\beta =\left(\begin{smallmatrix}
b_{1} \\a_{2} \\a_{3}\\ \dots \\a_{n}
\end{smallmatrix}\right)$,
$\gamma =
\left(\begin{smallmatrix}
a_{1} \\b_{2} \\a_{3}\\ \dots \\a_{n}
\end{smallmatrix}\right)$,
$\delta =
\left(\begin{smallmatrix}
b_{1} \\b_{2} \\a_{3}\\ \dots \\a_{n}
\end{smallmatrix}\right)$,
$
\alpha' = \left(\begin{smallmatrix}
a_{1} \\a_{2} \\b_{3}\\ a_{4}\\\dots \\a_{n}
\end{smallmatrix}\right),
\beta' = \left(\begin{smallmatrix}
b_{1} \\a_{2} \\c_{3}\\ a_{4}\\\dots \\a_{n}
\end{smallmatrix}\right),
\gamma' = \left(\begin{smallmatrix}
a_{1} \\b_{2} \\d_{3}\\ a_{4}\\\dots \\a_{n}
\end{smallmatrix}\right)$.

If $\delta\in \boldsymbol\rho$, then we are done.
Assume that $\delta\notin \boldsymbol\rho$, let us prove that
$\delta$ is essential for $\boldsymbol\rho$.

Let
$\zeta =
(b_{1},b_{2},\boldsymbol f^{(3)}(c_{3},b_{3},\ldots,b_{3},d_{3}),a_{4},\ldots,a_{n})$.
Assume that $\zeta\notin \boldsymbol\rho$.
Since $\boldsymbol\rho$ is a key relation, we can map $\zeta$ to a key tuple of $\boldsymbol\sigma$
,which by Lemma~\ref{CoreProperties} is a key tuple of $\boldsymbol\rho$, and then we can use a restricting vector-function for
the core $\boldsymbol\sigma$.
Thus, we get a vector-function $\Psi$ that maps $\zeta$ to a key tuple of $\boldsymbol\sigma$
such that $\Psi(\boldsymbol\rho)\subseteq \boldsymbol\sigma$.

Since $\beta$ and $\gamma$ are essential,
$(a_1',b_2,a_3,\ldots,a_{n}),(b_1,a_2',a_3,\ldots,a_{n})\in \boldsymbol\rho$ for some $a_{1}',a_{2}'\in A$.
Put $\beta'' = (a_1',b_2,a_3,\ldots,a_{n})$ and
$\gamma'' = (b_1,a_2',a_3,\ldots,a_{n})$.
Since $\Psi(\zeta)$ is key tuple, there exists a tuple $\zeta'\in\boldsymbol\sigma$ that differs from
$\Psi(\zeta)$ just in the third component.
Obviously, $\Psi(\delta)$ can differ from $\zeta'$ only in the third component,
can differ from $\Psi(\beta'')$ only in the first component,
can differ from $\Psi(\gamma'')$ only in the second component.
Thus $\Psi(\delta)$ is either a weak essential tuple for $\boldsymbol\sigma$, or $\Psi(\delta)\in\boldsymbol\sigma$.
If $\Psi(\delta)\notin\sigma$, we apply Lemma~\ref{CoreConnectedMain} to tuples $\zeta'-\Psi(\delta)-\Psi(\beta'')$
and show that $\Psi(\beta'')$ belongs to a key block.
Thus, $\Psi(\alpha),\Psi(\beta),\Psi(\gamma)$ belongs to a block of $\boldsymbol\sigma$ containing a key tuple.
By Theorem~\ref{CoreBlockIsCubic} every block of $\boldsymbol\sigma$ containing a key tuple
can be represented as $A_1\times\dots\times A_n$.
Hence,
$\Psi(\alpha),\Psi(\beta),\Psi(\gamma),\Psi(\delta), \Psi(\zeta)$ belong
to one nontrivial block of $\boldsymbol\sigma$ containing a key tuple.

Let us prove that
$\Psi^{(1)}(b_{1}) = \Psi^{(1)}(\boldsymbol f^{(1)}(b_{1},a_{1},\ldots,a_{1}))$.
If $\alpha\notin \boldsymbol\rho$ or $\beta\notin \boldsymbol\rho$
then it follows from Corollary~\ref{aaaabCor}.
Assume that $\alpha,\beta\in \boldsymbol\rho$.
Then
$\Psi(\beta)$ and $\Psi(\boldsymbol f(\beta,\alpha,\ldots,\alpha))$
belong to the nontrivial block of $\boldsymbol\sigma$ we mentioned before.
Hence, they belong to a nontrivial block containing a key tuple but differ at most in one coordinate.
By Theorem~\ref{CoreBlockIsLinear}, they are equal, and
$\Psi^{(1)}(b_{1}) = \Psi^{(1)}(\boldsymbol f^{(1)}(b_{1},a_{1},\ldots,a_{1}))$.

In the same way we prove that
$\Psi^{(2)}(b_{2}) = \Psi^{(2)}(\boldsymbol f^{(2)}(b_{2},a_{2},\ldots,a_{2}))$.

Therefore,
$\Psi(\zeta) = \Psi(\boldsymbol f(\beta',\alpha',\ldots,\alpha',\gamma'))$,
which contradicts the fact that
$\Psi(\zeta)\notin \boldsymbol\rho$.

Thus, we proved that $\zeta\in \boldsymbol\rho$, which means that
the third coordinate of $\delta$ can be changed to get a tuple from the relation $\boldsymbol\rho$.
In the same way we prove that we can change any other coordinate of $\delta$,
which means that $\delta$ is an essential tuple. This completes the proof.
\end{proof}

The following theorem follows directly from the previous lemma.

\begin{thm}\label{blockIsCubic}

Suppose $\boldsymbol\rho$ is a key relation of arity greater than 2 with full pattern preserved by a WNU.
Then every connected component of $\boldsymbol{\widetilde\rho}$ can be represented as
$A_{1}\times\dots\times A_{n}$ for some $A_1,\ldots,A_n\subseteq A$.

\end{thm}

Now we are ready to prove the main theorem for key relations with full pattern.
\begin{THMmainTheoremForFullPattern}

Suppose $\boldsymbol\rho$ is a key relation of arity greater than 2 preserved by a WNU $\boldsymbol f$, the pattern of $\boldsymbol\rho$ is
a full equivalence relation. Then
\begin{enumerate}
\item Every block of $\boldsymbol\rho$ equals $B_{1}\times \dots \times B_{n}$ for some
$B_{1},\ldots,B_{n}\subseteq A$.

\item For every nontrivial block $\boldsymbol B = B_{1}\times \dots \times B_{n}$ the intersection
$\boldsymbol\rho\cap\boldsymbol B$
can be defined as follows.
There exists an abelian group $(G;+)$, whose order is a power of a prime number,
and surjective mappings
$\phi_i: B_{i}\to G$ for $i=1,2,\ldots,n$ such that
\[\boldsymbol\rho\cap \boldsymbol B = \{(x_1,\ldots,x_n)\mid \phi_1(x_1)+\phi_2(x_2) + \ldots +\phi_n(x_n) = 0\}.\]

\end{enumerate}

\end{THMmainTheoremForFullPattern}

\begin{proof}
The first statement follows from Theorem~\ref{blockIsCubic}.
Let us prove the second statement.
Suppose $\boldsymbol B = B_{1}\times \dots \times B_{n}$ is a nontrivial block of $\boldsymbol\rho$.
By Lemma~\ref{DeriveRhoTilde} and Lemma~\ref{preserveConnectedComponent},
the connected component $\boldsymbol B$ is preserved by $\boldsymbol f$.

Let $\boldsymbol\sigma$ be a core of $\boldsymbol\rho$.
Suppose $\boldsymbol C = C_1\times \dots\times C_n$ is a nontrivial block of $\boldsymbol\sigma$ containing a key tuple for $\boldsymbol\sigma$.
By Theorem~\ref{CoreBlockIsLinear}
\[\boldsymbol \sigma\cap\boldsymbol C = \{(x_1,\ldots,x_n)\mid \phi_1(x_1)+\phi_2(x_2) + \ldots +\phi_n(x_n) = 0\}\]
for an abelian group $(G;+)$ and bijective mappings
$\phi_i: C_{i}\to G$.
Put
\begin{multline*}
\sigma_i(x,y) = \exists z_1 \dots\exists z_{i-1}\exists z_{i+1}\dots\exists z_{n}\;
\boldsymbol B(z_{1},\ldots,z_{i-1},x,z_{i+1},\ldots,z_{n})\wedge \\
\boldsymbol\rho(z_{1},\ldots,z_{i-1},x,z_{i+1},\ldots,z_{n})
\wedge \boldsymbol\rho(z_{1},\ldots,z_{i-1},y,z_{i+1},\ldots,z_{n}).
\end{multline*}
Let us show that
$\sigma_i(x,y)$ is an equivalence relation on $B_{i}$ for every $i\in\{1,2,\ldots,n\}$.
Without loss of generality we prove for $i=1$.
It is sufficient to prove that for every $\alpha,\beta \in A^{n-1}, a,b\in A$
such that $a \alpha,b\alpha\in\boldsymbol\rho\cap\boldsymbol B$
we have $a\beta\in \boldsymbol\rho\cap\boldsymbol B\Leftrightarrow b\beta\in \boldsymbol\rho\cap\boldsymbol B$.
Assume the converse. Let $a\beta\notin \boldsymbol\rho\cap\boldsymbol B$ and $b\beta\in \boldsymbol\rho\cap\boldsymbol B$.
Since $\boldsymbol\rho$ is a key relation and, by Lemma~\ref{CoreProperties}, any key tuple
of $\boldsymbol\sigma$ is also a key tuple of $\boldsymbol\rho$,
we can find a vector-function $\Psi$ that maps $a\beta$ to a key tuple of $\boldsymbol C$
such that $\Psi(\boldsymbol\rho)\subseteq\boldsymbol\sigma$.
We can easily check that $\Psi$ maps $a\beta$ and every tuple from $\boldsymbol B$
to the connected component $\boldsymbol C$.
Since $a\alpha,b\alpha\in\boldsymbol\rho$ we obtain
$\Psi(a\alpha),\Psi(b\alpha)\in\boldsymbol\sigma$ and
$\Psi^{(1)}(a) = \Psi^{(1)}(b)$. This contradicts the fact that
$\Psi(a\beta)\notin \boldsymbol \sigma$ and $\Psi(b\beta)\in \boldsymbol \sigma$.

Thus, $\sigma_i$ is an equivalence relation preserved by the WNU $\boldsymbol f^{(i)}$.
Then we have an algebra $\mathbb B_{i} = (B_{i};\boldsymbol f^{(i)})$
and $\sigma_i$ is a congruence for the algebra.
Then we consider the factor-algebra $\mathbb B_{i}/\sigma_{i}$ which we denote
as $(D_{i};g_i)$ and a natural homomorphism $\varphi_i$
that maps $B_i$ to $D_{i}$.
Put $\boldsymbol g = (g_1,\ldots,g_n)$.
Let us define a relation $\boldsymbol\rho'\subseteq D_1\times\dots \times D_n$.
Put $(\varphi_{1}(x_{1}),\ldots,\varphi_{n}(x_n))\in\boldsymbol \rho'\Leftrightarrow (x_{1},\ldots,x_{n})\in \boldsymbol\rho.$
As we proved before,
for every $\alpha,\beta \in A^{n-1}, a,b\in A$
such that $a \alpha,b\alpha\in\boldsymbol\rho\cap\boldsymbol B$
we have $a\beta\in \boldsymbol\rho\cap\boldsymbol B\Leftrightarrow b\beta\in \boldsymbol\rho\cap\boldsymbol B$.
Hence, the relation $\boldsymbol\rho'$ is well-defined.
We can check that $\boldsymbol\rho'$ is a strongly rich relation that is preserved by the WNU $\boldsymbol g$.
Then by Theorem~\ref{StronglyRichRelationTHM}
we have an abelian group $(G';+)$ and bijective mappings $\psi_i:D_{i}\to G'$ for $i=1,2,\ldots,n$
such that
\[\boldsymbol\rho' = \{(x_1,\ldots,x_n)\mid \psi_1(x_1)+\psi_2(x_2) + \ldots +\psi_n(x_n) = 0\}.\]
Then the original relation can be defined as follows
\[\boldsymbol\rho\cap \boldsymbol B = \{(x_1,\ldots,x_n)\mid \psi_1(\varphi_1(x_1))+\psi_2(\varphi_2(x_2)) + \ldots +\psi_n(\varphi_n(x_n)) = 0\}.\]

It remains to show that
the order of the group $G'$ is a power of a prime number.
We know from Theorem~\ref{CoreBlockIsLinear} that the order of the group $G$ is a power of a prime number $p$.
Assume that the order of the group $G'$ is not a power of a prime number.
Choose an element $c_1\in G'$ whose order is a prime number that differs from $p$,
and an element $c_2\in G$ of order $p$.
Denote one of the tuples corresponding to $(c_1,0,\ldots,0)$ in $\boldsymbol B$ by $\alpha_1$,
and the tuple corresponding to $(c_2,0,\ldots,0)$ in $\boldsymbol C$ by $\alpha_2$.
Since $\alpha_2$ is a key tuple for $\boldsymbol\rho$, there exists a
unary vector-function  $\Omega$ that preserves $\boldsymbol\rho$ and maps $\alpha_1$ to $\alpha_2$.
Since $\boldsymbol\sigma$ is a core, then there exists a restricting vector-function $\Omega'$.
Obviously, $\Omega'\circ\Omega$ maps
$\boldsymbol\rho$ to $\boldsymbol\sigma$
and maps the block $\boldsymbol B$ to the block $\boldsymbol C$.
Then, by Lemma~\ref{NoMappingBetweenDifferentGroups},
the order of $c_2$ in $G$ divides the order of $c_1$ in $G'$.
This contradiction completes the proof.
\end{proof}

\bibliographystyle{au_main}
\bibliography{refs}

\end{document}